\documentclass[12pt,leqno]{amsart}

\usepackage{amsmath,amscd,amsthm,amsxtra,amssymb}
\usepackage{epsfig,graphics,color,colortbl}
\usepackage{amssymb,latexsym}
\usepackage{mathrsfs}
\usepackage[poly,all]{xy}
\usepackage{marginnote}
\usepackage{xspace}
\usepackage[colorlinks=true, pdfstartview=FitV, linkcolor=blue,citecolor=blue,urlcolor=blue]{hyperref}

\usepackage{yfonts}
\usepackage{enumerate}

\usepackage[usenames,dvipsnames,svgnames,table]{xcolor}
\usepackage[normalem]{ulem}  
\usepackage{float}
\usepackage[colorinlistoftodos]{todonotes}
\newcommand{\arxiv}[1]{\href{http://arxiv.org/abs/#1}{\texttt{arXiv:#1}}}

\newdir{ >}{{}*!/-10pt/@{>}}

\allowdisplaybreaks[3]

\newlength{\mylength}
\renewcommand{\le}{\leqslant}
\renewcommand{\ge}{\geqslant}
\renewcommand{\emptyset}{\varnothing}

\theoremstyle{plain}
\newtheorem{thm}{Theorem}[section]
\newtheorem{mthm}[thm]{Main\,Theorem}

\newtheorem{prop}[thm]{Proposition}
\newtheorem{coro}[thm]{Corollary}
\newtheorem{lem}[thm]{Lemma}
\newtheorem{conj}{Conjecture}
\newtheorem{mainconj}[conj]{Main Conjecture}

\theoremstyle{definition}
\newtheorem{example}[thm]{Example}
\newtheorem{remark}[thm]{Remark}
\newtheorem{definition}[thm]{Definition}

\newtheorem{question}[conj]{Qusetion}
\newtheorem{problem}[conj]{Problem}

\newcommand{\nc}{\newcommand}

\newenvironment{answer}
{\noindent{\bf Answer}\hs{1ex}}
{\hfill \qedsymbol}
\nc{\Prop}{\begin{prop}}
\nc{\enprop}{\end{prop}}
\nc{\Lemma}{\begin{lem}}
\nc{\enlemma}{\end{lem}}
\nc{\Ex}{\begin{example}}
\nc{\enex}{\end{example}}
\nc{\Th}{\begin{thm}}
\nc{\enth}{\end{thm}}
\nc{\Def}{\begin{definition}}
\nc{\edf}{\end{definition}}
\nc{\Conj}{\begin{conj}}
\nc{\enconj}{\end{conj}}
\nc{\mConj}{\begin{mainconj}}
\nc{\enmconj}{\end{mainconj}}
\nc{\Quest}{\begin{question}}
\nc{\enquest}{\end{question}}
\nc{\Rem}{\begin{remark}}
\nc{\enrem}{\end{remark}}
\nc{\Ans}{\begin{answer}}
\nc{\enans}{\end{answer}}
\nc{\Prob}{\begin{problem}}
\nc{\enprob}{\end{problem}}
\nc{\MTh}{\begin{mthm}}
\nc{\enmth}{\end{mthm}}

\newenvironment{red}
{\relax\color{red}}
{\hspace*{.5ex}\relax}

\nc{\berm}{\ber{}\marginnote{\fbox{\scshape\lowercase{M}}}}
\nc{\berMH}{\ber{}\marginnote{\fbox{\scshape\lowercase{MH}}}}
\nc{\berE}{\ber{}\marginnote{\fbox{\scshape\lowercase{E}}}}
\nc{\bero}{\ber{}\marginnote{\fbox{\scshape\lowercase{O}}}}
\nc{\bebE}{\beb{}\marginnote{\fbox{\scshape\lowercase{E}}}}

\newenvironment{blue}
{\relax\color{Blue}}
{\hspace*{.5ex}\relax}

\newcommand{\beb}{\begin{blue}}
\newcommand{\eb}{\end{blue}}

\newenvironment{yellow}
{\relax\color{Dandelion}}
{\hspace*{.5ex}\relax}

\newcommand{\bey}{\begin{yellow}}
\newcommand{\ey}{\end{yellow}}

\nc{\on}{\operatorname}

\newcommand{\C}{{\mathbb C}}
\newcommand{\Q}{\mathbb {Q}}
\newcommand{\Z}{\ms{2mu}{\mathbb Z}}
\newcommand{\R}{{\mathbb R}}

\newcommand{\D}{\mathscr{D}\ms{1mu}}

\newcommand{\one}{{\ms{1mu}\bf{1}}}
\newcommand{\seteq}{\mathbin{:=}}

\newcommand{\hd}{{\mathrm{hd}}}      					 
\newcommand{\soc}{\mathrm{soc}}					
\newcommand{\To}[1][{\hs{0.8ex}}]{\xrightarrow{\ms{7mu}{#1}\ms{7mu}}}

\newcommand{\g}{{\ms{1mu}\mathfrak{g}\ms{1mu}}}
\newcommand{\n}{\mathfrak{n}}

\newcommand{\Hom}{\operatorname{Hom}}
\newcommand{\HOM}{\on{\mathrm{H{\scriptstyle OM}}}}

\newcommand{\isoto}[1][]{\mathop{\xrightarrow%
[{\raisebox{.3ex}[0ex][.3ex]{$\scriptstyle{#1}$}}]%
{{\raisebox{-.6ex}[0ex][-.6ex]{$\mspace{5mu}\sim\mspace{5mu}$}}}}}

\newcommand{\Mod}{\text{-$\mathrm{Mod}$}}
\newcommand{\gmod}{\text{-}\mathrm{gmod}}
\newcommand{\gMod}{\text{-}\mathrm{gMod}}

\newcommand{\F}{\mathscr{F}}

\def\T{{\mathcal T}}

\newcommand{\conv}[1][]{
\underset{\raisebox{.5ex}{$\scriptstyle{#1}$}}{\mathbin{\scalebox{1.2}{$\mspace{1.5mu}\circ\mspace{1.5mu}$}}}}
\newcommand{\hconv}{\mathbin{\scalebox{.9}{$\nabla$}}}

\newcommand{\sconv}{\mathbin{\scalebox{.9}{$\Delta$}}}

\newcommand{\de}{\on{\textfrak{d}}}

\newcommand{\tB}{\widetilde{B}}
\newcommand{\tb}{\widetilde{b}}
\newcommand{\seed}{\mathscr{S}}
\newcommand{\cmC}{\cartan}  
\newcommand{\wlP}{\mathsf{P}}   
\newcommand{\rlQ}{\mathsf{Q}}   
\newcommand{\weyl}{\mathsf{W}}  
\newcommand{\prD}{\Delta_+}            
\newcommand{\nrD}{\Delta_-}            
\newcommand{\sg}{\mathfrak{S}}   

\nc{\prt}{\prD}
\nc{\qQ}{Q}

\newcommand{\Zq}{{\Z[q,q^{-1}]}}  		

\newcommand{\wt}{\mathrm{wt}} 		
\newcommand{\bR}{\mathbf{k}} 		
\nc{\corp}{\bR}
\newcommand{\catC}{ \mathscr{C}}  	

\newcommand{\dM}{ \mathrm{M }}              
\newcommand{\dC}{ \mathsf{C }}              
\newcommand{\dS}{ \mathrm{S }}              
\newcommand{\gW}{\mathrm{W}}
\newcommand{\sgW}{\mathrm{W}^*}
\newcommand{\tf}{{\widetilde{f}}}  		
\newcommand{\te}{{\widetilde{e}}}  		
\newcommand{\tF}{\widetilde{\mathrm{F}}} 		
\newcommand{\tE}{\widetilde{\mathrm{E}}} 		
\newcommand{\tEs}{\widetilde{\mathrm{E}}^*} 		
\newcommand{\tFs}{\widetilde{\mathrm{F}}^*} 		
\newcommand{\tEm}{\widetilde{\mathrm{E}}^{\hskip 0.1em \rm max}}  		
\newcommand{\tEsm}{\widetilde{\mathrm{E}}^{*{ \hskip 0.1em  \rm max}}}  		
\newcommand{\ep}{\varepsilon}  		
\newcommand{\ph}{\varphi}  		

\renewcommand{\preceq}{\preccurlyeq}

\newcommand{\La}{\Lambda} 			
\newcommand{\tLa}{\widetilde{\Lambda}} 			

\nc{\Ma}{{\ms{1.5mu}\mathsf{M}}}
\nc{\Na}{\mathsf{N}}
\nc{\Xa}{\mathsf{X}}
\nc{\Ya}{\mathsf{Y}}
\nc{\Laa}{\mathsf{L}}

\newcommand{\zM}{{z_\Ma}}

\newcommand{\id}{\ms{2mu}{\mathsf{id}}\ms{1mu}}   				

\newcommand{\col}{\colon}
\nc{\be}{\begin{enumerate}}
\newcommand{\bnum}{\be[{\rm(i)}]}
\newcommand{\bna}{\be[{\rm(a)}]}

\newcommand{\rtl}{\rlQ}

\newcommand{\rmat}[1]{\ms{1mu}{\mathbf{r}}_%
{\mspace{-2mu}\raisebox{-.6ex}{${\scriptstyle{#1}}$}}}

\newcommand{\shc}{{\ms{2mu}\mathcal{C}}}

\nc{\ms}{\mspace}
\nc{\cl}{\colon}
\nc{\ro}{{\rm (}}
\nc{\rf}{{\rm )}\xspace}
\nc{\noi}{\noindent}
\nc{\bl}{\bigl(}
\nc{\br}{\bigr)}

\newenvironment{myequationn}
{\relax\setlength{\arraycolsep}{1pt}\begin{eqnarray*}}
{\end{eqnarray*}}

\newenvironment{myequation}
{\relax\setlength{\arraycolsep}{1pt}\begin{eqnarray}}
{\end{eqnarray}}

\nc{\eq}{\begin{myequation}}
\nc{\eneq}{\end{myequation}}
\nc{\eqn}{\begin{myequationn}}
\nc{\eneqn}{\end{myequationn}}

\newenvironment{myarray}[1]{\relax\setlength{\arraycolsep}{1pt}
\begin{array}{#1}}{\end{array}\relax}

\newcommand{\ba}{\begin{myarray}}
\newcommand{\ea}{\end{myarray}}

\nc{\hs}{\hspace*}
\nc{\vs}{\vspace*}
\nc{\set}[2]{\left\{{#1}\mid{#2}\right\}}
\nc{\snoi}{\smallskip\noi}
\nc{\mnoi}{\medskip\noi}
\nc{\al}{\alpha}
\nc{\rmz}{\setminus\{0\}}
\nc{\tens}
[1][]{\mathbin{\otimes{\ms{2mu}}_{\raise1.5ex\hbox to-.1em{}#1}}}
\nc{\vphi}{\varphi}
\nc{\ee}{\end{enumerate}}
\nc{\la}{\lambda}
\nc{\bc}{\begin{cases}}
\nc{\ec}{\end{cases}}
\nc{\qtq}[1][and]{\quad\text{#1}\quad}
\nc{\qt}[1]{\quad\text{#1}}
\nc{\dual}{{\displaystyle{\ms{1mu}\star}}}
\nc{\wle}{\preceq}
\nc{\epito}{\twoheadrightarrow}
\nc{\epiTo}[1][]{\xymatrix@C=4ex{{}\ar@{->>}[r]^-{#1}&{}}}
\nc{\Proof}{\begin{proof}}
\nc{\lan}{\langle}
\nc{\ran}{\rangle}
\nc{\ang}[1]{\lan{#1}\ran}
\nc{\QED}{\end{proof}}
\nc{\soplus}{\mathbin{\raisebox{.1ex}{\scalebox{.65}{\raisebox{.4ex}{$\displaystyle\bigoplus$}}}}}
\nc{\eps}{\varepsilon}
\nc{\supp}{\on{supp}}
\nc{\sct}{strongly commute\xspace}
\nc{\scts}{strongly commutes\xspace}
\nc{\bce}{\eta}			
\nc{\height}[1]{\on{ht}(\ms{.5mu}{#1}\ms{.5mu})}
\nc{\braid}{{\ms{1mu}\mathrm{br}}}
\nc{\gp}{\mathfrak{p}}
\nc{\gb}{\mathfrak{b}}
\nc{\wtl}{\wlP}
\nc{\ra}{real and admits an affinization}
\nc{\ras}{real and admit affinizations}
\nc{\Cor}{\begin{coro}}
\nc{\encor}{\end{coro}}
\nc{\shf}{\mathcal{F}}
\nc{\Cw}[1][{w}]{\catC_{{#1}}}
\nc{\Cwb}[1][{w}]{\catC^{\ms{2mu}\mathrm{big}}_{{#1}}}
\nc{\tCw}[1][{w}]{\widetilde{\catC}_{{#1}}} 
\nc{\tCwv}[1][{w,v}]{\widetilde{\catC}_{{#1}}} 
\nc{\akew}[1][1ex]{\rule[-1ex]{#1}{0ex}}
\nc{\ake}[1][2ex]{\rule[-1ex]{0ex}{#1}}
\nc{\akete}[1][-1ex]{\rule[{#1}]{0ex}{1ex}}
\nc{\tRm}{(R\gmod)\widetilde{\mbox{$\ake[2.5ex]\akew[.9ex]$}}}
\nc{\monoTo}[1][]{\xymatrix{\ar@{>->}[r]^-{{#1}}&}}
\nc{\monoto}[1][]{\rightarrowtail}
\nc{\tX}{\widetilde{X}}
\nc{\corps}{\corp}
\nc{\tL}{\widetilde{L}}
\nc{\prtl}{\ms{2mu}\rtl_+}
\nc{\nrtl}{\ms{2mu}\rtl_-}
\nc{\nrt}{\rt_-}
\nc{\tK}{\widetilde{K}}
\nc{\tep}{\widetilde\ep}
\nc{\teps}{\widetilde\ep}
\nc{\tmu}{\widetilde \mu} 
\nc{\teta}{\widetilde\eta}
\nc{\ga}{\mathfrak{a}}
\nc{\scbul}{{\,\raise1pt\hbox{$\scriptscriptstyle\bullet$}\,}}
\nc{\bwr}{\mbox{\large$\wr$}}
\nc{\tR}{{\widetilde{\mathrm{R}}}}
\nc{\lS}{\mathsf{S}}
\nc{\lZ}{\mathcal{Z}}
\nc{\prolim}[1][]{\mathop{\varprojlim}\limits_{{#1}}}
\nc{\sym}{\sg}
\newcounter{myc}

\nc{\txi}{\tilde{\xi}}
\nc{\rl}{\rlQ}
\nc{\sfC}{\mathsf{C}}
\nc{\cor}{{\ms{1mu}\mathbf{k}\ms{1mu}}}
\nc{\Pro}{\on{Pro}}

\nc{\hM}{\widehat{\mathsf{M}}}
\nc{\aff}{\mathrm{aff}}
\nc{\rDa}{{\mathscr{D}_\aff}}
\nc{\st}[1]{\left\{{#1}\right\}}
\nc{\W}{\mathsf{W}}
\nc{\rt}{\Delta}
\nc{\pwtl}{\wtl_+}
\nc{\rev}{{\mathrm{rev}}}
\nc{\E}{\mathrm{E}}

\nc{\Qt}[1][w]{\mathrm{Q}_{#1}}
\nc{\Qtl}[1][w]{\mathrm{Q}^{\mathrm l}_{#1}}
\nc{\Qtr}[1][w]{\mathrm{Q}^{\mathrm r}_{#1}}
\nc{\Ctr}{\mathsf{C}}
\nc{\Ctrs}{{\mathsf{C}^*}}
\nc{\Dynkin}{\Delta}
\nc{\cartan}{\mathsf{C}}
\nc{\sfc}{\mathsf{c}}
\nc{\sfa}{\mathsf{a}}
\nc{\sfb}{\mathsf{b}}
\nc{\SW}{\mathrm{K}}
\nc{\hSW}{\widehat{\mathrm{K}}}
\nc{\refl}{\mathscr{S}}
\nc{\Rre}{\mathrm{R}^{\mathrm{ren}}}
\nc{\Rpre}{\mathrm{R}^{\mathrm{ren}\;'}}
\nc{\bRre}{\ol{\mathrm{R}}^{\ms{2mu}\mathrm{ren}}}
\nc{\sfd}{\ms{1mu}\mathsf{d}\ms{1mu}}
\nc{\shm}{\mathcal{M}}
\nc{\sht}{\mathcal{T}}
\nc{\rank}{\mathrm{rank}}
\nc{\Da}{{\D}\ms{-2.8mu}\raisebox{-.35ex}{$\scriptstyle\mathrm{aff}$}}
\nc{\lDa}{\Da^{-1}}
\nc{\bchi}{{\scalebox{.9}{\mbox{$\mathscr{E}$}}}}
\nc{\bchis}{\bchi{}^{\ms{2mu}*}}
\nc{\Daf}{\mathscr{D}}
\nc{\Laf}{\mathscr{L}}
\nc{\tLaf}{\widetilde{\Laf}}
\nc{\wtaf}{\mathscr{W}\ms{-3mu}{\mathit{t}}}
\nc{\res}[1][]{\mathop\star\limits_{\raisebox{.4ex}{$\scriptstyle #1$}}\ms{2mu}}
\nc{\hchi}{\widehat{\chi}}
\nc{\convaff}{\mathop{\scalebox{1.1}{$\mspace{1.5mu}\circ\mspace{1.5mu}$}}\limits}
\nc{\cvb}{CVB\xspace}
\nc{\svelt}{essentailly small\xspace}
\nc{\Proc}{\on{Pro}_{\mathrm{coh}}}
\nc{\sha}{\mathcal{A}}
\nc{\Ker}{\on{Ker}}
\nc{\Coker}{\on{Coker}}
\nc{\Aff}[1][z]{\on{Aff}_{\ms{1mu}#1}}
\nc{\scb}{\scalebox}
\nc{\afr}{affreal\xspace}
\nc{\epifrom}{\ms{-5mu}\xymatrix@C=3ex{{}&{}\ar@{->>}[l]}\ms{-5mu}}
\nc{\Mid}{\bigm|}
\nc{\ol}{\overline}
\nc{\bpsi}{\ol{\psi}}
\nc{\Rat}[1][z]{\on{Raff}_{\ms{1mu}#1}}

\nc{\inddlim}{\mathop{\mbox{``{$\ms{1mu}\varinjlim$}''}}\limits}
\nc{\Rmat}{\mathrm{R}\ms{1mu}}
\nc{\Runi}{\mathrm{R}^{\mathrm{univ}}}
\nc{\Modg}{\mathrm{Mod}_{\mathrm{gr}}}
\nc{\KO}{quasi-rigid\xspace}
\nc{\hF}{\widehat{\F}}
\nc{\Modc}{\Mod_{\mathrm{coh}}}
\nc{\e}{\mathrm{e}}
\nc{\Idx}{\mathsf{\Lambda}}
\nc{\hA}{\widehat{A}}
\nc{\prood}{\mathop{\text{``}\prod\text{''}}\limits}

\nc{\hrefl}{\widehat{\mathscr{S}}}
\nc{\ev}{\mathrm{ev}}
\nc{\coev}{\mathrm{coev}}
\nc{\ihom}{\mathcal{H}om}
\nc{\tY}{\widetilde{Y}}
\nc{\tensz}{\tens[z]\ms{-3.5mu}}

\nc{\tRre}{\widetilde{\mathrm{R}}^{\mathrm{ren}}}
\nc{\dg}{\mathbf{\lambda}}
\nc{\htens}{\hconv}
\nc{\stens}{\sconv}
\nc{\Modgc}{\mathrm{Modg}_{\mathrm{coh}}}
\nc{\Aut}{\mathrm{Aut}}
\nc{\Rd}[1][\dg]{R_{#1}\gmod}
\nc{\nn}{\nonumber}
\nc{\Dual}{\mathrm{D}\ms{1mu}}
\nc{\DA}[1][A]{\ms{1mu}\mathrm{D}_{{#1}}}
\nc{\DmA}[1][A]{\ms{1mu}\Dual^{-1}_{{#1}}}
\nc{\tensa}{\tens[A]\ms{-3mu}}
\nc{\tensc}{\tens[{\ms{3mu}\cor}]\ms{-3mu}}

\nc{\ble}{\preccurlyeq}
\nc{\bge}{\succcurlyeq}
\nc{\Cwv}[1][{w,v}]{\catC_{#1}}

\nc{\afn}{affine object\xspace}
\nc{\afns}{affine objects\xspace}
\nc{\subafn}{affine subobject\xspace}
\nc{\Afns}{Affine objects\xspace}
\nc{\mr}{\mathrm{r}}
\nc{\ml}{\mathrm{l}}
\nc{\LQ}{\mathscr{L}}
\nc{\RQ}{\mathscr{R}}
\nc{\tM}{\widetilde{M}}
\nc{\tN}{\tilde{N}}
\nc{\convz}[1][z]{\underset{#1}{\circ}}
\nc{\uw}[1][w]{\underline{#1}}
\nc{\wuw}[1][{\uw}]{\widetilde{#1}}
\nc{\Bw}[1][{\uw}]{\mathbf{B}({#1})}
\nc{\CP}[1][{\uw}]{\mathrm{K}_{#1}} 
\nc{\CK}[1][{w,w'}]{\mathrm{K}'_{#1}} 
\nc{\Es}[1][i]{\mathrm{E}^*_{{#1}}\ms{1mu}}
\nc{\Est}[1][{w,w'}]{\mathbf{E}^*_{#1}}
\nc{\CBw}[1][w]{\mathfrak{B}_{#1}}
\nc{\Em}{E^{\max}}
\nc{\Esm}{\tE^{*\;\max}}
\nc{\car}{\mathrm{ch}}
\nc{\tdC}{\widetilde{\dC}}
\nc{\rert}{\rt^{\mathrm{real}}}
\nc{\Qti}[1][i]{\Qt(\ang{#1})}
\nc{\Esi}[1][i]{\mathbf{E}^*_{#1}}
\nc{\iz}{\ang{i}_z}
\nc{\by}{\mathrel{/}}
\nc{\epss}{\eps^*}
\nc{\Qr}[1][w,s_i]{\mathrm{Q}^{\mathrm r}_{#1}}
\nc{\tS}{\widetilde{S}}
\nc{\sdc}{\mathsf{c}}
\nc{\Rw}[1][w]{R_{\ms{-3mu}{#1}}}
\nc{\Pw}[1][w]{\mathsf{P}^{\ms{1mu}\raisebox{.25ex}{$\scriptstyle#1$}}}
\nc{\op}{\mathrm{op}}
\nc{\bi}{\mathbf{i}}
\nc{\tbi}{\tilde{\mathbf{i}}}
\nc{\bj}{\mathbf{j}}
\nc{\tbj}{\tilde{\mathbf{j}}}
\nc{\ibox}[1][$\bi$]{{#1}-box\xspace}
\nc{\iboxes}[1][$\bi$]{{#1}-boxes\xspace}
\nc{\calM}{\mathcal{M}}
\nc{\calC}{\mathcal{C}}
\nc{\Aqn}{A_q(\n)} 
\nc{\tseed}{\widetilde{\seed}}
\nc{\hseed}{\widehat{\seed}}
\nc{\K}{\mathscr{K}}
\nc{\G}{\mathscr{G}}
\nc{\fr}{{\mathrm{fr}}}
\nc{\ex}{{\mathrm{ex}}}
\nc{\monomial}[1][{\seed}]{${#1}$-monomial\xspace}
\nc{\monomials}[1][{\seed}]{${#1}$-monomials\xspace}
\nc{\wseed}{\mathscr{T}}
\nc{\twseed}{\widetilde{\mathscr{T}}}
\nc{\cM}{\mathcal{M}}
\nc{\cfac}{convolution factor\xspace}
\nc{\cf}{\cfac}

\nc{\Ks}{\mathsf{K}}
\nc{\tKs}{\widetilde{\mathsf{K}}}
\nc{\Js}{\mathsf{J}}
\nc{\tJs}{{\widetilde{\mathsf{J}}}}
\nc{\ko}{{k_0}}
\nc{\seedk}[1][k]{\seed^{(#1)}}
\nc{\Loc}[1][w,v]{\mathrm{Q}_{#1}}
\nc{\quot}[1][{w,v}]{\mathrm{Q}_{#1}}
\nc{\otens}{\mathop{\nabla}\limits^{\xrightarrow{}}}
\nc{\Kq}[1][{w,v}]{\mathrm{K}(\Cwv[#1])\vert_{q=1}}
\nc{\Kt}{\mathcal{K}}
\nc{\Ft}{\mathcal{F}}
\nc{\txc}{\textcircled}
\nc{\ef}{\mathrm{ef}}
\nc{\gL}{\textsf{g}^{\mathrm L}}
\nc{\gR}{\textsf{g}^{\mathrm R}}
\nc{\tC}{\widetilde{\mathcal{C}}}
\nc{\kt}{\mathsf{k}}
\nc{\ft}{\mathsf{f}}
\nc{\qF}{\mathrm{F}}
\nc{\qE}{\mathrm{E}}
\nc{\kfc}{\mathsf{c}}
\nc{\sK}{\mathbb{K}}
\nc{\sF}{\mathbb{F}}
\nc{\sG}{\mathbb{G}}
\nc{\bfa}{\mathbf{a}}
\nc{\bfb}{\mathbf{b}}
\nc{\bfc}{\mathbf{c}}
\nc{\bfd}{\mathbf{d}}
\nc{\bfe}{\mathbf{e}}
\nc{\ca}{\mathscr{A}}
\nc{\uca}{\mathscr{U}}
\nc{\sfJ}{\mathsf{J}}
\nc{\sfK}{\mathsf{K}}
\nc{\GLS}{\mathrm{GLS}}
\nc{\sgr}{\mathsf{r}}   
\nc{\calT}{\mathcal{T}}
\nc{\frakC}{\mathfrak{C}}
\nc{\frakc}{\mathfrak{c}}
\nc{\cC}{\mathcal{C}}
\nc{\Zqh}{\Z[q^{\pm1/2}]}
\nc{\Qqh}{\Q[q^{\pm1/2}]}
\nc{\cmu}{in the canonical mutation class\xspace}
\nc{\ncf}{have no common convolution factor other than $\one$\xspace}
\nc{\cms}{full quantum monoidal seed\xspace}
\nc{\cs}{full seed\xspace}
\nc{\prtm}{\prt^{\min}}
\nc{\chl}{\mathrm{span}_{\R_{\ge0}}}
\nc{\Rwv}[1][{w,v}]{\mathcal{R}_{{#1}}}
\nc{\seedLec}{\Sigma_{\uw, v}^{\rm Lec}}
\nc{\seedMenard}{\Sigma_{\uw, v}^{\rm Menard}}
\nc{\seedMBY}{\Sigma_{\uw, v}^{\rm MBY}}

\numberwithin{equation}{section}
\title[Monoidal seeds of $\Cwv$]{Monoidal seeds of the categories $\Cwv$ over quiver Hecke algebras}

\author[M. Kashiwara]{Masaki Kashiwara}
\thanks{The research of M.\ Kashiwara
	was supported by Grant-in-Aid for Scientific Research (B)  23K20206,  
	Japan Society for the Promotion of Science.}
\address[M. Kashiwara]{%
	Kyoto University Institute for Advanced Study, Research Institute
	for Mathematical Sciences, Kyoto University, Kyoto 606-8502, Japan;
    \& Korea Institute for Advanced Study, Seoul 02455, South Korea
}
\email[M. Kashiwara]{masaki@kurims.kyoto-u.ac.jp}

\author[M. Kim]{Myungho Kim}
\address[M. Kim]{Department of Mathematics, Kyung Hee University, Seoul 02447, South Korea}
\email[M. Kim]{mkim@khu.ac.kr}
\thanks{The research of M.\ Kim was supported by the National Research Foundation of Korea (NRF) Grant funded by the Korea government(MSIT) (NRF-2020R1A5A1016126).}

\author[S.-j. Oh]{Se-jin Oh}
\thanks{ The research of S.-j.\ Oh was supported by the National Research Foundation of
	Korea (NRF) Grant funded by the Korea government(MSIT) (NRF-2022R1A2C1004045).}
\address[S.-j. Oh]{ Department of Mathematics, Sungkyunkwan University, Suwon, 16419, South Korea}
\email[S.-j. Oh]{sejin092@gmail.com}

\author[E. Park]{Euiyong Park}
\thanks{The research of E.\ Park was supported by the National Research Foundation of Korea (NRF) Grant funded by the Korea Government(MSIT)(RS-2023-00273425 and NRF-2020R1A5A1016126).}
\address[E. Park]{Department of Mathematics, University of Seoul, Seoul 02504, South Korea}
\email[E. Park]{epark@uos.ac.kr}

\makeatletter
\@namedef{subjclassname@2020}{\textup{2020} Mathematics Subject Classification}
\makeatother

\keywords{Categorification, Cluster algebras, Monoidal seeds, Quiver Hecke algebras, Richardson varieties}
\subjclass[2020]{17B37, 18N25, 13F60}

\date{August 15, 2026}

\begin{document}

\maketitle
\begin{abstract}

In this paper, when the quiver Hecke algebra $R$ is  symmetric, we present a new construction of quantum monoidal seeds $ \seed_{\uw,v}$ for $\Cwv$ using the reflection functors $\F_i$ and the newly introduced operators $\K_i$. The monoidal seed $ \seed_{\uw,v}$ of $\Cwv$ is obtained as a subseed of the monoidal seed $\wseed_{\uw,\kfc} $ of $\Cw$ constructed by applying $\K_i$ and $\F_i$ along the special KF sequence $\kfc$ determined by $\uw$ and $v$. 
We further prove that the monoidal seed $ \seed_{\uw,v}$ coincides with the set of all prime factors of the determinantial modules $M(w_{\le k } \La_{i_k}, v_{\le k} \La_{i_k}  )$ for $k \in [1,r]^{\uw}_v$.
Let $\ca_{\uw,v}$ (resp.\ $\uca_{\uw,v}$) denote the (resp.\ upper) quantum cluster algebra with non-invertible frozen variables associated with the initial seed $\Sigma_{\uw,v}$ determined by $\seed_{\uw,v}$.
We prove that the Grothendieck ring $K(\Cwv) $ of $\Cwv$ lies between $\ca_{\uw, v}$ and $\uca_{\uw, v}$, and conjecture that   
$\ca_{\uw, v} = \uca_{\uw, v}$, which implies $\Cwv$ provides the monoidal categorification of the cluster algebra $K(\Cwv)$ without the localization process.

\end{abstract}
\tableofcontents

\section{Introduction} 

\emph{Quiver Hecke algebras}, introduced  independently by Khovanov--Lauda and Rouquier (\cite{KL1, KL2, R08}), have occupied a central position in the study of the categorification of quantum groups (see \cite{Brundan13, KKKO18, Kas18} for example and references therein). 
Let $\cmC$ be a symmetrizable generalized Cartan matrix and let $U_q(\g)$ and $R$ be the associated quantum group and quiver Hecke algebra respectively. 
It was shown in \cite{KL1, KL2, R08} that  the \emph{quantum unipotent coordinate ring} $A_q(\n)$, the dual of the positive half $U_q^+(\g)$, is categorified by the category $R\gmod$ of finite-dimensional graded $R$-modules, which means that the Grothendieck ring $K(R\gmod)$ is isomorphic to the integral form of $A_q(\n)$. Let $\weyl$ be the Weyl group associated with $\cmC$. 
For any $w,v \in \weyl$ with $v \le w$ in the \emph{Bruhat order}, the subcategory $\Cw[w,v]$ of $R\gmod$ was introduced by the authors in \cite{KKOP18} (see \eqref{eq:Cwv}). The category $\Cwv$ is defined as a subcategory of the category $\Cw$ which provides a categorification of the quantum unipotent coordinate ring $A_q(\n(w))$, regarded as a quantum deformation of the coordinate ring $\C[N(w)]$ of the unipotent subgroup $N(w)$ associated with $w$. When $v = \id$, $\Cw[w,v]$ coincides with the category $\Cw$. 
The category $\Cw[w,v]$ categorifies the $q$-deformation $A_{w,v} \subset A_q(\n)$ of the \emph{doubly-invariant algebra} ${}^{N'(w)}\C[N]^{N(v)}$, where $N$ is the unipotent subgroup (see \cite[Remark 2.19, Theorem 2.20]{KKOP18} for details) and its localized category $\tCw[w,v]$ by central \emph{determinantial modules} categorifies the coordinate ring $\C[\Rwv]$ of the \emph{open Richardson variety} $\Rwv$ at the quantum level (\cite{KKOP23A}). The open Richardson varieties have deep connections to various areas of mathematics, including Kazhdan--Lusztig polynomials, total positivity and cluster algebras (see \cite{ KL79, KL80, Luz98, Kons99, Lec16} and references therein).

\emph{Cluster algebras}, introduced by Fomin and Zelevinsky \cite{FZ02}, are commutative algebras equipped with distinguished generators 
grouped into \emph{clusters} together with procedures called 
\emph{mutations} to connect one cluster to neighbor clusters.
 A cluster algebra arises naturally from the open Richardson variety $\Rwv$. 
It was shown by Leclerc (\cite{Lec16}) that, when $\cmC$ is of finite $ADE$ type,  the coordinate ring $\C[\Rwv]$ contains a cluster algebra arising from the module category of the preprojective algebra determined by $w$ and $v$. 
The initial seed $\seedLec$ of this cluster algebra was constructed for each reduced expression $\uw$ of $w$, and the cluster variables of $\seedLec$ are obtained by taking irreducible factors of the generalized unipotent minors determined by $\uw$ and $v$ (see \cite[Section 4.8.4]{Lec16}). 
M\'enard later presented an algorithm to compute initial seeds for $\C[\Rwv]$ 
: starting from the seed of $\C[\Rw[w,\id]]$ associated with $\uw$, one applies a suitable mutation sequence and then freezes and   deletes certain cluster variables
 (\cite{Menard22}). This construction is generalized by Bao and Ye for  a symmetrizable Kac-Moody $\g$ in \cite{BY25}. 

Leclerc conjectured that the cluster algebra $\ca(\seedLec)$ with invertible frozen variables coincides with $\C[\Rwv]$, and showed the conjecture in some special cases in \cite{Lec16}. 
In type $A$, Leclerc's conjecture was proved by Serhiyenko and Sherman-Bennett (\cite{SSB24}) by comparing with Ingermanson's result (\cite{Ingermanson19}). 
In an arbitrary finite type, the conjecture was proved in the setting of 
\emph{braid varieties} which are generalizations of open Richardson varieties (see \cite{CGGLSS25, GLSB26} and see also \cite{CGGSSBS25}). For the upper cluster algebra structure, it was proved that $\C[\Rwv]$ coincides with the upper cluster algebra $\uca(\seedLec) $ in \cite{CB22} for $ADE$ case, \cite{CGGLSS25, GLSB26} for finite type, and  \cite{BY25} for Kac-Moody case.

\emph{Monoidal categorification} provides a powerful framework to study cluster algebras from the viewpoint of monoidal categories. The notion of monoidal categorification for cluster algebras was introduced by Hernandez-Leclerc in studying the Grothendieck ring of the category of finite-dimensional modules over quantum affine algebras (\cite{HL10}). When $\cmC$ is of symmetric type, the category $\Cw$ provides a monoidal categorification of the quantum unipotent coordinate ring $A_q(\n(w))$ (\cite{KKKO18}).
That is, every cluster monomial corresponds to a simple module in $\catC_w$. 
Let $\uw$ be a reduced expression of $w$ of length $r = \ell(w)$. 
For each \emph{admissible chain} $\frakC = ( [a_k,b_k])_{1 \le k \le r}$ of $i$-boxes associated with $\uw$, one can construct a  quantum monoidal seed  $\wseed_\frakC$ of $\Cw$ whose \emph{cluster variable modules}, simple modules corresponding to cluster variables  under the categorification, are given by the \emph{determinantial modules} $\dM_{\uw}[a_k, b_k]$ for $k \in [1,r]$ (see \cite{KKKO18,KK24} and see also Section \ref{subsec:monoidal} for details). The determinantial modules $\dM_{\uw}[a_k, b_k]$ categorify the quantum unipotent minors in $A_q(\n(w))$. 
In the case of $\frakC = ( \{1,k] )_{1 \le k \le r}$, 
the monoidal seed $\wseed_\frakC$  categorifies the seed of $A_q(\n(w))$ given by Geiss-Leclerc-Schr\"{o}er in \cite{GLS11, GLS13} under the categorification (see \eqref{Eq: GLS seed}).
Since the algebra $\C[\Rwv[w, \id]]$ is isomorphic to the localized algebra $\widetilde{A_q(\n(w))}|_{q=1}$ and the localized category $\tCwv$ categorifies the coordinate ring $\C[\Rwv]$ of the open Richardson variety $\Rwv$, the approach by M\'enard  and Bao-Ye to computing seeds for $\C[\Rwv]$  provides 
 monoidal seeds in $\Cwv$ from those of $\Cw$. 
In \cite{Bi26}, Bi showed that the monoidal seed obtained in this way is an initial seed of a cluster algebra contained in $K(\catC_{w,v})$, and that every cluster monomial corresponds to a simple module in $\catC_{w,v}$, based on the result of \cite{KKKO18}. In particular, when $\g$ is of finite $ADE$ type,  
he showed  that the seed of Leclerc is same as the seeds of M\'enard and proved that the localized category $\tCwv$ provides a monoidal  categorification of the cluster algebra $\C[\Rwv]$. 
Although M\'enard's  algorithm  and Bao-Ye's generalization show the existence of a monoidal seed of $\Cwv$ as a \emph{subseed} of a monoidal seed of $\Cw$, the algorithm does not provide a direct construction of cluster variable modules in a monoidal seed of $\Cwv$ since it proceeds by applying a certain sequence of mutations.

In this paper, when the quiver Hecke algebra $R$ is \emph{symmetric}, we present a new construction of quantum monoidal seeds for $\Cwv$ using the \emph{reflection functors} $\F_i$ (\cite{KKOP26, Kato20} and Section \ref{Sec: RF}) and the newly introduced operators $\K_i$ (Section \ref{Sec: Ki}). 
For each reduced expression $\uw$ of $w$, our new construction (Theorem \ref{thm: Swv for Cwv}) provides a quantum monoidal seed $ \seed_{\uw,v}$ of $\Cwv$ whose cluster variable modules  can be  constructed by
the crystal operators, since the operators $\K_i$ and $\F_i$  admit  explicit descriptions in terms of crystals (Definition \ref{Def: Ki and K*i} and \eqref{Eq: F_i in crystal}). 
We further prove that the seed $ \seed_{\uw,v}$ coincides with the set of all prime factors of the determinantial modules $M(w_{\le k } \La_{i_k}, v_{\le k} \La_{i_k}  )$ for $k \in [1,r]^{\uw}_v$ and that, when $\g$ is of finite type $ADE$, $ \seed_{\uw,v}$ corresponds to the seed of Leclerc in \cite[Corollary 4.4]{Lec16} under the categorification (Proposition \ref{prop:SeedLeclerc}).
In contrast to the approaches of M\'enard and Leclerc, our construction determines directly the cluster variable modules in terms of $\K_i$ and $\F_i$
without performing mutations or relying on the factorization of generalized unipotent minors into irreducible factors.

The operator $\K_i$ (resp.\ $\K^*_i$) is defined in the \emph{symmetrizable setting} using the crystal operators $\tF_i$, $\tE_i$ (resp.\ $\tFs_i$, $\tEs_i$) together with the $\de_i$ defined in \eqref{Eq: def of di} (see Definition \ref{Def: Ki and K*i}).  
For any simple module $M$,  the $\K_i(M)$ (resp.\ $\K^*_i(M)$) can be understood as the simple module on the $(\tF_i, \tE_i)$-generated (resp.\ $(\tFs_i, \tEs_i)$-generated) string through $M$  which commutes with $\ang{i}$ \emph{minimally} (see \eqref{Eq: de Ki Ksi} and Lemma \ref{Lem: basics for Ki} \eqref{Lem: basics for Ki (ii)}), where $\ang{i}$ is the 1-dimensional simple $R(\al_i)$-module.
We show that $\K_i$ and $\K^*_i$ preserve both of the \emph{affreal} property (Lemma \ref{Lem: basics for Ki}) and the commuting property (Lemma \ref{Lem: KiM KiN}).  
In the case of simple modules satisfying $\ep_i=0$ (resp.\ $\ep_i^* =0$), we further investigate how the integer-valued invariants $\La$ change under the operators $\K_i$ (resp.\ $\K^*_i$), and prove that $\K_i$ (resp.\ $\K^*_i$) preserves both the integer-valued invariants $\de$ and the primeness (see Lemma \ref{lem:Kmain} and Lemma \ref{Lem: primeness}).  

The operators $\K_i$ together with the reflection functor $\F_i$ lead to a new construction of monoidal seeds. \emph{We assume that $R$ is symmetric}.  
For any reduced expression $\uw = (i_1, i_2, \ldots, i_r)$ of an element $w \in \weyl$ and any sequence
$\kfc = (c_1, c_2, \ldots, c_r) \in \{ \kt, \ft \}^r$  (that we call
a \emph{KF sequence}),  we define the family of simple modules 
$$
\wseed_{\uw,\kfc} =  (M_k)_{k\in[1,r]}
$$ 
by applying $\K_i$ and $\F_i$ to one of $\ang{i_k}$ and the determinantial module $\dM(w_{\ge k}\La_{i_k},\La_{i_k})$ along the sequence $\kfc$ (see \eqref{Eq: Mk G} for the definition). Proposition \ref{prop: Twc} says that the family $\wseed_{\uw,\kfc}$ is a monoidal seed of $\Cw$ in the \emph{canonical mutation class} (see Remark \ref{Rmk: B from M}). 
We remark that, in general, the seed $\wseed_{\uw,\kfc}$  differs  from the monoidal seeds arising from \emph{admissible chains} of \emph{$i$-boxes} associated with $\uw$ (see Example \ref{Ex: Cw}). 

Let $w,v \in \weyl$ with $v \le w$ and take $i \in I$ such that 
$s_iw < w$.
 One of the key observations  (Proposition \ref{prop:stability}) is that,
\bna
\item
we have $\ang{i} \in \Cwv$
and
$\K_i$ sends simple modules in $\Cwv[s_iw,v]$ to simple modules in $\Cwv[w,v]$
whenever $v<s_iv$, 
\item $\F_i$ sends simple modules in $\Cwv[s_iw,s_iv]$ to simple modules in $\Cwv[w,v]$ whenever $ s_iv < v $.
\ee
This leads us to a new construction of monoidal seeds for $\Cwv$.
For any reduced expression $\uw$ of $w$, we define the subset $[1,r]_v^{\uw}$ of the interval $[1,r]$ by using $\uw$ and $v$ (see \eqref{Eq: [1,r]wv}) and consider the monoidal seed $ \wseed_{\uw,\kfc} =  (M_k)_{k\in[1,r]}$ of $\Cw$ associated with $\uw$ and the \emph{special KF sequence} $\kfc$ determined by $[1,r]_v^{\uw}$ (see \eqref{Eq: ct for T}). We then define 
$$
\seed_{\uw, v} =  ( M_k)_{k\in [1,r]_v^{\uw}} 
$$
as a subset of $\wseed_{\uw,\kfc}$. In Theorem \ref{thm: Swv for Cwv}, we prove that $\seed_{\uw, v}$ is a completely admissible monoidal seed (see
Definition~\ref{Def: def MC}) in  $\Cwv$ by investigating the relationship between the two seeds in $\Cwv[s_iw]$ and $\Cwv[s_iw, v]$ under the operator $\K_i$ (see Section \ref{Sec: pair}).
Theorem \ref{thm: one mutation class for Cwv} says that the mutation class of the seed $\seed_{\uw, v} $ does not depend on the choice of  reduced expressions $\uw$.
The set $\sfJ_{\uw, v}^{\fr}$ of the frozen vertices of $\seed_{\uw, v}$ is described explicitly in  \eqref{Eq: index for Cwv} and Lemma \ref{Lem: frozen Jfr}. 
Note that the number $|\sfJ_{\uw, v}^{\fr}|$ is related to the coefficient of the second largest
power in \emph{Kazhdan-Lusztig R-polynomials} (see \cite[Section 3.4]{ELPSW26} and Remark \ref{Rmk: KL R poly}).  
It turns out that the seed $\seed_{\uw, v}$ is equal to the seed of Leclerc 
(see Proposition \ref{prop:SeedLeclerc}).
Hence  \cite[Theorem 5.11]{Bi26} implies that $\seed_{\uw, v}$ is the same with the seed of 
M\'enard and Bao-Ye's generalization when   $\g$ is symmetric.  
We illustrate the monoidal seeds $\wseed_{\uw,\kfc}$ and $\seed_{\uw,v}$  with explicit examples (see Example \ref{Ex: Cw} and \ref{Ex: Cwv}) in the paper.

We also investigate the compatibility of $\K_i$ and  $\F_i$ with \emph{Laurent families} introduced in \cite{KKOP24B}. The Laurent family is a special commuting family of affreal simple modules which satisfies a categorical property analogous to the \emph{Laurent phenomenon} (see Section \ref{Sec: prep}). 
In Theorem \ref{th:cao}, we give a formula for computing $\La$ in terms of $g$-vectors in the context of Laurent families. This can be viewed as a generalization of the formula expressed in terms of the \emph{tropical invariant} by Cao (\cite[Theorem 5.16]{Cao23}) from cluster variable modules to real simple modules (Remark \ref{Rmk: Cao}).
 We show that $\K_i$ and $\F_i$ preserve Laurent families (Proposition \ref{prop: laurent}), and prove that
$$
\ca_{\uw, v} \subset K(\Cwv) \subset \uca_{\uw, v},
$$
where $\ca_{\uw, v}$ and $\uca_{\uw, v}$ are
the quantum cluster algebra and the upper quantum cluster algebra with \emph{non-invertible frozens} arising from the initial seed $\Sigma_{\uw, v}$ determined by $\seed_{\uw,v}$, respectively (Theorem \ref{Thm: A K U}). 
We conjecture that   
$\ca_{\uw, v} = \uca_{\uw, v}$, which implies $\Cwv$ provides the monoidal categorification of $K(\Cwv)$ \emph{without the localization process} (see Conjecture \ref{Conj: MC}).

This paper is organized as follows.
In Section \ref{Sec: Prelim}, we briefly review the necessary background including quantum groups, quiver Hecke algebras and Laurent families, and introduce the new operators $\K_i$ and $\K_i^*$. 
In Section \ref{Sec: Laurent Cwv}, we investigate how $\K_i$ and $\F_i$ act on Laurent families in $\Cwv$. 
Section \ref{Sec: MC Cw} reviews the notion of monoidal categorification, and Section \ref{Sec: seed for Cw} explains how the monoidal seeds $\wseed_{\uw,\kfc}$ for $\Cw$ are constructed  using $\K_i$ and $\F_i$. 
Section \ref{Sec: Swv for Cwv} introduces the new construction of monoidal seeds $\seed_{\uw,v}$ of $\Cwv$ in terms of $\K_i$ and $\F_i$, and Section \ref{Sec: mutation invariance} shows the mutation invariance of the monoidal seeds $\seed_{\uw, v} $.
Section \ref{Sec: cluster alg str of Cwv} shows  the cluster algebra structure on $\Cwv$.

\medskip
\noindent
{\bf Acknowledgments} \ 
The second, third, and fourth authors gratefully acknowledge the
hospitality of RIMS, Kyoto University during their visit in 2026. 
E.\ P.\ would like to thank the Department of Mathematics at the University of Connecticut for its hospitality and excellent research environment during his visit.

\vskip 1em 

\subsection*{Convention}
Throughout this paper, we use the following convention.
\bnum
\item For a statement $P$, we set $\delta(P)$ to be $1$ or $0$ depending on whether $P$ is true or not. In particular, $\delta_{i,j}=\delta(i =j)$
is the Kronecker delta. 

\item For $a\in \Z \cup \{ -\infty \} $ and $b\in \Z \cup \{ \infty \} $ with $a\le b$, we set 
\begin{align*}
& [a,b] =\{  k \in \Z \ | \ a \le k \le b\}, &&  [a,b) =\{  k \in \Z \ | \ a \le k < b\}, \allowdisplaybreaks\\
& (a,b] =\{  k \in \Z \ | \ a < k \le b\}, &&  (a,b) =\{  k \in \Z \ | \ a < k < b\},
\end{align*}
and call them \emph{intervals}. 
When $a> b$, we understand them as empty sets.
\item For a monoid $A$ and an index set $K$, we denote
by $A^{\oplus K}$ the subset of $A^K$ consisting of $\bfa=(a_k)_{k\in K}$
such that $a_k$ coincides with the identity element except finitely many $k\in K$.
\ee

\vskip 2em

\section{Preliminaries} \label{Sec: Prelim}

\subsection {Quantum cluster algebras} \label{Sec: QCA}

In this subsection, we briefly recall the definition of a quantum cluster algebra (\cite{BZ05}). 

Let $\sfJ$ be a finite index set with a decomposition $\sfJ = \sfJ^\ex \sqcup \sfJ^\fr$.
The elements in $\sfJ^\ex$ (resp.\ $\sfJ^\fr$) are called \emph{exchangeable} (resp.\ \emph{frozen}) indices. Let $q$ be an indeterminate. 
For a skew-symmetric integer-valued matrix $L = (l_{ij})_{i,j \in \sfJ}$, a family $x = \{x_i\}_{i\in \sfJ} $ in a skew field $ K$ over $\Q(q^{1/2})$
is \emph{$L$-commuting} if it satisfies $x_ix_j = q^{l_{ij}}x_jx_i $ for $ i,j \in \sfJ.$
We denote by $\T_x$ the  $\Zqh$-subalgebra of $K$ generated by $x$, which is called the \emph{quantum torus} associated with $x$. 
For any $\bfa = (a_i)_{i \in \sfJ} \in \Z^\sfJ$, we set
$$
x(\bfa) \seteq  q^{\frac{1}{2} \sum_{i>j}a_i a_j l_{ij}} \prod^{\longrightarrow}_{i \in \sfJ} x_i^{a_i},
$$
where the ordered product is taken in the decreasing order with respect to a total order on $\sfJ$. 
Note that $x(\bfa)$ does not depend on the choice of a total order.

Let $\tB = (b_{ij})_{i \in \sfJ,j\in \sfJ^\ex}$ be an \emph{exchange matrix}, which is an integer-valued matrix whose principal part $(b_{ij})_{i,j\in \sfJ^\ex}$ is skew-symmetric. We often write $\tB_{ij}$ for the entry $b_{ij}$. 
We say that a pair $(L,\tB)$ is \emph{compatible} if there is a positive integer $d$ such that 
\begin{align*}
\sum_{k \in \sfJ} l_{ik} b_{kj} = d\delta_{i,j} \qt{for any $i\in \sfJ$ and $j\in\sfJ^\ex$. }
\end{align*}
A triple $\Sigma = ( \{ x_k \}_{k \in \sfJ}, L ,\tB)$ is called a \emph{quantum seed} in $K$ if $\{ x_k \}_{k \in \sfJ}$ is an algebraically independent $L$-commuting family in $K$ 
and $(L, \tB)$ is compatible.

For any $k \in \sfJ^\ex$, the \emph{mutation} $\mu_k(\Sigma)$ of a quantum seed $\Sigma = ( \{ x_k \}_{k \in \sfJ}, L ,\tB)$ is written as
$$
\mu_k(\Sigma) \seteq  \bl \mu_{k} (x), \mu_k(L), \mu_k(\tB)\br,
$$ 
where the mutation $\mu_k(\tB)$ (resp.\ $\mu_k(L)$) of $ \tB$ (resp.\ $L$) is described in \cite[Section 5.2]{KKKO18}, and  $\mu_{k} (x) = \{ x_i' \}_{i\in \sfJ} $ is defined as 
$$
x_i' \seteq \bc
x(\bfa') + x(\bfa'') & \text{ if } i =k, \\
x_i  & \text{ if } i \ne k,
\ec
$$
where $ \bfa' = (a_i')_{i\in \sfJ}$ and $ \bfa'' = (a_i'')_{i\in \sfJ}$ are defined as
$$
a'_i = \bc
-1 & \text{ if } i =k, \\
\max(0,b_{ik}) & \text{ if } i \ne k,
\ec
{\qtq}
a''_i = \bc
-1 & \text{ if } i =k, \\
\max(0,-b_{ik}) & \text{ if } i \ne k.
\ec
$$
We say two seeds $\Sigma$ and $\Sigma'$ are \emph{mutation equivalent} if $\Sigma'$ can be obtained from $\Sigma$ by applying suitable mutations up to a permutation of the index set $\sfJ$. 
Note that mutation equivalence gives an equivalence relation on the set of all quantum seeds. The equivalence classes arising from this relation are called \emph{mutation classes}.  

Let $\Sigma = ( \{ x_k \}_{k \in \sfJ}, L ,  \tB)$ be a quantum seed in a skew field $K$ over $\Qqh$. 
The \emph{quantum cluster algebra} $\ca (\Sigma)$ associated with $\Sigma$ is the $\Zqh$-subalgebra of $K$ generated by all cluster variables in quantum seeds obtained from $\Sigma$ by any finite sequence of mutations. For a quantum seed $\Sigma' = ( \{ x_k' \}_{k \in \sfJ}, L' ,  \tB')$ mutation equivalent to $\Sigma$, the \emph{quantum Laurent phenomenon} (\cite{BZ05}) says that $\ca (\Sigma)$ is contained in the quantum torus $\T_{\Sigma'}$ generated by $\{ x_k' \}_{k \in \sfJ}$ inside $K$. We define $\uca (\Sigma)$ to be the intersection of all quantum tori $\T_{\Sigma'}$ in $K$ for all quantum seeds $\Sigma'$ which are mutation equivalent to $\Sigma$. The algebra $\uca (\Sigma)$ is called the \emph{quantum upper cluster algebra} associated with $\Sigma$. 
Note that, by the definition, $\ca (\Sigma)$ is a $\Zqh$-subalgebra of $\uca (\Sigma)$.

\subsection{Quantum groups} \label{Sec: quantum group}
In this subsection, we briefly review quantum groups (\cite{KashBook02, LusztigBook}). 

A {\it Cartan datum} is a family $ (\cmC,\wlP,\Pi,\wlP^\vee,\Pi^\vee,(\cdot,\cdot)) $
 consisting of  a symmetrizable generalized  \emph{Cartan matrix} $\cmC=(c_{ij})_{i,j\in I}$,
 a free abelian group $\wlP$  called the \emph{weight lattice}, its dual $\wlP^{\vee} \seteq \Hom_{\Z}( \wlP, \Z )$ called the \emph{co-weight lattice}, the set $\Pi \seteq  \{ \alpha_i \in \wlP \mid i\in I \}$ of \emph{simple roots}, the set $\Pi^{\vee} \seteq  \{ h_i \in \wlP^\vee \mid i\in I\}$ of \emph{simple coroots} and a 
$\Q$-valued symmetric bilinear form $(\cdot,\cdot)$ on $\wtl$  such that    
 \bnum
\item $\lan h_i, \alpha_j \ran = a_{ij}$ for $i,j \in I$,
\item $\Pi$ is linearly independent over $\Q$,
\item for each $i\in I$, there exists $\Lambda_i \in \wlP$ such that $\lan h_j,\Lambda_i \ran =\delta_{j,i}$ for all $j \in I$.
\item 
$(\al_i,\al_i)\in 2\Z_{>0} $ and 
$ \lan h_i,  \lambda\ran = {2 (\alpha_i,\lambda)}/{(\alpha_i,\alpha_i)}$. 
\ee
We call $\La_i$ the $i$-th \emph{fundamental weight} for $i\in I$, and write $\wlP_+ \seteq  \{ \La \in \wlP \mid \ang{h_i, \La} \ge 0  \text{ for any } i\in I\}$
 and call it the \emph{dominant weight lattice}. 
Let $\rlQ \seteq  \sum_{i\in I} \Z \al_i$ be the \emph{root lattice} and
let $\rlQ_+ \seteq  \sum_{i\in I} \Z_{\ge0} \al_i$ and
$\nrtl\seteq  \sum_{i\in I} \Z_{\le0} \al_i$ be the positive and negative root lattice respectively.
We set $\height{\beta} \seteq  \sum_{i\in I} |a_i|$ 
for an element $\beta = \sum_{i\in I} a_i \al_i \in \rlQ$. 
We denote by $\prD$ (resp.\ $\nrD$) the set of all positive roots (resp.\ negative roots).

We denote by $\weyl = \langle s_i \mid i\in I \rangle$ the \emph{Weyl group} associated with $\cmC$, where $s_i\in\Aut(\wtl)$ is the $i$-th reflection defined by
$$ 
s_i\la = \la -\lan h_i,\la \ran\al_i \quad \text{for $\la \in \wlP$}.  
$$
For $w \in \weyl$, an expression $\uw = s_{i_1} s_{i_2} \cdots s_{i_r} $ of $w$ is \emph{reduced} if $r$ is minimal among all such expressions. 
For simplicity, we write $\uw = (i_1, i_2, \ldots, i_r)$ if no confusion arises. 
For $w,v\in \weyl$, we write $v \le w$ if there is a reduced expression of $w$ containing a reduced expression of $v$ as a subexpression. This partial order $\le$ is called the \emph{Bruhat order} on $\weyl$.

Let $q$ be an indeterminate. 
The \emph{quantum group} $U_q(\g)$ associated with 
a Cartan datum $(\cmC,\wlP,\Pi,\wlP^\vee,\Pi^\vee,(\cdot,\cdot))$ 
is the associative algebra over $\Q(q)$
generated by $e_i,f_i$ $(i \in I)$ and $q^h$ $(h \in \wlP^\vee)$ satisfying certain defining relations (see \cite{LusztigBook} for details), and 
let $U_q^-(\g)$ (resp.\ $U_q^+(\g)$) be the subalgebra of $U_q(\g)$  generated by $f_i$ (resp.\ $e_i$) for all $i \in I$. We also denote by $U(\g)$ the universal enveloping algebra of $\g$ and by $U^-(\g)$ (resp.\ $U^+(\g)$) the corresponding half of $U(\g)$ generated by $f_i$ (resp.\ $e_i$) for all $i \in I$. 

A \emph{crystal} is a set $B$ with  maps $\wt\col  B \rightarrow \wlP$,  $\varphi_i$,  $\ep_i \col  B \rightarrow \Z \,\sqcup \{-\infty\}$
and $\te_i$, $\tf_i \col  B \rightarrow B\,\sqcup\{0\}$ 
for $i\in I$ which satisfy certain conditions determined by the Cartan matrix $\cmC$. We set $B(\infty)$ to be the \emph{infinite crystal} of $U_q^-(\g)$. 
The $\Q(q)$-antiautomorphism $*$ on $U_q(\g)$ defined by
\begin{align*} 
	(e_i)^* = e_i, \qquad (f_i)^* = f_i, \qquad  (q^h)^*  = q^{-h}, 
\end{align*}
gives another crystal structure $ \te_i^*$, $ \tf_i^*$, $\ep_i^*$, $\varphi_i^*$  on $B(\infty)$. For details on crystals, we refer the reader to \cite{K95, KashBook02}. 

For any $\Lambda \in \pwtl$, we denote by $V_q(\La)$ the irreducible highest weight $U_q(\g)$-module of highest weight $\La$ and by $B(\La)$ its crystal.
Let $u_\La$ be the highest weight vector of $V_q(\La)$.
For any $w\in \weyl$,
$\dim(V_q(\La)_{w\La})=1$ and
let $u_{w\La}$ be a generator of  $V_q(\La)_{w\La}$, normalized suitably
(see \cite[\S\; 12]{K95}).
 The vector  $u_{w\La}$ is called the \emph{extremal vector} of $V_q(\La)$ of weight $w\La$.  We also denote by $V(\La)$ the highest weight $U(\g)$-module with highest weight $\La$.

\Lemma\label{lem:bru}
Let $w,v\in\weyl$.
Then the following conditions are equivalent:
\bna
\item $v\le w$,
\item   $u_{w\La}\in U^-(\g)u_{v\Lambda}$
for any $\Lambda\in\pwtl$,
\item   $u_{v\La}\in U^+(\g)u_{w\Lambda}$
for any $\Lambda\in\pwtl$,

\ee
\enlemma
\Proof
 We denote by $( - , - )$ the Shapovalov bilinear form on the highest weight module $V(\La)$, and by $\sigma: U(\g) \buildrel \sim \over \longrightarrow U(\g)$ the anti-involution sending $e_i \mapsto f_i$ and $f_i \mapsto e_i$ for $i\in I$.
Since the Shapovalov  form is non-degenerate and $ ( u x, y ) =  (  x, \sigma(u)y ) $ for any $x,y\in V(\Lambda)$ and $u\in U(\g)$, it is easy to see that  (b) and (c) are equivalent.

Since it was shown in \cite[Corollary 3.2.2]{K93A} that (a) implies (c), we will focus on proving that  (c) implies (a). 
Let us show it  by induction on $\ell(v) + \ell(w)$.
Since it is trivial if $\ell(w)=0$, let us assume that $\ell(w)>0$.
Take $i\in I$ such that $s_i w< w$.
Let $\gp_i$ be the Lie subalgebra of $\g$ generated by the Borel $\gb$ and $f_i$.
Then $f_i u_{w\La}=0$ and $U^+(\g)u_{w\Lambda} = U(\gp_i)u_{w\Lambda}$. 

Assume first that $s_i v < v$. Since $U(\gp_i)$ is invariant by 
$s_i={\rm exp}(f_i){\rm exp}(-e_i){\rm exp}(f_i)$, 
we have
\eqn
s_i\bl U^+(\g)u_{w\Lambda}\br
&&=s_i\bl U(\gp_i)u_{w\Lambda}\br=
U(\gp_i)u_{w\Lambda}\\
&&=U^+(\g)\C[f_i] u_{w\Lambda}
=U^+(\g)u_{w\Lambda}
\ni u_{v\La},
\eneqn
which says that 
$ u_{s_iv\La} \in  U^+(\g)u_{ w\Lambda} $.
Then the induction hypothesis implies
$s_iv\le w$. As $s_i v< v$ and $ s_i w < w$, we conclude that $v\le w$.

\smallskip
We now assume  that $s_i v > v$. Since  $e_iu_{v\La}=0$ and 
$$
u_{v\La} \in U^+(\g)u_{w\La} = U(\gp_i) u_{s_iw\La} = \C[f_i] 
(U^+(\g)u_{s_i w\La}\cap\Ker e_i),
$$
we have $ u_{v\La} \in  U^+(\g)u_{s_i w\La}\cap\Ker e_i $. By the induction hypothesis, we have 
$ v \le s_i w$, which implies $v \le w$.
\QED

\subsection{Quiver Hecke algebras} \label{Sec: QHA}

Let $\bR$ be a field.
A {\em quiver Hecke datum} associated with $(\cmC,\wlP,\Pi,\wlP^\vee,\Pi^\vee,(\cdot,\cdot)) $ is a family  $\{ \qQ_{i,j}(u,v)\}_{i,j\in I} \subset \bR[u,v]$ of the form
\begin{align*}
\qQ_{i,j}(u,v) =\bc
                   \sum\limits
_{p(\alpha_i , \alpha_i) + q(\alpha_j , \alpha_j) = -2(\alpha_i , \alpha_j) } t_{i,j;p,q} u^pv^q &
\text{if $i \ne j$,}\\[3ex]
0 & \text{if $i=j$,}
\ec
\end{align*}
such that $t_{i,j;-a_{ij},0} \in  \bR^{\times}$ and
$\qQ_{i,j}(u,v)= \qQ_{j,i}(v,u)$ for any $i,j\in I.$ For $\beta\in \rlQ_+$ with $\height{\beta}=n$, define
$$
I^\beta\seteq  \Bigl\{\nu=(\nu_1, \ldots, \nu_n ) \in I^n \bigm| \sum_{k=1}^n\alpha_{\nu_k} = \beta \Bigr\}.
$$
The {\em quiver Hecke algebra} $R(\beta)$ associated with $\st{\qQ_{i,j}(u,v)}_{i,j\in I}$
is the $\Z$-graded algebra over $\bR$ generated by
$$
\{e(\nu) \mid \nu \in I^\beta \}, \quad  \{x_k \mid 1 \le k \le n \},
 \quad \{\tau_t \mid 1 \le t \le n-1 \}
$$
satisfying certain defining relations determined by $\st{\qQ_{i,j}(u,v)}_{i,j\in I}$ (see \cite{KL1, KL2, R08} and see also \cite{KKKO18, KKOP18} for details).
The $\Z$-grading of $R(\beta)$ is defined as
\begin{align*}
\deg(e(\nu))=0, \quad \deg(x_k e(\nu))= ( \alpha_{\nu_k} ,\alpha_{\nu_k}), \quad  \deg(\tau_t e(\nu))= -(\alpha_{\nu_{t}} , \alpha_{\nu_{t+1}}).
\end{align*}
The quiver Hecke algebra $R$ is \emph{symmetric} if $\cmC$ is symmetric and $\qQ_{i,j}(u,v)$ is a polynomial in $u-v$ for all $i,j\in I$. 

For an $R(\beta)$-module $M$, we set $\wt(M) \seteq  -\beta \in -\rlQ_+$.

We set $R(\beta)\gmod$ to be the category of finite-dimensional graded $R(\beta)$-modules, and write $R\gmod \seteq  \bigoplus_{\beta \in \rlQ_+} R(\beta)\gmod$.
 {\em Throughout the paper, an $R(\beta)$-module means a graded $R(\beta)$-module and a homomorphism means a graded homomorphism.}

For $\beta,\gamma \in \rlQ_+$, we define $e(\beta,\gamma)\seteq \sum_{\mu \in I^{\beta}, \, \nu\in I^{\gamma}} e(\mu ,\nu)\in R(\beta+\gamma)$, where 
$e(\mu ,\nu)$ denotes the concatenation $e(\mu_1,\ldots,\mu_m,\nu_1,\ldots,\nu_n)$ for $\mu=(\mu_1,\ldots,\mu_m)$ and $\nu = (\nu_1,\ldots,\nu_n)$.
 By using the algebra embedding
$R(\beta)\tens R(\gamma)\subset e(\beta,\gamma) R(\beta+\gamma)e(\beta,\gamma)$, 
we define 
\eqn
M\conv N\seteq R(\beta+\gamma)e(\beta,\gamma) \otimes_{R(\beta)\tens R(\gamma)} (M\tens N)
\eneqn
for an $R(\beta)$-module $M$ and an $R(\gamma)$-module $N$, and call it
the \emph{convolution product} of $M$ and $N$. 
The abelian category $R\gmod$ is a monoidal category with
the
convolution product as its tensor product. 

Let $q$ denote the {\em grading shift functor} given by $(qM)_k = M_{k-1}$ for a graded module $M = \bigoplus_{k \in \Z} M_k $.
For graded $R(\beta)$-modules $M$ and $N $, we write $\Hom_{R}(M,N)$ for the space of degree preserving homomorphisms, and define
\[
\HOM_{R}( M,N ) \seteq \bigoplus_{k \in \Z} \HOM_{R}( M,N )_k\qt{
where $\HOM_{R}( M,N )_k\seteq\Hom_{R}(q^{k}M, N)$.}
\]
For a homomorphism  $f \in \HOM_{R}(M, N)_a$, 
we set $ \deg(f) \seteq a$, and simply write $f \cl M \to N$ if it is not necessary to specify the grading. 
We also  write  $M\simeq N$ to indicate that $M$ and $N$ are isomorphic up to a grading shift if no confusion arises.

\medskip
 For $n \in \Z_{>0}$ and a simple module $M$, we write $M^{\circ n} \seteq  M \conv M^{\circ n-1}$, where we understand $M^{\circ 0} $ is the trivial module $ \one$.
We say that $M$ and $N$ \emph{ commute } if $M \conv N $ is isomorphic to $ N \conv M$, and say that $M$ and $N$ \emph{ strongly commute } if $M\conv N$ is simple. 
Note that, if $M$ and $N$  strongly commute, then they commute.
A simple $R(\beta)$-module $M$ is \emph{real} if $M^{\circ 2}$ is simple.

For $R$-modules $X$ and $Y$, we set 
$$X\hconv Y \seteq  \hd(X\conv Y) \qtq X\sconv Y \seteq  \soc(X\conv Y)$$ 
where $\hd(M)$ (resp.\ $\soc(M)$) denotes the head (resp.\ socle) of a module $M$.

For $M \in R\gmod$, the dual $R(\beta)$-module $M^* \seteq \Hom_\bR(M,\bR)$
is defined by $(r \cdot f)(u)= f(\psi(r)u)$ for $r \in R(\beta)$,  
$u \in M$,  $f \in M^*$, 
where $\psi$ is the $\bR$-algebra anti-involution on $R(\beta)$ fixing the generators.
 A simple $R(\beta)$-module $M$ is called \emph{self-dual}
if there exists an isomorphism between $M^* $ and $ M$ preserving the grading.
For any simple module $M$, there exists a unique $n\in\Z$ such that $q^nM$ is self-dual.

For $i\in I$ and $M \in R(\beta)\gmod$, we set
\begin{align*}
\E_iM &=e(i,*)M, \qquad \ep_i(M) = \max\{ n \ge 0 \mid \E_i^n M\ne0 \}, \\ 
\E^*_iM &=e(*,i)M,\qquad \ep^*_i(M) = \max\{ n \ge 0 \mid \E_i^{*n} M\ne0 \},
\end{align*}
where $e(i,*)=e(\al_i,\beta-\al_i)$ and $e(*,i)=e(\beta-\al_i,\al_i)$. 
We regard $\E_i$ and $\E_i^*$ as exact functors from $R(\beta)\gmod$ to $R(\beta-\al_i)\gmod$.
We denote by $\ang{i}$  a unique simple $R(\al_i)$-module of degree $0$. 
For any simple module $X \in R(\beta)\gmod$, we define  
\begin{align*}
\tF_i(X) &\seteq  \ang{i}\hconv X, \qquad   \qquad \tE_i(X) \seteq  \soc( \E_iX) \simeq\hd(\E_iX), \\
\tFs_i(X) &\seteq  X \hconv \ang{i}, \qquad   \qquad \tEs_i(X) \seteq  \soc( \E^*_iX) \simeq\hd(\E^*_iX).
\end{align*}
Note that $\tF_i(X)$ and $\tFs_i(X)$  (resp.\ $\tE_i(X)$ and $\tEs_i(X)$) are simple modules in $R(\beta+\al_i)\gmod$ (resp.\ $R(\beta-\al_i)\gmod$). We write 
$$
\tEm_i (X) \seteq  \tE_i^{\ep_i(X)}(X) {\qtq} \tEsm_i (X) \seteq  (\tEs_i)^{\ep^*_i(X)}(X).
$$
We remark that the operators 
$\tF_i$, $\tE_i$, $\tFs_i$ and $\tEs_i$ correspond to the crystal operators $\tf_i$, $\te_i$, $\tf^*_i$ and $\te^*_i$ under the crystal-theoretic categorification (\cite{LV11}).

\medskip
If $(i_1, i_2, \ldots, i_r) \in I^r$ satisfies the conditions,
\bna \item $(\al_{i_{k}}, \al_{i_{k+1}}) < 0$ for any $k\in [1,r-1]$,
\item $i_k\not=i_{k+2}$ for any $k\in[1,r-2]$ such that
$\ang{h_{i_k},\al_{i_{k+1}}}=-1$,
\ee
 then we denote by
$\ang{i_1, i_2, \ldots, i_r}$ the $1$-dimensional $R$-module,
where $ x_t, \tau_l$ act as $0$ on
$\ang{i_1, i_2, \ldots, i_r}$ and $e(i_1,i_2,\ldots, i_r)\ang{i_1, i_2, \ldots, i_r}
=\ang{i_1, i_2, \ldots, i_r}$.

When $(\al_i, \al_j) < 0$,  any simple module in $R(\al_i+\al_j)\gmod$
is isomorphic to either $\ang{i,j}$ or $\ang{j,i}$. 
For $n\in \Z_{\ge0}$, we simply write 
$\ang{i^n} \seteq  q_i^{n(n-1)/2} \ang{i}^{\circ n} $.
It is a unique self-dual simple $R(n\al_i)$ module. 

\medskip

A pair $(M,N)$ is \emph{$\La$-definable} if $\HOM (M\conv N,N\conv M)=\cor \rmat{}$ for some non-zero homomorphism $\rmat{}$, which is called  the \emph{R-matrix}. 
In this case,
define 
\begin{align*}
\La(M,N)\seteq  \deg(\rmat{})
\end{align*} 
and
$$
\tLa(M,N)\seteq  \dfrac{1}{2}  \bl \La(M,N)+ (\wt(M),\wt(N)) \br\in\Z .
$$
If a simple module $M$ is \emph{affreal}, which means that $M$ is real and has an \emph{affinization}, then $(M,N)$ and $(N,M)$ are $\La$-definable for any simple $N$ (\cite{KP18}).
When $R$ is symmetric, any simple has an affinization (\cite[Section 2.2]{KKKO18}). 
Note that,  when either $M$ or $N$ is an affreal simple module, 
$M$ and $N$ commute if and only if $M$ and $N$ strongly commute (\cite{KKKO15}).   
If  both of $(M,N)$ and $(N,M)$  are $\La$-definable, then  we say that $(M,N)$ is {\em $\de$-definable} and define
$$
\de(M,N) \seteq \dfrac{1}{2}(\La(M,N)+\La(N,M)).
$$
For $i\in I$ and a simple module $M$, we define  
\begin{align} \label{Eq: def of di}
\de_i(M) \seteq   \ep_i(M) + \ep^*_i(M) + \ang{h_i, \wt(M)}.
\end{align}
Note that $ \frac{(\al_i, \al_i)}{2} \de_i(M) = \de(\ang{i}, M)$ (\cite[Lemma 2.15]{KKOP23A}).

A sequence $ \underline{L} = (L_1,\ldots,L_r)$ of simple modules is \emph{almost affreal} if all $L_i$ $(1 \le i \le r)$ are  affreal  except for at most one.
 A sequence $\underline{L} = (L_1,\ldots,L_r)$ is called $\La$-definable
if $(L_i,L_j)$ is $\La$-definable for any $i,j$ such that
$1\le i<j\le r$.  An almost \afr sequence is $\La$-definable.
 An almost \afr  sequence $\underline{L} = (L_1,\ldots,L_r)$
is said to be a \emph{normal sequence} if the composition
\begin{align*}
\rmat{\underline{L} }  \seteq \prod_{1 \le i < k \le r} \rmat{L_i,L_k }   =& (\rmat{L_{r-1},L_r}) \circ \cdots (\rmat{L_2,L_r}\circ \cdots \circ \rmat{L_2,L_3}) \circ (\rmat{L_1,L_r} \circ \cdots \circ \rmat{L_1,L_2}) \\
& \colon \; q^{\sum_{1 \le i < k \le r} \La(L_i,L_k)} L_1 \conv \cdots \conv L_r \To L_r \conv \cdots L_1 
\end{align*}
does not vanish. Note that, if $(L,M,N)$ is normal, then the image of $\rmat{L,M,N}$ is simple and coincides with $\hd(L \conv M \conv N)$ up to grading shifts and 
\bna
\item $\La(L,M\hconv N) =\La(L,M)+\La(L,N)$, 
\item $\La(L\hconv M, N) =\La(L,N)+\La(M,N)$
\ee
(see \cite[Section 2.3]{KK19} for details).

We say that a simple $S$ is a {\em \cfac} of a simple $M$ if there exists a simple $X$ such that $M\simeq S\conv X$. 
A simple $S$ is called {\em prime} if $S\simeq X \conv Y$ for $X,Y \in R\gmod$, then $X\simeq \one$ or $Y\simeq \one$.

\Lemma\label{lem:ifac}
Let $i\in I$ and let $M$ be a simple module in $R\gmod$.
Then the following conditions are equivalent:
\bna
\item $\ang{i}$ is a \cfac of $M$,
\item $\ep_i(M)>0$ and $\de_i(\tE_iM)=0$.
\item $\ep^*_i(M)>0$ and $\de_i(\tEs_iM)=0$.
\item
$\ep_i^*\bl\tEm_i M\br<\ep^*_i(M)$,

\item
$\ep_i\bl\tEsm_i M\br<\ep_i(M)$,
\ee
\enlemma
\Proof
Set $N\seteq \tEm_i \tEsm_i (M) \simeq \tEsm_i \tEm_i (M)$. Applying \cite[Proposition 2.16]{KKOP23A} to $N$, one can show that 
 the conditions are equivalent to $\de_i(N)<\ep_i(M)+\ep^*_i(M)$.
\QED

\subsection{Categories $\Cw$ and $\Cwv$} \label{Sec: Cwv} \
In this section, we briefly recall the  subcategories $\Cw$ and $\Cwv$ of $R\gmod$ (\cite{KKKO18, KKOP18}).

For an $R(\beta)$-module $M$, we define
\begin{align*}
\gW(M) & \seteq \{ \gamma \in \rl_+\cap (\beta-\rlQ_+) \mid e(\gamma,\beta-\gamma)M \ne 0 \}, \\
\sgW(M) & \seteq \{ \gamma \in \rl_+ \cap (\beta-\rlQ_+) \mid e(\beta-\gamma,\gamma)M \ne 0 \}.
\end{align*}
It is easy to see that 
$\gW(M \conv N) = \gW(M) + \gW( N)$ and $\sgW(M \conv N) = \sgW(M) + \sgW( N)$.

By \cite{TW16}, we have
\begin{align*}
\gW(M)\subset\chl(\gW(M)\cap\prt)\qt{for any simple module $M$}.
\end{align*}
Here, for $A\subset \R\tens \rtl$, $\chl(A)$ denotes the smallest
convex subset of $\R\tens \rtl$ which is stable by the multiplication
by $\R_{>0}$ and contains $A\cup\st{0}$.

An ordered pair $(M, N)$ of $R$-modules is called \emph{unmixed}  if
$$ 
\sgW(M) \cap  \gW(N) \subset \{0\}.
$$ 
Note that  if $(M,N)$ is unmixed, then
it is $\La$-definable and $\tLa(M,N)=0$.
An almost \afr triple $(M,X,N)$ of simple modules $M,X,N$
is normal if $(M,N)$ is unmixed 
(see \cite[Proposition 2.12]{KKOP18} and \cite[Corollary 2.13]{KKOP23A}).

For $w\in \weyl$, we define $\catC_{w}$ to be the  full subcategory of $R\gmod$ consisting of modules $M$ satisfying
\begin{align*} 
\gW(M) \subset \prtl\cap w \nrtl.
\end{align*}
In a similar manner, for $v\in \weyl$, we define $\catC_{*,v}$ to be the  full subcategory of $R\gmod$ consisting of modules $N$ satisfying
\begin{align*}
\sgW(N) \subset \prtl\cap v\prtl.
\end{align*}
We now set
$\Cwv $ to be the full subcategory of $R\gmod$ whose objects are contained in both of $ \catC_w$ and $\catC_{*,v}$, i.e., 
\eq
\Cwv = \Cw \cap \catC_{*, v}.
\label{eq:Cwv}
\eneq
We denote by $K(\Cw)$ and $K(\Cwv)$ the Grothendieck rings of 
$\Cw$ and $\Cwv$, respectively.  They are $\Z[q^{\pm1}]$-algebras
where the indeterminate $q$ is induced by the grading shift functor. 
The category $\Cw$ gives a categorification of the quantum unipotent coordinate ring $A_q(\n(w))$ (\cite{KL1, KL2, R08, KKKO18}), i.e., 
\begin{align*} 
K(\Cw) \simeq A_q(\n(w)), 
\end{align*} 
and $\Cwv$ categorifies the quantization of the coordinate ring of the open 
\emph{Richardson variety} $\mathcal{R}_{w,v}$ after localization (\cite{KKOP18}). 
We sometimes write  $K(\Cw)|_{q=1}$ and $K(\Cwv)|_{q=1}$ when we ignore the grading shift, i.e., $$K(\Cw)|_{q=1}\seteq K(\Cw)/(q-1)K(\Cw) \qtq K(\Cwv)|_{q=1}\seteq K(\Cwv)/(q-1)K(\Cwv).$$

\subsection{Reflection functors} \label{Sec: RF}
For each $i\in I$, we set $(R\gmod){}_i$ (resp.\ ${}_i (R\gmod)$) to be the full subcategory of $R\gmod$ consisting of modules killed by $\E_i$ 
(resp.\ $\E_i^*$).
In \cite{KKOP26}, we define
the monoidal equivalence (cf.\ \cite{Kato20})
$$\F_i\cl (R\gmod)_i\isoto{}_i(R\gmod).$$
called the \emph{reflection functor} 
 which categorifies Lusztig's braid symmetries (\cite{LusztigBook}). 
Since $\F_i$ categorifies 
the \emph{Saito crystal reflection} (\cite{Saito94}) at the crystal level, 
$\F_i(M)$ for a simple module $ M \in (R\gmod){}_i $ can be computed as follows:
\begin{align} \label{Eq: F_i in crystal}
\F_i(M) \simeq   \tF_i^{\ph^*_i(M)} {\tE_i}^{* \ep_i^*(M)}  (M).
\end{align}

The following proposition is proved (at least implicitly) in
\cite[\S\;2.3]{TW16}.

\Prop \label{prop:Refwv}
Let $i\in I$.
\bnum
\item For any $w\in\weyl$ such that $s_iw<w$, $\F_i$ induces
 an equivalence of monoidal categories
$$\F_i\cl\Cw[s_iw]\to\Cw\cap{}_i(R\gmod)$$ 

\item For any $v\in\weyl$ such that $v<s_iv$, $\F_i$ induces
 an equivalence of monoidal categories
$$\F_i\cl \Cwv[*,v]\cap(R\gmod){}_i\to\Cwv[*,s_iv].$$
\ee
\enprop
\Proof
Recall first a result of \cite[\S\;2.3]{TW16}.
For convex preorder, we refer the reader to \cite{TW16} (see also \cite[Section 1.3]{KKOP18}).

Let $\prtm\seteq\st{\al\in\prt\mid \Q_{>0}\,\al\cap\prt=\Z_{>0}\,\al\cap\prt}$ be the set of minimal positive roots.
For a convex order $\preceq$ on $\prtm$, we extend it to
the convex preorder on $\prt$ such that
$m\al\preceq n\beta$ for any
$m,n\in\Z_{>0}$ and $\al,\beta\in\prtm$ with $\al\preceq\beta$.

For $i\in I$ and a convex order $\preceq$ on $\prtm$
such that $\al_i$ is the smallest element of $\prtm$ with respect to $\preceq$,
let
$\preceq^{s_i}$ be the convex order on $\prtm$ defined by:
\eqn
&&\parbox{60ex}{
\bna
\item
$\al_i$  is the largest element of $\prtm$  with respect to  $\preceq^{s_i}$,
\item
for $\al,\beta\in\prtm\setminus\st{\al_i}$,
$\al\preceq^{s_i}\beta$ if and only if $s_i\al\preceq s_i\beta$.
\ee}
\eneqn
For $\beta\in\prt$, we denote
by $(R\gmod)^{\preceq}_\beta$ the full subcategory
of $R\gmod$ consisting of $M\in R\gmod$ such that
$\gW(M)\cap\prt\subset \st{\al\in\prt\mid \al\preceq\beta}$.
Similarly we denote
by $(R\gmod)^*_\beta{}^{\preceq}$ the full subcategory
of $R\gmod$ consisting of $M\in R\gmod$ such that
$\sgW(M)\cap\st{\al\in\prt\mid \al\preceq\beta}=\emptyset$.

Now, take a convex order $\preceq$ on $\prtm$
such that $\al_i$ is the smallest element of $\prtm$ with respect to $\preceq$.
Then by \cite[\S\;2.3]{TW16}, we have the following:

\noi
for any $\beta\in\prt\setminus\st{\al_i}$, we have
\bna
\item 
the reflection functor $\F_i$ gives an equivalence
$$(R\gmod)^{\preceq^{s_i}}_\beta
\isoto {}_iR\gmod\cap (R\gmod)^{\preceq}_{\beta},$$
\item
the reflection functor $\F_i$ gives an equivalence
$$R\gmod_i\cap(R\gmod)^*_\beta{}^{\preceq^{s_i}}
\isoto (R\gmod)^*_{\beta}{}^{\preceq}.$$
\ee

\mnoi
Now we shall prove (i) and (ii).

\noindent
(i) Take a convex order $\preceq$ on $\prtm$ such that
$\al_i$ is the smallest and $\prt\cap w\nrt\preceq \prtm\cap w\prt$,
and let $\beta$ be the largest element of $\prt\cap w\nrt$.
Then we have
$(R\gmod)^{\preceq^{s_i}}_{s_i\beta}=\Cw[s_iw]$ and 
$(R\gmod)^{\preceq}_\beta=\Cw$. Hence (a) implies (i).

\mnoi
(ii) Take a convex order $\preceq$ on $\prtm$ such that
$\al_i$ is the smallest and $\prt\cap s_iv\nrt\preceq \prtm\cap s_iv\prt$,
and let $\beta$ be the largest element of $\prt\cap s_iv\nrt$.
Then we have
$(R\gmod)^*_{s_i\beta}{}^{\preceq^{s_i}}=\Cw[{*,v}]$
and $(R\gmod)^*_{\beta}{}^{\preceq}=\Cw[{*,s_iv}].$
\QED

\Lemma\label{lem:tEmst}
Let $w\in\weyl$ and let $i\in I$. 
\bnum 
\item 
Let $M$ be a simple module in $R\gmod$.
If $\de_i(M)=0$, then $\tEsm_iM\simeq\F_i(\tEm_iM)$.
\item 
If $s_iw<w$, then $\tEm_i N \in\Cw[s_iw]$ for any simple module $N\in\Cw$.
\ee
\enlemma
\Proof 
(i) We set $\beta \seteq  \wt(M)$ and 
$$
M' \seteq  \tEm_i (M),\quad M'' \seteq  \tEsm_i (M), \qtq M_\circ \seteq  \tEm_i \tEsm_i (M).
$$
We then write $M = \ang{i}^{\circ a} \hconv M' = M'' \hconv \ang{i}^{\circ b}$ where 
$a \seteq  \ep_i(M)$ and $b \seteq  \ep^*_i(M)$. Note that $\tEsm_i M' \simeq M_\circ$ and  
$$ 
\de_i(M) = a + b + \ang{h_i, \beta} =0. 
$$ 
Since $\wt(M_\circ) = \beta+ (a + \ep^*_i(M'))\al_i = \beta+ (b + \ep_i(M''))\al_i  $, we have
$$
a + \ep^*_i(M') = b + \ep_i(M''),
$$
which implies 
$$
\ph_i^*(M') =  \ep^*_i(M') + \ang{h_i, \beta + a\al_i }
= \ep^*_i(M') + a-b = \ep_i(M'').
$$
Hence, we have 
$$
\F_i(\tEm_iM) = \F_i( M') \simeq \tF_i^{\ph^*_i(M)} (M_\circ) = \tF_i^{\ep_i(M'')} (M_\circ) = M''.
$$

(ii) Replacing $N$ with $\tF_i^n N$ ($n\gg0$), we may assume that
$\de_i(N)=0$.  
As $N\in \Cw$, we have $ L\seteq \tEsm_i(N)  \in\Cwv[w,s_i]$, which implies 
$$
\F_i^{-1}(L) \simeq \tEm_iN \in\Cw[s_iw]
$$
by (i) and Proposition \ref{prop:Refwv}.
\QED

\vskip 2em 
 
\subsection{Miscellaneous results}
In this subsection, we show several lemmas and propositions which will be used in the paper.

The following lemma can be obtained from \cite[Lemma 3.2.11 and 3.2.12]{KKKO18} by a straightforward modification to the setting of arbitrary quiver Hecke algebras.

\begin{lem} [{\cite[Lemma 3.2.11 and 3.2.12]{KKKO18}}] \label{lem: Normal sequence generalized}
Let $L,M,N \in R\gmod$ be simple modules and assume that $L$ is affreal. 
\bnum
\item If $L$ and $M$ commute, then 
$$
\La(L,M)+ \La(L,N) = \La(L,S) \quad \text{ for any simple quotient $S$ of $M \conv N$.}
$$
\item 
If $L$ and $N$ commute, then 
$$
\La(M,L)+ \La(N,L) = \La(S,L) \quad \text{ for any simple quotient $S$ of $M \conv N$.}
$$
\ee
\end{lem}

We then have the following proposition.

\begin{prop} \label{prop: Cao general}
Let $L,M,N \in R\gmod$ be simple modules and assume that $L$ is affreal.
\bnum
\item 
There exists a simple subquotient $S$ of $M \conv N$ such that
$$
\La(L,M) + \La(L,N) = \La(L,S). 
$$
Similarly, there exists a simple subquotient $S'$ of $M \conv N$ such that
$$
\La(M,L) + \La(N,L) = \La(S',L). 
$$ 
\item  
We have
\begin{align*}
\La(L,M)+\La(L,N) &= \max \{ \La(L,S) \mid \text{$S$ is a composition factor of $M \conv N$}\}, \\
\La(M, L)+\La(N,L) &= \max \{ \La(S,L) \mid \text{$S$ is a composition factor of $M \conv N$}\}. 
\end{align*}
\ee
\end{prop}
\begin{proof}  
(i) We focus on proving the first identity since the second can be shown in the same manner. 

Let us take a positive integer $n$ such that $X \seteq L^{\circ n} \hconv M$ commutes with $L$ (see \cite[Corollary 3.18]{KKOP21A}). Take a simple quotient $Y$ of $X \conv N$. Then Lemma~\ref{lem: Normal sequence generalized} implies that 
$$
\La(L,M)+\La(L,N) = \La(L,X) + \La(L,N) = \La(L,Y).
$$  
Since $Y$ is a simple quotient of $L^{\circ n} \conv M \conv N$, there exists a simple subquotient 
$S$ of $M \conv N$ such that $Y \simeq L^{\circ n} \hconv S$. Indeed, take a minimal submodule $Z$
of $M \conv N$ such that the composition 
$$
L^{\circ n} \conv Z \to  L^{\circ n} \conv M \conv N \twoheadrightarrow Y
$$
does not vanish. Take a simple quotient $S$ of $Z$.  
Then
$$
L^{\circ n} \conv Z \twoheadrightarrow Y \text{ factors through } L^{\circ n} \conv S.  
$$
Hence we have 
\[
\La(L,S) = \La(L,L^{\circ n} \hconv S) =\La(L,Y) = \La(L,M)+ \La(L,N). 
\]

\snoi
(ii) It follows from (i) and \cite[Proposition 3.2.10]{KKKO18}.
\end{proof}

\begin{prop} \label{prop: XYZ}
Let $(X, Y, Z)$ be an almost affreal normal sequence of simple modules.
 \bnum
\item  We assume that  
 \bna
 \item  $X$ is affreal, 
 \item  $X$ commutes with $Y \hconv Z$, 
 \ee
 Then $X$ commutes with $Y$.
 
 \item  We assume that  
 \bna
 \item  $Z$ is affreal, 
 \item  $Z$ commutes with $X \hconv Y$, 
 \ee
 Then $Z$ commutes with $Y$.
 \ee
\end{prop}
\begin{proof}
Since (ii) can be proved in the same manner, we only prove (i). 
We may assume that $Y$ and $Z$ do not commute since the assertion is obvious otherwise. 
From the following commutative diagram
$$
\xymatrix{
X \conv Y \conv Z   \ar@{->>}[d] \ar[r]^{\rmat{X,Y}} &  Y \conv X \conv Z  \ar[r]^{\rmat{X,Z}} & Y \conv Z \conv X  \ar@{->>}[d]  \\
 X \conv (Y \hconv Z) \ar[rr]^{\sim} &&  (Y \hconv Z) \conv X,
}
$$	
we have 
$$
\Ker ( X \conv Y \buildrel {\rmat{X,Y}} \over \longrightarrow Y \conv X ) \conv Z \subset  X \conv \Ker ( Y \conv Z \twoheadrightarrow Y \hconv Z ).
$$
Thus the quasi-rigidity (see \cite[Lemma 3.1]{KKKO15} and see also \cite[Section 6.3]{KKOP24D}) implies that there exists $S \subset Y$  such that
\begin{align*}
\Ker ( X \conv Y \buildrel {\rmat{X,Y}} \over \longrightarrow Y \conv X )   \subset X \conv S, \qquad  S \conv Z \subset     \Ker ( Y \conv Z \twoheadrightarrow Y \hconv Z ).
\end{align*}
Since $ \Ker ( Y \conv Z \twoheadrightarrow Y \hconv Z ) \ne  Y\conv Z$, we have $S \ne Y$, which implies that $S=0$.
Hence the $R$-matrix $X \conv Y \buildrel {\rmat{X,Y}} \over \longrightarrow Y \conv X $ is injective, which means that it is an isomorphism (see \cite[Section 3]{KKKO15}). 
\end{proof}

\begin{coro} \label{cor: MNL12}
Let $M$ be an affreal simple module and $N$ a simple module. Let $ \{ L_k \}_{k=1,2} $ be a commuting family of affreal simple modules.	
\bnum
\item  If $M$ commutes with $N$ and $\hd( N  \conv L_1 \conv L_2 )$, then $M$ commutes with $N \hconv L_1$.
\item If $M$ commutes with $N$ and  $\hd(L_2\conv L_1 \conv N)$, then $M$ commutes with $L_1 \hconv N$.
\ee
\end{coro}
\begin{proof}
We shall prove only (i) because (ii) can be shown in the same manner. Since $M$ and $N$ commute, $(M, N, L_1 \conv L_2)$ is a normal sequence. Hence $(M, N, L_1, L_2)$ is also a normal sequence, and so is $(M, N \hconv L_1, L_2)$.
Therefore, Proposition \ref{prop: XYZ} implies the desired result.
\end{proof}

Applying Corollary \ref{cor: MNL12} to the setting $L_1 \seteq  \ang{i^{k}}$ and $L_2 \seteq  \ang{i^{n-k}}$, we have the following.

\begin{coro} 
Let $i\in I$, let $M$ be an affreal simple module and $N$ a simple module. 
Suppose that $M$ and $N$ commute.
\bnum
\item 
If $n \in \Z_{\ge0}$ and $M$ commutes with $\tF_i^n (N)$, then $M$ commutes with $\tF_i^k (N)$ for any $ k \in [1,n]$.

\item 
If $n \in \Z_{\ge0}$ and $M$ commutes with $\tF_i^{* n} (N)$, then $M$ commutes with $\tF_i^{*k} (N)$ for any $ k \in [1,n]$.
\ee	
\end{coro}	

The following lemma is an analogue of \cite[Lemma 2.23]{KKOP24A} 
for quiver Hecke algebras. 
We omit the proof since it follows from the same argument
 as in \cite[Lemma 2.23]{KKOP24A} once the existence of affinizations 
is established.
Note that the commutativity of \textcircled{c} in the diagram of 
\cite[Proof of Lemma 2.23]{KKOP24A} follows from the affinization of $M\hconv N$.

\begin{lem} [{cf.\  \cite[Lemma 2.23]{KKOP24A}}] \label{Lem: MhconvN real}
Let $X$ and $Y$ be affreal simple modules such that $X\htens Y$ admits an affinization.
If $X\htens Y$ commutes with $X$, then $X\htens Y$ is an affreal simple module.
\end{lem}

\subsection{Laurent families} \label{Sec: prep}
In this section, we recall the notion of Laurent families introduced in 
\cite{KKOP24B} and study their properties that we need in this paper. 
\subsubsection{Definition of Laurent families}
Let $\calC$ be a full subcategory of $R\gmod$ such that $\calC$ has $\one$ and is stable under taking convolution products, subquotients, extensions and grading shifts, and  let $K$ be an index set.  A family $\calM = \{ M_k \}_{k\in K} \subset \calC$ of \afr simple modules is called a \emph{commuting family} if $M_k$ is an \afr simple module for any $k\in K$ and 
$$
M_i \conv M_j \simeq M_j \conv M_i \qquad \text{ (up to a grading shift) for any $i,j\in K$.}
$$
For $\bfa = (a_k)_{k\in K} \in \Z_{\ge0}^{\oplus K}$, the module $\calM (\bfa)$ is defined as follows:
\begin{equation} \label{Eq:M(a)}
\begin{aligned}
&\text{$\calM (\bfa)$ is self-dual,}\\
&\text{$\calM (\bfa) \simeq \conv[{k\in K}]  \left( M_k^{\circ a_k} \right)$ up to a grading shift.}
\end{aligned}
\end{equation}
Then we have
\eq
\calM(\bfa)\conv\calM(\bfb)\simeq q^{\la(\bfa,\bfb)}\calM(\bfa+\bfb),
\label{eq:coM}
\eneq
where $\la(\bfa,\bfb)=-\sum_{j,k\in K}a_jb_k\tLa(M_j,M_k)$
with $\bfb = (b_k)_{k\in K}$.

\Def[{\cite{KKOP24B}}] 
Let $\calM = \{ M_k \}_{k\in K}$ be a commuting family of
\afr simple modules.
\bnum
\item
A simple module $X$ is called an {\em $\calM$-monomial} if $X$ is isomorphic to $\calM(\bfa)$ for some $\bfa\in \Z_{\ge0}^{\oplus K}$.
\item \label{Eq:qLf (a)}
$\calM$ is \emph{independent} if $\bfa = \bfb$ whenever $ \calM (\bfa) \simeq \calM (\bfb)$.
\item
$\calM $ is a \emph{quasi-Laurent family} in $\calC$ if it is independent and satisfies the following condition:
\eq \label{Eq:qLf (b)}
&&\hs{5ex}\parbox{\textwidth-10ex}%
{if a simple module $X \in \calC$ commutes with $M_k$ for any $k\in K$, then there exist $\bfa, \bfb \in \Z_{\ge0}^{\oplus K}$ such that $X \conv \calM(\bfa) \simeq \calM(\bfb)$.  }
\eneq

\item
$\calM$ is  a \emph{Laurent family}
if it is independent and satisfies the following condition \eqref{Eq:qLf (c)} instead of \eqref{Eq:qLf (b)}:
\eq
\label{Eq:qLf (c)}
&&\hs{5ex}\parbox{\textwidth-10ex}%
{any simple module $X \in \calC$ which commutes with $M_k$ for any $k\in K$
is an $\calM$-monomial.
}\eneq
\ee
\edf

\subsubsection{$g$-vectors}
Let $\calM=\st{M_k\mid k\in K}$ be a quasi-Laurent family 
in the category $\shc$.

Let $\La_{\calM} \seteq  (\La(M_i, M_j))_{i,j\in K}$,
i.e., $ q^{\La(M_i, M_j)}[M_i][M_j] = [M_j][M_i]$, 
and let $\T_\calM$ be the quantum torus generated by $\st{[M_k]}_{k\in K}$. 

We extend the definition of $[\calM(\bfa)]$ 
for any $\bfa\in\Z^{\oplus K}$ so that
\eqref{eq:coM} holds for any $\bfa,\bfb\in\Z^{\oplus K}$. 
For any simple module $X \in \cC$, there exist
some $\bfa, \bfb, \bfc, \bfd \in \Z_{\ge 0}^{\oplus K}$
such that
$$ 
\calM(\bfa) \hconv X  \simeq \calM(\bfb) \qtq 
X \hconv \calM(\bfc)   \simeq \calM(\bfd).
$$
We define 
\eq
\gL_\calM (X) \seteq  \bfb -\bfa \qtq  \gR_\calM (X) \seteq  \bfd -\bfc.
\eneq
By \cite[Lemma 3.9]{KKOP24B}, $\gL_\calM$ and $\gR_\calM$ are well-defined. 

\medskip
By \cite[Corollary 3.4]{KKOP24B}, for any module $X \in \shc$,  there exist $\bfa,  \bfb_1, \ldots,  \bfb_t  \in \Z_{\ge 0}^{\oplus K}$ such that 
\begin{align} \label{Eq: X Ma}
[X \conv \calM(\bfa)] = \sum_{s=1}^t  q^{  c_s } [\calM(\bfb_s)] \qquad \text{ for some $c_s\in \Z$.}
\end{align}
Then the following lemma is straightforward.
\Lemma
The correspondence $[X]\mapsto \sum_{s=1}^t q^{c_s } 
[\calM(\bfb_s)][\calM(\bfa)]^{-1}\in \T_\calM$
gives a well-defined injective $\Zq$-algebra homomorphism
\begin{align} \label{Eq: qlp}
\vphi_\calM\cl K(\shc)\rightarrowtail \T_\calM.
\end{align}
\enlemma

\Th\label{th:cao}
Let $X, L \in \cC$ be simple modules such that $L$ is \afr,  and write 
$$
\vphi_\calM([X])=\sum_k q^{a_k} [\calM(\bfa_k)] \qquad  \text{ for some $a_k \in \Z$ and $\bfa_k \in \Z^{\oplus K}$.} 
$$
Then we have
\begin{align*}
\La(L,X) &=\max_{k} \{ \La(L, \calM(\bfa_k) ) \} = \max_k \{ \gR_\calM(L) \cdot \La_\calM \cdot \bfa_k^T  \}, \\ 
\La(X,L) & =\max_{k} \{ \La( \calM(\bfa_k), L)\} = \max_k \{  \bfa_k \cdot \La_\calM \cdot  (\gL_\calM(L))^T  \},
\end{align*}
where $v^T$ is the transpose of a vector $v$.
\enth
\begin{proof}
Replacing $X$ with $X  \conv \shm(\bfc)$ for $\bfc\gg0$,
we may assume that $ \bfa_k  \in\Z_{\ge0}^{\oplus K}$.
Applying Proposition \ref{prop: Cao general} to \eqref{Eq: X Ma}, we have 
\begin{align} \label{Eq: LLX LXL}
\La(L,X) =\max_{k} \{ \La(L,  \calM(\bfa_k) ) \} \qtq \La(X,L)  =\max_{k} \{ \La( \calM(\bfa_k),  L)\}.
\end{align}
We write 
$
\calM(\bfb) \hconv L  \simeq \calM(\bfc)
$ for some $\bfb = (b_i)_{i\in K}, \bfc = (c_i)_{i\in K} \in \Z_{\ge0}^{\oplus K} $.
Note that $\gL_\calM(L) = \bfc - \bfb$.
For any $\bfa = (a_i)_{i\in K} \in \Z^{\oplus K}$, we have 
\begin{align*}
 \La( \calM(\bfa), \calM(\bfc) ) &= \sum_{i,j\in K} a_i \La( M_i, M_j ) c_j,\\
\La( \calM(\bfa), \calM(\bfb) \hconv L) &= \La( \calM(\bfa), \calM(\bfb) ) + \La( \calM(\bfa), L) \\
&= 
  \sum_{i,j\in K} a_i \La( M_i, M_j ) b_j + \La( \calM(\bfa), L),
\end{align*}
which implies 
$$
\La( \calM(\bfa), L) = \sum_{i,j\in K} a_i \La( M_i, M_j ) (c_j - b_j) = \bfa \cdot \La_\calM \cdot (\gL_\calM(L))^T.  
$$
Hence, we have $\La(X,L) = \max_k \{  \bfa_k \cdot \La_\calM \cdot  (\gL_\calM(L))^T  \}$ by \eqref{Eq: LLX LXL}. 

The case for $\gR_\calM(L)$ can be proved in the same manner.
\end{proof}

\begin{remark} \label{Rmk: Cao}
Suppose that $\calM$ is a monoidal seed of $\cC$ (see \S\;\ref{subsec:monoidal}). In this case, we have the following.
\bnum
\item $\gL_\calM(X)$ and $\gR_\calM(X)$ are the \emph{extended g-vectors} of a simple module $X$
with respect to $\calM$ introduced in  \cite{KK19,KKOP24B}.
\item Let $\tB$ be the exchange matrix associated with $\calM$.
The formula in Theorem \ref{th:cao} can be written as follows:
\begin{equation} \label{Eq: Cao formula}
\begin{aligned} 
\La(X,L) & = \max_k \{  \bfa_k \cdot \La_\calM \cdot  (\gL_\calM(L))^T  \} \\
&= \gL_\calM(X) \cdot \La_\calM \cdot  (\gL_\calM(L))^T + \max_k \{  (\bfa_k - \gL_\calM(X)) \cdot \La_\calM \cdot  (\gL_\calM(L))^T  \}
\end{aligned}
\end{equation}
Since $\bfa_k$ is larger than $\gL_\calM(X)$ in the \emph{dominance order}, we can write 
$\bfa_k - \gL_\calM(X) = (\tB v)^T$ for some column vector $v$ with non-negative entries (\cite[Lemma 3.6]{KK19}).
Since the pair $(-\La_\calM, \tB)$ is compatible with $d=2$, the formula \eqref{Eq: Cao formula} can be rewritten in terms of  the $F$-polynomial, which coincides with Cao's formula expressed in terms of the \emph{tropical invariant} introduced in \cite[Theorem 5.16]{Cao23}. 
Therefore, Theorem \ref{th:cao} may be regarded as a generalization of
 Cao's formula \cite[Theorem 5.16]{Cao23}  for  
cluster variable modules to the case of real simple modules. 
\ee
\end{remark}

\subsubsection{Stability of Laurent families}

We shall show that the notion of (quasi)-Laurent families is stable under a mutation-like procedure.

\Lemma
Let $\calM=\{ M_k \}_{k\in K}$ and $\calM'=\{ M'_k \}_{k\in K}$
be  commuting families of \afr simple modules.
Let $s\in K$.
Assume that
\bna
\item
$M_k\simeq M'_k$ for $k\in K\setminus\st{s}$,
\item there exist $\shm$-monomials $U$ and $V$ and an exact sequence
$$0\To U\To M_s\conv M'_s\To V\To0.$$
\ee
If $\shm$ is independent, then $\shm'$ is also independent.
\enlemma
\Proof

Let $\bfa=(a_k)_{k\in K},\;\bfb=(b_k)_{k\in K}\in\Z_{\ge0}^{\oplus K}$.
Assuming that
$\shm'(\bfa)=\shm'(\bfb)$, let us show that $\bfa=\bfb$.
We may assume that
$a_s\le b_s$.
Let $\st{\bfe_k\mid k\in K}$ be the natural basis of $\Z_{\ge0}^{\oplus K}$.
Then we have 
\eqn
[\shm'(\bfa)]\cdot [M_s]^{b_s}&&=
[\shm'(\bfa-a_s\bfe_s)]\cdot([U]+[V])^{a_s}\cdot[M_s]^{b_s-a_s}
\qtq{}\\
{}[\shm'(\bfb)]\cdot[M_s]^{b_s}&&=
[\shm'(\bfb-b_s\bfe_s)]\cdot([U]+[V])^{b_s}.
\eneqn
Hence we have
$$[\shm(\bfa+(b_s-2a_s)\bfe_s)]\cdot([U]+[V])^{a_s}
=[\shm(\bfb-b_s\bfe_s)]\cdot([U]+[V])^{b_s}.$$
Then $\bfa=\bfb$ follows from the independence of $\shm$.
\QED

\Prop
Let $\calM=\{ M_k \}_{k\in K}$ and $\calM'=\{ M'_k \}_{k\in K}$
be commuting families of \afr simple modules in $\shc$.
Let $s\in K$.
Assume that
\bna
\item
$M_k\simeq M'_k$ for $k\in K\setminus\st{s}$,
\item there exist $\st{M_k\mid k\in K\setminus\st{s}}$-monomials $U$ and $V$ 
and an exact sequence
$$0\To U\To M_s\conv M'_s\To V\To0.$$
\ee
Then,  $\shm'$ is a quasi-Laurent family in $\shc$
if $\shm$ is a quasi-Laurent family in $\shc$.
\enprop
\Proof
By the preceding lemma, $\shm'$ is independent.

\smallskip
Let $X$ be a simple module in $\shc$ which commutes with all $M'_k$.
Let us show that there exists $\bfa\in \Z_{\ge0}^{\oplus K}$
such that $X\conv\shm'(\bfa)$ is an $\shm'$-monomial.
By \cite[Corollary 3.18]{KKOP21A}, there exists $n\in\Z_{\ge0}$ such that
any simple subquotient of $X\conv M_s^n$ commutes with $M_s$.
Since $M_k$ ($k\in K\setminus\st{s}$) commutes with $X$ and $M_s$,
$M_k$ commutes with any simple subquotient of $X\conv M_s^n$.
Hence
any simple subquotient of $X\conv M_s^n$ commutes with all members of $\shm$.
Since $\shm$ is a quasi-Laurent family, there exists $\bfa\in \Z_{\ge0}^{\oplus K}$ such that
$\shm(\bfa)\conv S$ is an $\shm$-monomial for
any simple subquotient $S$ of $X\conv M_s^n$.
Hence there exist  finitely  many $\bfb(t)\in \Z_{\ge0}^{\oplus K}$ 
such that
$$[X]\cdot[M_s]^n\cdot[\shm(\bfa)]=\sum_t[\shm(\bfb(t))].$$
Thus replacing $\bfa$ with $\bfa+n\bfe_s$, we have
$$[X]\cdot[\shm(\bfa)]=\sum_t[\shm(\bfb(t))]\quad \text{
for some $\bfa\in \Z_{\ge0}^{\oplus K}$ and finitely  many $\bfb(t)\in \Z_{\ge0}^{\oplus K}$ .}$$
Let us take $m\in\Z$ such that
$m\ge a_s$ and $m\ge \bfb(t)_s$ for any $t$, where $\bfb(t)_s$ is the $s$-th component of $\bfb(t)$.
Set $\bar\bfa=\bfa-a_s\bfe_s$ and $\bar\bfb(t)=\bfb(t)-\bfb(t)_s\bfe_s$.
Multiplying $[M'_s]^m$, we obtain
\eqn
[X]\cdot[\shm(\bar\bfa)]\cdot([U]+[V])^{a_s} \cdot [M'_s]^{m-a_s}
=\sum_{t}[\shm(\bar\bfb(t))]\cdot(U+V)^{\bfb(t)_s}\cdot[M'_s]^{m-\bfb(t)_s}
\eneqn
Note that the right-hand side is a sum of $\shm'$-monomials.
Hence a component $[X]\cdot[\shm(\bar\bfa)]\cdot[U]^{a_s}\cdot[M'_s]^{m-a_s}$
of the left-hand side
is an $\shm'$-monomial.
\QED

\Prop\label{prop:Laurentm}
Let $\calM=\{ M_k \}_{k\in K}$ and $\calM'=\{ M'_k \}_{k\in K}$
be commuting families of \afr simple modules in $\shc$.
Let $s\in K$.
Assume that
\bna
\item
$M_k\simeq M'_k$ for ant $k\in K\setminus\st{s}$,
\item there exist $\st{M_k\mid k\in K\setminus\st{s}}$-monomials $U$ and $V$ 
and an exact sequence
$$0\To U\To M_s\conv M'_s\To V\To0$$
such that
$U$ and $V$ have \ncf.
\ee
Then,  $\shm'$ is a Laurent family in $\shc$
if $\shm$ is a Laurent family.
\enprop
\Proof
By the preceding lemma, $\shm'$ is independent.

Let $X$ be a simple module in $\shc$ which commutes with all $M'_k$.
Let us show that $X$ is an $\shm'$-monomial.

By the same argument as the one in the proof of the preceding proposition,
there exists $n\in\Z_{\ge0}$ such that
any simple subquotient of $X\conv M_s^n$ commutes with all members of $\shm$.
Hence there exist  finitely  many $\bfb(t)\in \Z_{\ge0}^{\oplus K}$ 
such that
$$[X]\cdot[M_s]^n=\sum_t[\shm(\bfb(t))].$$
Set $\bar\bfb(t)=\bfb(t)-\bfb(t)_s\bfe_s$.
Multiplying $[M'_s]^n$, we obtain
\eqn
[X]\cdot([U]+[V])^{n}
&&=\sum_{\bfb(t)_s\ge n}[\shm(\bfb(t)-n\bfe_s)]\cdot(U+V)^{n}\\
&&\hs{10ex}+\sum_{\bfb(t)_s<n}[\shm(\bar\bfb(t)]\cdot(U+V)^{\bfb(t)_s}[M'_s]^{n-\bfb(t)_s}.
\eneqn
The right-hand side is a sum of
$\shm$-monomials and $\shm'$-monomials.
Hence $[X]\cdot [U]^n$
is either an $\shm$-monomial or an $\shm'$-monomial.
Note that an $\shm$-monomial which commutes with $M'_s$ is
an $\shm'$-monomial because $\de(M_s,M'_s)>0$.
Since $X\conv U^n$ commutes with $M'_s$,
we conclude that $[X]\cdot [U]^n$ is an $\shm'$-monomial.
Similarly, $[X]\cdot [V]^n$ is also an $\shm'$-monomial.
Set
$U^{\circ n}\simeq\shm(\bfa)$ and $V^{\circ n}\simeq\shm(\bfa')$.
Then $a_s=a'_s=0$ and $a_k a'_k=0$ for any $k\in K$.

Thus there exist $\bfb$ and $\bfb'$ such that
 $$X\conv\shm'(\bfa)\simeq\shm'(\bfb)\qtq
X\conv\shm'(\bfa')\simeq\shm'(\bfb').$$
Hence we have $X\conv\shm'(\bfa)\conv\shm'(\bfa')
\simeq \shm'(\bfa')\conv\shm'(\bfb)\simeq\shm'(\bfa)\conv\shm'(\bfb')$,
which implies that
$$\bfa'+\bfb=\bfa+\bfb'.$$
Now, let us show that
$a_k\le b_k$ for any $k$.
If $a_k=0$, it is trivial.
Otherwise, we have $a'_k=0$,
which implies  $b_k=a_k+b'_k\ge a_k$.

\medskip
Thus we obtain $X\simeq\shm'(\bfb-\bfa)$.
\QED

\subsection{The operators $\K_i$ and $\K^*_i$} \label{Sec: Ki}
In this subsection, we introduce and study new operators $\K_i$ and $\K^*_i$ on simple modules, which will be crucially used to construct seeds of $\Cwv$.

\Def \label{Def: Ki and K*i}
For any $i\in I$ and a simple $M\in R\gmod$, we define
$$
\K_i(M) \seteq \tF_i^{a} \left( \tEm_i (M) \right) \qtq  \K^*_i(M)  \seteq \tFs_i{}^{\ms{3mu} b} \left( \tEsm_i (M)\right),
$$
where $a = \de_i(\tEm_i (M) )$ and $b =   \de_i(\tEsm_i (M)) $.
\edf
Note that 
$$
\K_i(M)\simeq \tF_i^{\de_i(M)}(M)\qt{if $\ang{i}$ is not a \cf of $M$.}
$$

Thanks to \cite[Proposition 2.16]{KKOP23A},  $\ang{i}$ is not a \cfac of $\K_i(M)$ and $\K^*_i(M)$, and 
\begin{equation} \label{Eq: de Ki Ksi}
\begin{aligned}
\de_i(\K_i(M)) &=0, \qquad   \K_i\bl\K_i(M)\br=\K_i(M), \\
\de_i(\K^*_i(M)) &=0, \qquad \K^*_i \bl \K^*_i(M)\br=\K^*_i(M).
\end{aligned}
\end{equation}

For a family $S = \{ M_k \}_{k\in K}$ of simple modules, we write $\K_i(S) \seteq  \st{ \K_i(M_k)}_{ k\in K }$ and 
$\K^*_i(S) \seteq  \st{ \K^*_i(M_k)}_{k\in K }$. 

We now show several properties of $\K_i$ and $\K_i^*$ in the following lemmas, which will play a crucial role in constructing seeds in $\Cwv$.

\begin{lem} \label{Lem: basics for Ki}
Let  $M\in R\gmod$ be a simple module. 
\bnum
\item  \label{Lem: basics for Ki (i)}
If $M$ is affreal, then $\K_i(M)$ and $\K^*_i(M)$ are affreal.

\item \label{Lem: basics for Ki (ii)}
The following conditions are equivalent:
\bna
\item 
$\de_i(M)=0$ and $\ang{i}$ is not a \cfac of $M$,
\item $\K_i(M)\simeq M$,
\item $\K^*_i(M)\simeq M$.
\ee
\item  \label{Lem: basics for Ki (iii)}
If $M\in(R\gmod){}_i$, then $\K_i(M)\simeq\K_i^*\F_i(M)$
\item If $M\in(R\gmod){}_i \cap {}_i(R\gmod)$, then $\K_i(M)\simeq\F_i(M) $.
\ee	
\end{lem}	
\begin{proof}
(i) Since the case for $ \K^*_i$ can be proved by the same argument, we only prove  the case for $\K_i$.

Let  $(\Ma, \zM )$ be an affinization of $M$. Thanks to \cite[Lemma 3.3]{KKOP21A}, we may assume that $\ep_i(M)=0$. 
We set  $\La \seteq  \sum_{i\in I}  \ep^*_i(M) \La_i  \in \pwtl$ and define $a_\La \seteq  \{  a_{\La, i} (t_i) \}_{i\in I}$ by 
$$
a_{\La, i} (t_i) \seteq  \bchis_i(\Ma) (t_i) \in \Z[\zM, t_i] \qquad \text{ for any $i\in I$,}
$$
where  $\bchis_i(\Ma)$ is the lift of $\ep^*_i(M)$ defined in \cite[Definition 9.1]{KKOP24D} (see also \cite[Definition 3.4]{KKOP21A}).
We write $\beta \seteq  \wt(M)\in\nrtl $, $A \seteq  \Z[\zM]$, and denote by $R^{a_\La}_A$ the cyclotomic quiver Hecke algebra defined in \cite[(1.2)]{KKOP21A}.  By the construction, $\Ma$ is contained in  $R^{a_\La}_A(\La + \beta)\gMod \cap \Ker \qE_{i}$. 
Since we have 
$$
n \seteq  \ang{h_i, \La + \beta} = \ang{h_i, \La  } + \ang{h_i,   \beta} = \de_i(M),
$$
the following category equivalence (see \cite[Lemma 4.14]{R08} and \cite{KK12}) 
$$
\xymatrix{
R^{a_\La}_A(\La + \beta)\Mod \cap \Ker \qE_{i}  \ar[rr]^{\qF_i^{\La  (n)}} && R^{a_\La}_A(s_i (\La + \beta))\Mod \cap \Ker \qF_{i}^{\La}
}
$$
implies that $ \qF_i^{\La  (n) } (\Ma)$ is an affinization of $\qF_i^{\La  (n) } (M) = \K_i(M)$, where $\qF_i^{\La}$ is the functor over $R^{a_\La}_A$ categorifying the action of $f_i$ on $V(\La)$ (see \cite[(1.3)]{KKOP21A}).  

The reality of $\K_i(M)$ follows by applying Lemma \ref{Lem: MhconvN real} to the setting $X = \ang{i^{\de_i(M)}}$ and $Y = M$.

\mnoi
(ii) Set $N\seteq \tEm_i \tEsm_i (M) \simeq \tEsm_i \tEm_i (M)$.
By \cite[Proposition 2.16]{KKOP23A}, the conditions are equivalent to $\de_i(N) = \ep_i(M) + \ep^*_i(M)$. 

\snoi
(iii) Let $M \in (R\gmod)_i$ and write $M' \seteq  \tEsm_i(M)$.  Then we have
\begin{align*}
\de_i(M) &= \ep_i^*(M) + \ang{h_i, \wt(M)} =\ph_i^*(M), \\ 
\de_i(M') &= \ang{h_i, \wt(M')} = 2\ep_i^*(M) + \ang{h_i, \wt(M)} = \de_i(M) + \ep^*_i(M),
\end{align*}
which implies that 
\begin{align*}
\K_i^* \F_i(M)  &=  \K_i^*  \left( \tF_i^{\ph^*_i(M)} (M') \right) 	
= \tF_i^{*\ep^*_i(M)}  \tF_i^{\ph^*_i(M)} (M') \\
&=  \tF_i^{\ph^*_i(M)}  \left( \tF_i^{*\ep^*_i(M)}  (M') \right) = \tF_i^{\de_i(M)} (M) \\ 
 &= \K_i(M),
\end{align*}	 
where the third equality follows from \cite[Theorem 2.17]{KKOP23A}.

\mnoi
(iv) It can be proved by the same argument in (iii).
\end{proof}

\Lemma \label{Lem: KiM KiN}
Let $M$ and $N$ be simple modules and one of them is affreal.
If $M$ and $N$ commute, then 
\bnum
\item \label{it: Ki commute}  $\K_i(M)$ and $\K_i(N)$ commute and $\K_i(M\conv N)\simeq\K_i(M)\conv\K_i(N)$,
\item \label{it: Kis commute} $\K^*_i(M)$ and $\K^*_i(N)$ commute and $\K^*_i(M\conv N)\simeq\K^*_i(M)\conv\K^*_i(N)$.
\ee
\enlemma
\begin{proof}
\eqref{it: Ki commute} \  Note that $  \tEm_i  (M\conv N)\simeq(\tEm_i M)\conv (\tEm_iN)$.
Hence replacing $M$ and $N$ with $\tEm_i M$ and $\tEm_iN$,
we may assume that $\eps_i(M)=\eps_i(N)=0$. 

Let $m = \de_i(M)$ and $n = \de_i(N)$. Since $ m+ n = \de_i(M\conv N) $, we have   
\begin{align*}
\K_i(M \conv N) & = \ang{i}^{\circ (m+n)} \hconv (M \conv N) \simeq \left( \ang{i}^{\circ (m+n)} \hconv M \right) \hconv N \\
& \simeq  \left( \ang{i}^{\circ  n} \hconv \K_i(M) \right)  \hconv N \simeq \left(  \K_i(M) \hconv \ang{i}^{\circ  n} \right) \hconv N \\
& \simeq   \K_i(M) \hconv  \left( \ang{i}^{\circ  n}  \hconv N \right) \\
& \simeq \K_i(M) \hconv \K_i(N).
\end{align*}
In the same manner, we have $\K_i(N \conv M)  \simeq \K_i(N) \hconv \K_i(M) $. 
As $M\conv N \simeq N \conv M$, we have 
$$ 
\K_i(M) \hconv \K_i(N)  \simeq  \K_i(N) \hconv \K_i(M).
$$  
Then  Lemma \ref{Lem: basics for Ki} \eqref{Lem: basics for Ki (i)} says that one of $\K_i(M)$  and $\K_i(N)$ is  affreal,  which implies the assertion by \cite[Corollary 3.9]{KKKO15} and \cite[Proposition 2.10]{KP18}. 

\smallskip

\noindent
\eqref{it: Kis commute} It can be proved in the same manner as above. 
\end{proof}

\Lemma\label{lem:MNa}
Let $M$ be an  affreal simple module, $N$ a simple module and $i\in I$.
Let $a\in\Z$ such that $0\le a\le \de_i(N)$.
We assume that $M$ and $N$ commute.
\bnum
\item \label{it: Ki Fisa commute} If $M$ and $\tFs_i{}^aN$ commute, then $\K_i(M)$ and $\tF_i^{\de_i(N)-a}N$ commute. 
\item \label{it: Kis Fia commute} If $M$ and $\tF_i^a N$ commute, then $\K^*_i(M)$ and $\tF_i^{* \de_i(N)-a}N$ commute. 
\ee

\enlemma
\Proof
Since~\eqref{it: Kis Fia commute} can be proved in the same manner, we focus on~\eqref{it: Ki Fisa commute}.

We may assume that $\ang{i}$ is not a \cfac of $M$.
Set $n \seteq \de_i(N)-a\in\Z_{\ge0}$ and $m \seteq \de_i(M)$.
Since $ (\tF_i^n N, \K_iM, \ang{i^a})$ is normal by the simplicity of
$\K_iM\conv\ang{i^a}$, we have
\eqn
\tFs_i{}^a\bl\tF_i^nN\hconv\K_iM\br
&&\simeq \hd\bl \tF_i^n N\conv\K_iM\conv\ang{i^a}\br  \allowdisplaybreaks\\
&&\simeq  \hd\bl (\tF_i^n N)\conv\ang{i^a}\conv\K_iM\br  \allowdisplaybreaks\\
&& \simeq  \hd\bl \tF_i^n( N\hconv\ang{i^a})\conv\K_iM\br  \allowdisplaybreaks\\
&&\simeq  \hd\bl \K_i(N\hconv\ang{i^a})\conv\K_iM\br  \allowdisplaybreaks \\
&& \underset{*}\simeq  \K_i\bl (\tFs_i{}^aN)\conv M\br,
\eneqn
where $\underset{*}\simeq$ follows from Lemma \ref{Lem: KiM KiN}.
On the other hand, we have
\eqn
\K_iM\hconv\tF_i^nN
&&\simeq \hd\bl \K_iM\conv\ang{i^n}\conv N\br \simeq \hd\bl\ang{i^n}\conv \K_iM\conv N\br\\
&&\simeq \hd\bl\ang{i^{n+m}}\conv M\conv N\br,
\eneqn
which implies 
\eqn
\tFs_i{}^a\bl\K_iM\hconv\tF_i^n N\br
&&\underset{*}\simeq\hd\bl \ang{i^{n+m}}\conv M\conv N\conv\ang{i^a}\br\\
&&\simeq\hd\bl \ang{i^{n+m}}\conv M\conv (\tFs_i{}^aN)\br\\
&&\simeq \K_i(M\conv \tFs_i{}^aN).
\eneqn
Here $\underset{*}\simeq$ follows the fact that
$\ang{i^b}\conv X\conv \ang{i^c}$ has a simple head
for any simple $X$ if $b+c\le\de_i(X)$ (\cite[Lemma 2.17]{KKOP23A}).
In conclusion,  we have
$$\tFs_i{}^a\bl\tF_i^nN\hconv\K_iM\br
\simeq\tFs_i{}^a\bl   \K_iM\hconv\tF_i^nN\br,
$$
which implies that
$\tF_i^nN\hconv\K_iM\simeq \K_iM\hconv\tF_i^nN$. 
Thus we have the assertion. 
\QED

  \Lemma\label{lem:Kmain}
Let $M$ and $N$ be simple modules.
Assume that one of them is \afr.
\bnum
\item \label{it:Mainhconv}
If $\ep_i(N)=0$, then we have
\eqn
&&\tF_i^{\de_i(M)+\de_i(N)}(M\hconv N)\simeq\K_i(M)\hconv\K_i(N)\qtq\\
&&  \La\bl\K_i(M),\K_i(N)\br=\La(M,N)+\de_i(M)\cdot\La(\ang{i},N)
-\de_i(N)\cdot\La(\ang{i}, M).
\eneqn
In particular, we have
$$\K_i(M)\hconv\K_i(N)\simeq\ang{i^a}\conv     \K_i(M\hconv N)
\qt{where $a=\de_i(M)+\de_i(N)-\de_i(M\hconv N)$.}$$
\item \label{it:Mainhconv2} 
If $\ep^*_i(M)=0$, then we have
\eqn
&&\tF_i^{* \de_i(M)+\de_i(N)}(M\hconv N)\simeq\K^*_i(M)\hconv\K^*_i(N)\qtq\\
&&  \La\bl\K^*_i(M),\K^*_i(N)\br=\La(M,N)+\de_i(N)\cdot\La(M, \ang{i})
-\de_i(M)\cdot\La( N, \ang{i}).
\eneqn
In particular, we have
$$\K^*_i(M)\hconv\K^*_i(N)\simeq      \K^*_i(M\hconv N) \conv \ang{i^a} 
\qt{where $a=\de_i(M)+\de_i(N)-\de_i(M\hconv N)$.}$$

\item\label{it:Mainde}
If $\ep_i(M)=\ep_i(N)=0$, then we have
\eqn
&&\de\bl\K_i(M),\K_i(N)\br=\de(M,N).
\eneqn
\item \label{it:Mainde2}
If $\ep^*_i(M)=\ep^*_i(N)=0$, then we have
\eqn
&&\de\bl \K^*_i(M),\K^*_i(N)\br=\de(M,N).
\eneqn

\ee

\enlemma
\Proof
\eqref{it:Mainhconv} 
We  set $m \seteq \de_i(M)$ and $n \seteq \de_i(N)$.  Since $(\ang{i},N)$
is an
unmixed pair, the triple $(\ang{i^{m+n}},M,N)$ is a normal sequence by \cite[Corollary 2.13]{KKOP23A}.
Hence   $\ang{i^{m+n}}\conv M\conv N$ has a simple head and
\begin{align*}
\hd\bl\ang{i^{m+n}}\conv M\conv N\br &\simeq
\hd\bl\ang{i^{n}}\conv\K_i(M)\conv N\br
\simeq\hd\bl\K_i(M)\conv\ang{i^{n}}\conv N\br \\
&\simeq\K_i(M)\hconv\K_i(N).
\end{align*}
The normality of $(\ang{i},M,N)$ tells us that
$\La\bl\K_i(M),N\br=\La(\ang{i^m},N)+\La(M,N)$,
which implies 
\eqn
\La\bl\K_i(M),\K_i(N)\br
&&     =     \La\bl\K_i(M),\ang{i^n}\hconv N\br
=\La\bl\K_i(M),\ang{i^n}\br+\La\bl\K_i(M),N\br\\
&&     =-n\La\bl\ang{i},\K_i(M)\br+\La(\ang{i^m},N)+\La(M,N)\\
&&     =-n\La(\ang{i},M)+m\La(\ang{i},N)+\La(M,N),
\eneqn
where the second and third equality follow from the fact that $\K_i(M)$ and $\ang{i}$ commute.
Since $\de_i(M)+\de_i(N)\ge \de_i(M\hconv N)$ (see \cite[Section 3.2]{KKKO18} and see also \cite[Proposition 4.2]{KKOP20}), we have the last assertion. 

\smallskip
\noindent
\eqref{it:Mainhconv2} It can be proved in the same manner as~\eqref{it:Mainhconv}.

\smallskip
\noindent
\eqref{it:Mainde} and~\eqref{it:Mainde2} follow from~\eqref{it:Mainhconv} and~\eqref{it:Mainhconv2} respectively. 
\QED

For a module $M \in R\gmod$, we denote by $\ell(M)$ the \emph{composition length of $M$}.
\begin{lem}  \label{lem:ifactor}
Let $M$ and $N$ be simple modules such that one of them is affreal and  $\ell(M\conv N)\le2$. 
\bnum
\item \label{it: epiMN0}
If $\ep_i(M)=\ep_i(N)=0$, 
then we have either 
\bna
\item $\K_i(M) \hconv \K_i(N) \simeq  \K_i(M \hconv N)$ and $\de_i(M)+\de_i(N)= \de_i(M\hconv N)$, or 
\item $\K_i(N) \hconv \K_i(M) \simeq  \K_i( N\hconv M)$ and 
$\de_i(M)+\de_i(N)= \de_i(N\hconv M)$. 
\ee

\item \label{it: episMN0}
If $\ep^*_i(M)=\ep^*_i(N)=0$, 
then we have either 
\bna
\item $\K^*_i(M) \hconv \K^*_i(N) \simeq  \K^*_i(M \hconv N)$ and $\de_i(M)+\de_i(N)= \de_i(M\hconv N)$, or 
\item $\K^*_i(N) \hconv \K^*_i(M) \simeq  \K^*_i(N \hconv M)$ and 
$\de_i(M)+\de_i(N)= \de_i(N\hconv M)$. 
\ee
\ee

\end{lem}
\begin{proof}
We focus on proving~\eqref{it: epiMN0} since~\eqref{it: episMN0} can be shown in the same manner.

If $\ell(M\conv N)=1$, then the assertion is obvious.
Hence, we may assume that $\ell(M\conv N)=2$.

Then,  we have the short exact sequence
$$
0 \to N \hconv M \to M \conv N
 \to M \hconv N \to 0. 
$$
By Proposition \ref{prop: Cao general}, we have either
\begin{itemize}
\item  $\La(M \hconv N,\ang{i}) = \La(M,\ang{i}) + \La(N,\ang{i})$ or 
\item  $\La(N \hconv M,\ang{i}) = \La(M,\ang{i}) + \La(N,\ang{i})$. 
\end{itemize}
Since the proof for the other case is similar, we assume the first case, i.e., 
\begin{align} \label{Eq: LMNi=LMi+LNi}
\La(M \hconv N,\ang{i}) = \La(M,\ang{i}) + \La(N,\ang{i}).
\end{align}
 We write $m = \de_i(M)$ and $n = \de_i(N)$.
The triple $(\ang{i},M,N)$ is normal as the pair $(\ang{i},N)$ is unmixed. Thus, by \eqref{Eq: LMNi=LMi+LNi}, we have 
\begin{align*}
 \de_i(M \hconv N) & = \frac{1}{(\al_i,\al_i)}  \bl \La(\ang{i},M\hconv N) + \La(M\hconv N,\ang{i}) \br
 \allowdisplaybreaks\\
& =  \frac{1}{(\al_i,\al_i)}  \bl \La(\ang{i},M) + \La(\ang{i},N) + \La(M,\ang{i}) + \La( N,\ang{i}) \br  \allowdisplaybreaks\\
& = \de_i(M) + \de_i(N).
\end{align*}
Hence the result follows from Lemma~\ref{lem:Kmain}. 
\end{proof}

\Lemma \label{Lem: primeness}
Let $M$ be a prime simple module. 
\bnum
\item \label{it: epi0prime}
If $\ep_i(M)=0$, then $\K_i(M)$ is prime.
\item \label{it: epis0prime} If $\ep^*_i(M)=0$, then $\K^*_i(M)$ is prime.
\ee
\enlemma
\Proof
As~\eqref{it: epis0prime} can be proved in the same manner, we prove only~\eqref{it: epi0prime}.

Let $\K_i(M)\simeq X\conv Y$.
Then we have
$$M\simeq\tEm_i\K_i(M)\simeq\tEm_iX\conv\tEm_iY.
$$
This implies that 
either $\tEm_iX$ or $\tEm_iY$ is  isomorphic to $\one$, say
$\tEm_iX\simeq\one$. Then we have  $X\simeq\ang{i^n}$ for some $n\in\Z_{\ge0}$, which implies $n=0$ by Lemma \ref{Lem: basics for Ki} \eqref{Lem: basics for Ki (ii)}.
\QED

\Lemma\label{lem;ncf}
Let $i\in I$ and let $M$ and $N$ be \afr simple modules such that
$\eps_i(M)=\eps_i(N)=0$.
Assume further that $M$ and $N$ \ncf.
Then $\K_i(M)$ and $\K_i(N)$ \ncf.
\enlemma
\Proof
Assume that $S$ is a \cf of $\K_i(M)$ and $\K_i(N)$. 
Then $\tEm_i(S)$ is a  \cf of $M$ and $N$. Hence $\tEm_i(S) \simeq\one$
and hence $S\simeq\ang{i^n}$ for some $n\in\Z_{\ge0}$.
Since $S$ is a \cf of $\K_i(M)$, we have $n=0$.
\QED

\vskip 2em

\section{Laurent families in  $\Cwv$}  \label{Sec: Laurent Cwv}

In this section, we investigate how the operations $\K_i$ and $\F_i$ act on Laurent families in the category $\Cwv$. 

\subsection{Stability for $\Cwv$}

Let $w, v \in \weyl$. 
From the definitions of $\gW$ and $\gW^*$, we have
$$
\gW(X)\cup\gW(Y)\subset\gW(X\conv Y) \qtq \gW^*(X)\cup\gW^*(Y)\subset\gW^*(X\conv Y)
$$
for {\em non-zero} modules $X$ and $Y$ in $R\gmod$. Thus we obtain the following:
\begin{equation} \label{Eq:cfactorst}
\begin{aligned}
&\text{ $X\conv Y\in\Cw$ implies that  $X$ and $Y$ belong to $\Cw$,} \\
&\text{ $X\conv Y\in\Cwv[*,v]$ implies that $X$ and $Y$ belong to $\Cwv[*,v]$.}
\end{aligned}
\end{equation}

\Prop\label{prop:stability}
Let $w,v\in\weyl$ such that $s_iw<w$.
\bnum
\item \label{prop:stability (i)}
$\tEm_i$ sends any simple in $\Cwv$ to a simple  in $\Cwv[s_iw,v]$.
\item \label{prop:stability (ii)}
If $s_iv<v$, then
$\F_i\cl\Cwv[s_iw,s_iv]\to\Cwv$ is an equivalence.
\item \label{prop:stability (iii)} $\K_i$ sends any simple in either $\Cw[s_iw]$ or $\Cw[w]$ to a simple in $\Cw$.
\item \label{prop:stability (iv)} $\K_i$ sends   any simple in $\Cw[*,v]$ to a simple in $\Cw[*,v]$ whenever $s_iv>v$.
\ee
\enprop
\Proof
By the assumption, we have $\ang{i} \in \Cw$.

\smallskip
\noindent
\eqref{prop:stability (i)} follows from Lemma~\ref{lem:tEmst} and the fact that 
$\Cw[*,v]$ is stable by $\tE_i$.

\smallskip
\noindent
\eqref{prop:stability (ii)} follows from Proposition~\ref{prop:Refwv}.

\smallskip
\noindent
\eqref{prop:stability (iii)}
If $M$ is a simple in $\Cw[s_i w]$,  then the statement follows from Proposition \ref{prop:Refwv} and Lemma \ref{Lem: basics for Ki} \eqref{Lem: basics for Ki (iii)}.
We now assume that $M$ is a simple in $\Cw$.
If $\de_i(M)=0$, then
$$
M\simeq \ang{i^n}\conv\K_i(M) \qquad \text{ for some $n\in\Z_{\ge0}$,}
$$ 
which gives the assertion by \eqref{Eq:cfactorst}.
If $\de_i(M)>0$, then the assertion follows from
$\K_i(M)\simeq \tF_i^{\;\de_i(M)}M$.

\smallskip
\noindent
\eqref{prop:stability (iv)} follows from the fact that $\Cw[*,v]$ is stable by $\tF_i$.
\QED

\Lemma\label{lem:bruD} 
Let $i\in I$ and $w,v\in\weyl$ such that $v\le w$.
\bnum
\item \label{it: siw<w}
Assume that $s_iw<w$.
Then we have
\bna\item
$\ep_i\bl\dM(w\La,v\La)\br\le -\ang{h_i,w\La}$ for any $\La\in\pwtl$,
\item   $\ep_i\bl\dM(w\La,v\La)\br=-\ang{h_i,w\La}$ for any $\La\in\pwtl$
if and only if $v\le s_iw$.
\ee
\item \label{it: v<siv}
Assume that $v<s_iv$.
Then we have
\bna\item
$\ep^*_i\bl\dM(w\La,v\La)\br\le \ang{h_i,v\La}$ for any $\La\in\pwtl$,
\item   $\ep^*_i\bl\dM(w\La,v\La)\br=\ang{h_i,v\La}$ for any $\La\in\pwtl$
if and only if $s_iv\le w$.
\ee
\ee
\enlemma
\Proof
Since the proof of~\eqref{it: v<siv} is similar, we prove only~\eqref{it: siw<w}.
Since (a) and ``if'' part of (b) are well known (see \cite[Lemma 9.1.4 and Lemma 9.1.5]{KKKO18} for example),
we will focus on proving the ``only if'' part of (b).

Assume that $\ep_i\bl\dM(w\La,v\La)\br=-\ang{h_i,w\La}$
for any $\La\in\pwtl$.
The isomorphism class of $\dM(w\La,v\La)$ 
corresponds to  the element of $\Aqn$ given by $U_q^+(\g)\ni a\mapsto (au_{w\La}, u_{v\La})\in\cor$.
Putting $n\seteq-\ang{h_i,w\La}$, the element $ \E^{(n)}_i\dM(w\La,v\La)$  corresponds to
$$
U_q^+(\g)\ni a\longmapsto (ae_i^{(n)}u_{w\La}, u_{v\La})
=(au_{s_iw\La}, u_{v\La}).
$$
Since it does not vanish, we have
$u_{v\La}\in U_q^+(\g)u_{s_iw\La}$, which implies $v\le s_iw$ by
Lemma~\ref{lem:bru}
\QED

In the following theorem, we give a criterion on $w$ and $v$ to determine whether $\ang{i}$ is central in $\Cwv$. 

\Th\label{th:ifro}
Let $i\in I$ and let $v,w\in\weyl$ such that 
$$
v\le w, \qquad v<s_iv \qtq s_iw<w.
$$
Then the following conditions are equivalent:
\bna
\item \label{it: i-center a}
$\ang{i}$ is a \cfac of $\dM(w\La,v\La)$ for some $\La\in\pwtl$,
\item \label{it: i-center b} $\ang{i}$ lies in the center of $\Cwv$, namely any simple in
$\Cwv$ commutes with $\ang{i}$,
\item \label{it: i-center c} $s_iv\not\le s_iw$.
\ee
\enth
\begin{proof}
We have $v\le s_iw$.
Note also that $\ang{ i }$ is contained in $\Cwv$ by the assumption.

\mnoi
\eqref{it: i-center a} $  \Rightarrow $ \eqref{it: i-center b} follows from \cite[Theorem 4.4]{KKOP23A}.

\mnoi
\eqref{it: i-center b} $  \Rightarrow $ \eqref{it: i-center c} \ 
Assuming that
\begin{align} \label{Eq: svlesw}
s_iv\le s_iw,
\end{align}
we shall show that there exists a simple in $\Cwv$ which does not commute with $\ang{i}$ by using induction on $\bl\ell(w),\ell(w)-\ell(v)\br$.
Here, the induction is taken with respect to the lexicographic order on $\Z_{\ge0}^2$, 
i.e.,
$(a,b) > (c,d)$ if and only if either ($a>c$) or ($a=c$ and $b>d$).

By \eqref{Eq: svlesw}, we have $\ell(s_iw)>0$.
Hence there exists $j\in I$ such that $s_iws_j<s_iw$, which implies 
\begin{align} \label{Eq: wsw swsws}
ws_j<w \qtq  s_iws_j<ws_j.
\end{align}
We now consider the following cases.
\bnum
\item Suppose that  $vs_j<v$. Since $ v s_j < v < s_i v$, we have 
$vs_j < s_i v s_j$ and $ s_ivs_j < s_i v $. By \eqref{Eq: svlesw} and \eqref{Eq: wsw swsws}, we have
$$
v s_j \le w s_j \qtq  s_ivs_j \le s_i w s_j.
$$
By the induction hypothesis, $\ang{i}$ is not in the center of $\Cwv[ws_j,vs_j]$.
There is an equivalence of monoidal categories $\tCwv\simeq\tCwv[ws_j,vs_j]$ which sends $\ang{i}$ to $\ang{i}$, where $\tCwv$ and $\tCwv[ws_j,vs_j]$ are the localized categories (\cite[Theorem 4.8]{KKOP24C}). Hence $\ang{i}$ is also not in the center of $\Cwv$.

\item Suppose that  $v< vs_j$. 
\renewcommand{\themyc}{\roman{myc}}
\setcounter{myc}{\value{enumi}}

\begin{enumerate} [({\themyc}-1)] 
\item   Assume $s_iv<s_ivs_j$. By \eqref{Eq: svlesw} and \eqref{Eq: wsw swsws}, we have
$$
v s_j \le w \qtq s_i v s_j \le s_i w.
$$
The induction hypothesis implies that 
$\ang{i}$ is not in the center of $\Cwv[w,vs_j]\subset \Cwv$.

\item 
Assume that $s_ivs_j<s_iv$.
Since 
$$
v < vs_j, \quad s_i(vs_j)<vs_j \qtq  s_ivs_j < s_iv = (s_ivs_j)s_j, 
$$
we have $s_ivs_j\le v$, which implies  
$$
v=s_ivs_j.
$$
This tells us that $v\al_j=\al_i$, i.e.,  $h_j=v^{-1}h_i$, and 
$$
\dM(w\La_j,s_iv\La_j)=\dM(w\La_j,vs_j\La_j)\in\Cwv[w,vs_j]\subset\Cwv.
$$
We have 
\eqn\ang{h_i, s_iv\La_j}=-\ang{v^{-1}h_i,\La_j}
=-\ang{h_j, \La_j}=-1.\eneqn
Hence, we have
\eqn
\hs{10ex}\de_i\bl\dM(w\La_j,s_iv\La_j)\br
=\eps_i\bl\dM(w\La_j,s_iv\La_j)\br+&&
\eps^*_i \bl\dM(w\La_j,s_iv\La_j)\br\\
&&+
\ang{h_i,w\La_j-s_iv\La_j}.
\eneqn
By Lemma~\ref{lem:bruD}, we have
$\eps_i\bl\dM(w\La_j,s_iv\La_j)\br=-\ang{h_i,w\La_j}$, and
$v<s_iv$ implies that $\eps_i^*\bl\dM(w\La_j,s_iv\La_j)\br=0$.
Hence we obtain
$\de_i\bl\dM(w\La_j,s_iv\La_j)\br=\ang{h_i,-s_iv\La_j}=1$.
Thus, $\dM(w\La_j,s_iv\La_j)$ does not commute with $\ang{i}$
since $\de_i\bl\dM(w\La_j,s_iv\La_j)\br=1$. 
\end{enumerate}
\ee

\mnoi
\eqref{it: i-center c} $ \Rightarrow $ \eqref{it: i-center a} \ 
Assume that $s_iv\not\le s_iw$.
As $v \le w$,  $v < s_i v$ and $s_iw < w$, we have $s_iv\le w$. Set 
$$
C \seteq  \dM(w\La,v\La) \qtq n\seteq  \ep_i(C)= -\ang{h_i,w\La}.
$$ 
By Lemma~\ref{lem:bruD}, there exists $\La\in\pwtl$
such that
$$
\eps_i\bl \tEsm_i(C) \br = \eps_i\bl\dM(w\La,s_iv\La)\br<n = \eps_i(C).
$$
Hence Lemma~\ref{lem:ifac} says that $\ang{i}$ is a \cfac of $C$.
\end{proof}

\begin{lem} \label{lem:iaKi}
Let $w,v \in \weyl$ and $i\in I$.
Assume that $v< w$,  $s_iw<w$, and $v<s_iv$.
Then for $\La \in \pwtl$ we have
\eqn
\dM(w\La, v\La) \simeq \ang{i^n} \conv \K_i(\dM(s_iw\La,v\La)) \quad \text{for some} \ n \in \Z_{\ge 0}.
\eneqn
\end{lem}
\begin{proof}
Since  $\dM(s_iw\La,v\La) \simeq \tE_i^{\max} (\dM(w\La, v\La))$, we have
$$
\K_i(\dM(w\La,v\La))= \K_i(\dM(s_iw\La,v\La)).
$$

Note that $\ang{i} \in \catC_{w,v}$ and  $\de_i(\dM(w\La,v\La))=0$ (\cite[Theorem 4.4]{KKOP23A}). 
 By Lemma \ref{Lem: basics for Ki} \eqref{Lem: basics for Ki (ii)},  we conclude that $$\dM(w\La,v\La) \simeq \ang{i^n} \conv \K_i(\dM(w\La,v\La)) \simeq \ang{i^n} \conv \K_i(\dM(s_iw\La,v\La))$$
for some $n \in \Z_{\ge 0}$.
\end{proof}

\subsection{Stability of Laurent families}
In this subsection, we investigate properties of $\K_i$ and $\F_i$ related to Laurent families.
Recall the notion of Laurent families given in Section \ref{Sec: prep}.

The following proposition says that $\F_i$ and $\K_i$ are compatible with Laurent families.

\Prop \label{prop: laurent}
Let $w,v\in\weyl$ and $i\in I$ such that $v<w$, $v<s_iv$ and $s_iw<w$, and 
let $\cM=\st{M_k}_{k\in K}$ be a Laurent family in $\Cwv[s_iw,v]$.
Then, we have
\bnum
\item \label{it: Fim Laurent}
$\F_i(\cM)$ is a Laurent family  in  $\Cwv[w,s_iv]$,
\item \label{it: Kim Laurent}
$\cM'\seteq\st{\ang{i}}\cup\K_i(\cM)$ is a Laurent family  in  $\Cwv[w,v]$.
\ee
\enprop
\Proof
~\eqref{it: Fim Laurent} is obvious since $\F_i\cl \Cwv[s_iw,v]\to\Cwv[w,s_iv]$ is an equivalence of monoidal categories by Proposition \ref{prop:stability} \eqref{prop:stability (ii)}.

\smallskip \noindent
\eqref{it: Kim Laurent} Set $K' \seteq \st{0}\sqcup K$ and 
$$
M'_0 \seteq \ang{i} \qtq M'_k\seteq\K_i(M_k)
$$
so that $\cM'=\st{M'_k}_{k\in K'}$. Note that $\tEm_i M'_k\simeq M_k$ for any $k\in K$.

Suppose that  $\bfa' = (a_k)_{k\in K'}$ and $\bfb' = (b_k)_{k\in K'}$
satisfy $\cM'(\bfa') \simeq \cM'(\bfb')$. Applying $\tEm_i$, we obtain
$\conv[{k\in K}]M_k{}^{\circ a_k}\simeq\conv[{k\in K}]M_k{}^{\circ b_k}$. Hence we obtain $a_k=b_k$ for any $k\in K$ by the assumption, which implies $a_0=b_0$. This shows that $\cM'$ is independent.

 Let $X\in\Cwv$ be a simple module commuting with $M'_k$  for any $k\in K'$.
By Proposition \ref{prop:stability} \eqref{prop:stability (i)} and \cite[Lemma 3.1]{KKOP18}, $\tEm_iX\in\Cwv[s_iw,v]$ commutes with every $\tEm_iM'_k\simeq M_k$ for $k\in K$.
Hence the assumption says that there exists $\bfa= (a_k)_{k\in K}$ such that
$\tEm_iX\simeq \mathop{\circ}\limits_{k\in K}M_k{}^{\circ a_k}$.
Since $X$ commutes with $\ang{i}$, there exists $a_0$ such that
$$X\simeq \ang{i^{a_0}}\conv\K_i(\tEm_iX)
\simeq\mathop{\circ}\limits_{k\in K'}M'_k{}^{\circ a_k},
$$
which completes the proof.
\QED

Let $\cM$ be a commuting family of \afr simple modules in $\Cwv$.
Let us consider the following condition on $\cM$:
\eq
\label{cond:frozen}
\hs{4ex}\parbox{76ex}{Any $M\in\cM$ satisfies
	exactly one of the following conditions:
	\bna\setlength{\itemindent}{3ex}
	\item \label{cond:frozen (a)}
    there exists a simple module in $\Cwv$ which does not commute with $M$,
	\item \label{cond:frozen (b)}
	$M$ is a \cfac of $\dM(w\La,v\La)$ for some $\La\in\pwtl$.
	\ee
}
\eneq

\Prop \label{prop: FK ex fr}
Let $i\in I$ and let $w,v \in \weyl$ such that $v<w$, $v<s_iv$, and $s_iw<w$.
Let $\cM$ be a  commuting family of \afr simple modules in $\Cwv[s_iw,v]$ 
satisfying \eqref{cond:frozen}.
\bnum
\item \label{it: Fim satisfying}
$\F_i(\cM)\subset \Cwv[w,s_iv]$ also satisfies \eqref{cond:frozen},
\item \label{it: Kim satisfying}
$\{\ang{i} \}\cup \K_i(\cM)\subset \Cwv[w,v]$
also satisfies \eqref{cond:frozen}.
\ee
\enprop
\Proof
~\eqref{it: Fim satisfying} It follows from Proposition \ref{prop:stability} \eqref{prop:stability (ii)} and the fact that $\F_i\bl\dM(s_iw\La,v\La)\br
\simeq \dM(w\La,s_iv\La)$ for any $\La \in \pwtl$.

\smallskip
\noindent
\eqref{it: Kim satisfying} It follows from Theorem \ref{th:ifro} 
and Lemma~\ref{lem:iaKi} together with the following: 
for an affreal simple $M$ and a simple $N$ in  $\Cwv[s_iw,v]$,
$\K_i(M)$ and $\K_i(N)$ do not strongly commute if $M$ and $N$ do 
not strongly commute by Lemma \ref{lem:Kmain} \eqref{it:Mainde}.
\QED

\vskip 2em

\section{Monoidal categorifications} \label{Sec: MC Cw}
 {\em From now until the end of the paper, we assume that $R$ is symmetric.}
Hence the Cartan matrix $\cmC$ is symmetric, and
we choose a bilinear form $( \cdot \, , \cdot )$ on the weight lattice $\wtl$
such that $ (\al_i, \al_i)=2$ for all $i\in I$.
Note that until now the theory of monoidal categorification by quiver Hecke algebras is available only in the symmetric case.

\subsection{Monoidal seeds}\label{subsec:monoidal}

In this subsection, we briefly review the monoidal categorification of $\Cw$ (\cite{KKKO18}).

We first recall the notion of monoidal categorification. Let $\cC$ be a full subcategory of $R\gmod$ which has $\one$ and is stable under taking convolution products, extensions, subquotients and grading shifts. 
\begin{definition} [{\cite{KKKO18}}] \label{Def: def MC} 
Let $\sfJ$ be an index set with a decomposition
$\sfJ= \sfJ^\ex \sqcup \sfJ^\fr$.
\bnum
\item
A triple $\Sigma = (\{M_i\}_{i\in \sfJ}, L, \tB )$ is a \emph{quantum monoidal seed} in $\cC$ if it satisfies the following:  
\bna
\item $\{M_i\}_{i\in \sfJ}$ is a commuting family of real simple modules in $\cC$, 
\item $L = (l_{ij})_{i,j\in \sfJ}$ is an integer-valued skew-symmetric matrix,
\item $\tB = (b_{ij})_{i\in \sfJ, j\in \sfJ^\ex}$ is an exchange matrix, 
\item $M_i \conv M_j \simeq q^{l_{ij}} M_j \conv M_i $ for any $i,j\in \sfJ$,
 i.e., $\La(M_j,M_i)=l_{ij}$, 
\item $(L,\tB)$ is compatible with $d=2$ (see Section \ref{Sec: QCA}),
\item $\sum_{i\in \sfJ} b_{i,k} \wt(M_i)=0$ for all $k\in \sfJ^\ex$.
\ee
\item  A quantum monoidal seed $\Sigma = ( \{M_i\}_{i\in \sfJ}, L, \tB)$ in $\cC$ admits a \emph{mutation} in direction $k\in \sfJ^\ex$ if 
\bna
\item  there  exists a real simple module $M'_k \in \cC$  
such that $\de(M_k, M'_k )=1$  and an exact sequence
\begin{align*} 
0 \to q^{a}  \conv[{b_{ik} >0}]  M_i^{\circ  b_{ik}} \to  M_k \conv M_k' \to q^{b} \conv[{b_{ik} <0} ] M_i^{\circ  (-b_{ik})} \to 0,
\end{align*}
where the exponents $a,b$ are determined by $\La$ (see \cite[Definition 6.2.3]{KKKO18} for more details on the exponents $a,b$),
\item
the triple $ \mu_k(\Sigma)=( \{M_i\}_{i\in \sfJ \setminus \{ k\}  } \cup \{M_k' \}, \mu_k(L), \mu_k(\tB))$ is a quantum monoidal seed in $\cC$. 
\ee
\item If a quantum monoidal seed $\Sigma$ admits successive mutations in all directions, then we say that $\Sigma$ is 
{\em completely admissible}.  
\item The category $\cC$ is called a \emph{monoidal categorification} of a quantum
cluster algebra $A$ over $\Z[q^{\pm1/2}]$ if
\bna
\item the Grothendieck ring $\Z[q^{\pm1/2}] \otimes_{\Z[q^{\pm1}]} K(\cC) $ is isomorphic to $A$,
\item there exists a completely admissible quantum monoidal seed 
$\Sigma = (\{M_i\}_{i\in \sfJ}, L, \tB )$ in $\cC$ such that 
$ ( \{ q^{-(\wt(M_i), \wt(M_i))/4} [M_i]\}_{i\in \sfJ}, L, \tB)$ is a quantum 
seed of $A$.
\ee
\item Let
$\Sigma = (\{M_i\}_{i\in \sfJ}, L, \tB )$ in $\cC$ be a quantum monoidal seed
in $\shc$.
A simple object $X$ of $\shc$ is called a {\em $\Sigma$-monomial}
if $X\simeq\conv[{k\in\sfJ}]M_k^{\circ n_k}$
for some $(n_k)_{k\in\sfJ}\in\Z_{\ge0}^\sfJ$. 
\ee
\end{definition}

By \cite[Lemma 3.2]{KK19}, the exchange matrix $\tB$ can be computed by using the matrix  $\La$ and $\st{\wt(M_s)}_{s\in\sfJ}$.  Hence  $ \{ M_s \}_{s\in \sfJ}$ along with the decomposition $\sfJ = \sfJ^\ex\sqcup\sfJ^\fr$ determines  the exchange matrix $\tB$. 
For a monoidal seed $\Sigma = (\st{M_s}_{s\in \sfJ}, -\Lambda, \tB)$, 
it  suffices to write
$$
 \seed = (M_s)_{s\in \sfJ} 
$$
to denote the seed $\Sigma$ instead of specifying the entire triple. From now on, we simply write $\seed = (M_s)_{s\in \sfJ}$ for the monoidal seed $\Sigma$ instead of writing the whole triple, and simply call it 
a seed  instead of  calling  a quantum monoidal seed.
We sometimes write $\seed_\ex \seteq  (M_s)_{s\in \sfJ^\ex} $ and 
$\seed_\fr \seteq  (M_s)_{s\in \sfJ^\fr} $.

The following proposition immediately follows from 
Proposition~\ref{prop:Laurentm}.
\Prop
Let $\seed=(M_k)_{k\in\sfJ}$ be a  seed of $\shc$,
and let $\seed'=(M'_k)_{k\in\sfJ}$ be the mutation of $\seed$
at $t\in\sfJ^\ex$.
Assume that $\seed$ is a Laurent family.
Then $\seed'$ is also a Laurent family.
\enprop

We now focus on the category $\Cw$. 
Let $w \in \weyl$ and choose a reduced expression of $w$
$$
\uw = (i_1, i_2, \ldots, i_r).
$$
For an index $1 \le k \le r$ and $j \in I$, we set
\begin{align*}
k_{\uw}(j)^+ &\seteq  \min \bl \{ t \in [1,r] \mid t\ge k,\; i_t=j \} \sqcup \{ +\infty \} \br, \\
k_{\uw}(j)^- &\seteq  \max \bl \{t \in [1,r] \mid t\le k,\; i_t=j \} \sqcup \{-\infty\} \br,
\end{align*}
and
\begin{align*}
k_{\uw}^+ &\seteq  \min \bl \{ t \in [1,r]\mid t> k,\; i_t= i_k \} \sqcup \{+\infty\} \br, \\
k_{\uw}^- &\seteq  \max \bl \{t \in [1,r] \mid t< k,\; i_t=i_k \} \sqcup \{-\infty\} \br. 
\end{align*}
We often drop $\uw$ in the above notations when no confusion arises.

An interval $ [a,b] \subseteq [1,r]$ is called an \emph{$i$-box of $\uw$} if $a\le b$ and $i_a=i_b$.   
From any interval $[a,b] \subseteq [1,r]$, we obtain two $i$-boxes
\[
[a,b \} \seteq  [a,b(i_a)^-] \quad \text{and} \quad \{a,b] \seteq  [a(i_b)^+,b].
\]
A chain $\frakC = ( [a_t,b_t])_{1 \le t \le r}$ of $i$-boxes of $\uw$
is called \emph{admissible} if, for each $t=1,\ldots,r$, the union
\[
 [\tilde{a}_t,\tilde{b}_t] \seteq  \bigcup_{1 \le j \le t} [a_j,b_j]
\]
is an interval of length $t$ and either
\[
[a_t,b_t] = [\tilde{a}_t,\tilde{b}_t\}
\quad\text{or}\quad
[a_t,b_t] = \{\tilde{a}_t,\tilde{b}_t].
\]
For any $i$-box $[a,b]$, the corresponding \emph{determinantial module} is defined by 
$$
\dM_{\uw}[a,b] \seteq  \hd ( \dS_{b} \conv \dS_{ b^- } \conv \cdots \conv  \dS_{ a^+} \conv  \dS_{ a}),
$$
where $\dS_k \seteq  \F_{i_1} \F_{i_2} \ldots \F_{i_{k-1}} (\ang{i_k}) $ is the $k$-th \emph{cuspidal module} for $k\in [1,r]$. We simply write $\dM[a,b]$ for $\dM_{\uw}[a,b]$ if no confusion arises.

For each admissible chain $\frakC = ( \frakc_k )_{k\in [1,r]}$ of $i$-boxes, we set 
$$
\wseed_\frakC \seteq  ( \dM(\frakc_k) )_{k\in \sfJ}\qt{with $\sfJ=[1,r]$.}
$$
where $\sfJ^\fr=\st{k\in[1,r]\mid\text{
$\frakc_k=[1(i)^+,r(i)^-]$ for some $i\in I$}}$.
Define
$$
\La_\frakC \seteq  ( \La(\dM(\frakc_i), \dM(\frakc_j) ) )_{i,j\in \sfJ}
$$
and $\tB_\frakC$ to be  the exchange matrix determined by $\frakC$ in \cite[Section 3]{KK24}. 
It was shown in \cite[Section 11.2]{KKKO18} (see also \cite{KK24}) 
that the triple $\Sigma_\frakC = ( \wseed_\frakC, -\Lambda_\frakC, \tB_\frakC)$  is a quantum monoidal seed of $\Cw$, and $\Cw$ provides a \emph{monoidal categorification} of $A_q(\n(w))$.

This tells us that every cluster monomial corresponds to a real simple module under the monoidal categorification. Such a simple module is called a \emph{cluster monomial module}. When the cluster monomial is a cluster variable (resp.\ frozen variable), we call it a \emph{cluster variable module} (resp.\ \emph{frozen variable module}). 

\begin{remark} \label{Rmk: B from M} 
Even if we choose different reduced expressions and different admissible chains, the corresponding seeds are connected by mutations and permutations (\cite{KKOP24A, KK24}).
Hence the mutation class of $\Sigma_\frakC$ does not depend on the choices of reduced expressions and admissible chains.  
We call it the {\em canonical mutation class of $\Cw$}.
\enrem

\medskip
The following monoidal seed arising from the admissible chain $\frakC = ( \{1,k] )_{k\in [1,r]}$ is called the \emph{GLS seed} of $\Cw$: 
\begin{align} \label{Eq: GLS seed}
 \wseed_{\uw,\GLS} \seteq  (\dM_{\uw}\{1,k])_{k\in[1,r]}
\end{align}
(see \cite{GLS11, GLS13} and see also \cite[Section 11]{KKKO18}).

  \vskip 2em 
  
\section{Construction of monoidal seeds in $\Cw$} \label{Sec: seed for Cw} 

In this section, we give a method to construct a seed of $\Cw$,
which may not be obtained by the method of admissible chains of
$i$-boxes.

\subsection{Operators $\K_i$ and $\F_i$} \

By \cite[Theorem 4.10]{KK19}, we have the following lemma. 

\Lemma \label{Lem: criterion for cv} 
Let $\wseed=(M_k)_{k\in\sfJ}$ be a seed of $\Cw$ 
in the canonical mutation class, and let $t\in \sfJ^\ex$.
\bnum 
\item \label{it: seedprime}
Let $\wseed'=(M'_k)_{k\in\sfJ}$ be a seed of $\Cw$  \cmu.
If  $M_k\simeq M'_k$ for $k\in \sfJ \setminus \{ t\}$ and $M_t\not\simeq M'_t$, then $\wseed'=\mu_t(\wseed)$.
\item \label{Lem: criterion for cv (ii)}
Suppose that  a simple module $X \in \Cw$ satisfies: 
\bna
\item $X$ commutes $M_k$ for any $k\in \sfJ\setminus \st{t}$, 
\item $\de(X, M_t)=1$,   
\item $X \hconv M_t$ and  $M_t\hconv X$ \ncf. 
\ee
Then we have $X \simeq \mu_t(M_t)$.  

\ee
\enlemma
\Proof
\eqref{it: seedprime}
Set $\wseed''=\mu_t(\wseed)=(M''_k)_{k\in \sfJ}$.
Since $M_t'$ commutes with $M_k$ for any $k\in\sfJ \setminus\{t\}$,
\cite[Theorem 4.10]{KK19} implies that $M'_t$ 
is a $\wseed$-monomial or a $\wseed''$-monomial.
Since $[M'_t]$ is prime, $M'_t=M''_t$, 
which implies that $ \wseed' = \wseed''$.

\mnoi
\eqref{Lem: criterion for cv (ii)} Let $\wseed'=\mu_t(\wseed)$ and let $M_t' = \mu_t(M_t)$. From the assumption, we have the following two exact sequences:
\begin{equation} \label{Eq: two exs}
\begin{aligned} 
0 \longrightarrow M'_t \hconv M_t \longrightarrow  &M_t \conv M'_t  \longrightarrow M_t \hconv M'_t \longrightarrow 0,	\\
0 \longrightarrow X \hconv M_t \longrightarrow  & M_t \conv X  \longrightarrow M_t \hconv X \longrightarrow 0.	
\end{aligned}
\end{equation}	
Combining \cite[Theorem 4.10]{KK19} with (a), 
we know that $X$ is a $\wseed'$-monomial  or a $\wseed$-monomial.
Then (b) implies that
 $X$ is a $\wseed'$-monomial,  i.e.,
$X\simeq M_t'{}^{\circ n}\conv Y$  
for some $n\in\Z_{\ge0}$ and 
a $(\wseed\setminus\st{M_t})$-monomial $Y$.
By (b), we have $n=1$.
Hence
$X\hconv M_t\simeq Y\conv(M'_t\hconv M_t)$ and 
$M_t\hconv X\simeq Y\conv(M_t\hconv M'_t)$.
Then (c) implies that $Y\simeq \one$.
\QED

\Prop\label{prop:Fseed}
Let $i\in I$ and $w\in\weyl$ such that $s_iw<w$.
For a seed $\wseed$ of $\Cw[s_iw]$
\cmu, 
we define 
\begin{align*}
\twseed\seteq  \st{\dM(w\La_i,\La_i)}\cup\F_i(\wseed).
\end{align*}
Then we have the following.
\bnum
\item \label{prop:Fseed (i)}
$\twseed$ is a seed of $\Cw$ \cmu.

\item \label{prop:Fseed (ii)}
Moreover, we have 
\eqn
\twseed_\ex&&= \bc \st{\dM(w\La_i,s_i\La_i)}\cup\F_i(\wseed_\ex)&\text{if $s_iw>s_i$,}\\
\F_i(\wseed_\ex)&\text{otherwise,}\ec  \\
   \twseed_\fr 
  &&= 
\bc\st{\dM(w\La_i,\La_i)}\cup \bl \F_i(\wseed_\fr) \setminus \st{\dM(w\La_i, s_i\La_i)} \br&\text{if $s_iw>s_i$,}\\
   \F_i(\wseed_\fr) \cup \st{\dM(w\La_i, \La_i)=\ang{i}}&\text{otherwise.}     \ec   
  \eneqn

\item \label{prop:Fseed (iii)}
If $\wseed'$ is a mutation of
$\wseed$, 
then $\twseed' $ is also a mutation of $\twseed$.
\ee
\enprop
\Proof
Since it is obvious when $\ell(w)=1$, we assume that $\ell(w)>1$.

Let $r=\ell(w)$ and take a reduced sequence
$  \uw  =(i_1,i_2, \ldots,i_r)$ of $w$ with $i_1=i$,
and 
$\underline{s_iw}=(i_2,\ldots, i_r)$ of $s_iw$. 
Take a seed $\wseed_\circ=\st{\dM_{\underline{s_iw}}[k,r-1\}}_{k\in[1,r-1]}$ of $\Cw[s_iw]$
which is \cmu.
Then 
$$
\twseed_\circ = \st{\dM(w\La_i,\La_i}\cup\F_i(\wseed_\circ)
=\st{\dM_{\uw }[k,r\}}_{k\in[1,r]}
$$
is a seed of $\Cw$. Note that $ \twseed_\circ$ is \cmu.

Since $\F_i$ is a monoidal functor from $\Cw[s_iw]$ to $\Cw$, taking $\twseed$ is compatible with mutations. 
Since $\wseed$ is connected to $\wseed_\circ$ via mutations, we have~\eqref{prop:Fseed (i)} and~\eqref{prop:Fseed (iii)} . 

If $s_iw \not> s_i$, then $\dM(w\La_i,s_i\La_i)$ is trivial. Suppose that $s_iw > s_i$. Then  $\dM(s_iw\La_i,\La_i)$ is frozen in $\Cw[s_iw]$ but $\dM(w\La_i,s_i\La_i) = \F_i(\dM(s_iw\La_i,\La_i))$ (\cite[Lemma 10.6]{KKOP24D}) is exchangeable in $\Cw$. Thus we have~\eqref{prop:Fseed (ii)}.
\QED

\Prop \label{prop:Kseed}
Let $i\in I$ and $w\in \weyl$ such that $s_iw<w$.
For 
a seed $\wseed$ of $\Cw[s_iw]$ \cmu,
we define 
\begin{align*}
	\twseed\seteq\st{\ang{i}}\cup\K_i(\wseed).
\end{align*}
Then we have the following.
\bnum
\item \label{prop:Kseed (i)}
$\twseed$ is a seed of $\Cw$ \cmu.

\item\label{prop:Kseed (ii)}
Moreover, we have 
\eqn
\twseed_\ex&&=\bc\st{\ang{i}}\cup\K_i(\wseed_\ex)&\text{if $s_iw>s_i$,}\\
\K_i(\wseed_\ex)&\text{otherwise,}
\ec\\
\twseed_\fr&&=\bc\K_i(\wseed_\fr)&\text{if $s_iw>s_i$,}\\
\st{\ang{i}}\cup\K_i(\wseed_\fr)&\text{otherwise.}
\ec
\eneqn

\item \label{prop:Kseed (iii)}
If $\wseed'$ is a mutation of
$\wseed$, 
then $\twseed'$ is also a mutation of $\twseed$.
\ee
\enprop
\begin{proof}
We may assume that $\ell(w) > 1$ since the statement is trivial otherwise. 

Let $r=\ell(w)$ and take a reduced sequence
$\uw=(i_1,i_2, \ldots,i_r)$ of $w$ with $i_1=i$,
and $\underline{s_i w} =(i_2,\ldots, i_r)$ of $s_iw$.
Taking the seed 
\begin{align} \label{Eq: initial T0}
\wseed_\circ = \wseed_{\underline{s_i w} ,\GLS} = (\dM_{\underline{s_i w}}\{1,k])_{k\in[1,r-1]}
\end{align}
of $\Cw[s_iw]$,
we have  
\begin{align} \label{Eq: initial tT0}
\twseed_\circ  = \st{\ang{i}}\cup\K_i(\wseed))= (\dM_{\uw }\{1,k])_{k\in[1,r]} = \wseed_{\uw ,\GLS}
\end{align}
which is a seed of $\Cw$ \cmu. 

We now consider a seed  $\wseed$ of $\Cw[s_iw]$, and $\wseed'$ is a mutation of $\wseed$. Suppose that $\twseed =  (\ang{i},\K_i(\wseed))$ is a seed of $\Cw$. 
Let $M'\in\wseed'$ be the mutation of $M\in\wseed$  and set 
$$
X\seteq M\hconv M' \qtq Y\seteq M' \hconv M.
$$
Since $M'$ is a mutation of $M$,  $\de(M,M')=1$, $X$ and $Y$ are \monomials[\wseed], and $X$ and $Y$ do not have
a common convolution factor.  Moreover $M'$ commutes with any member of $\wseed\setminus\st{M}$.

Since the value $\ep_i $ of any module in $\Cw[s_iw]$ is zero, we have 
\bna
\item if   $\K_i(U) \simeq \K_i(V)$ for some simple modules $U,V \in \Cw[s_iw]$, then $U\simeq V$, 
\item $\de\bl\K_i(M),\K_i(M')\br=1$ by Lemma~\ref{lem:Kmain}\;\eqref{it:Mainde},
\item \label{Eq: nm (b)} $\K_i(M)\hconv \K_i(M')\simeq \ang{i^m}\conv\K_i(X)$ and
$\K_i(M')\hconv \K_i(M)\simeq \ang{i^n}\conv\K_i(Y)$ for some $m,n\in \Z_{\ge0}$ by  Lemma~\ref{lem:Kmain}\;\eqref{it:Mainhconv},
\item 
$\K_i(X)$ and $\K_i(Y)$ are \monomials[\K_i(\wseed)] and they \ncf.
\item 
$\K_i(M')$ commutes with
any member of $\K_i(\wseed)\setminus\st{\K_i(M)}$
by Lemma~\ref{lem:Kmain}\;\eqref{it:Mainde}.
\ee
Since $\K_{i}(M')$ commutes with $\ang{i}$ and 
has no $\ang{i}$ as a factor by Lemma \ref{Lem: basics for Ki} \eqref{Lem: basics for Ki (ii)}, either of $m$ and $n$ in \eqref{Eq: nm (b)} is zero. 
Lemma \ref{Lem: criterion for cv} \eqref{Lem: criterion for cv (ii)} says that  $\K_i(M')$ is a mutation of $\K_i(M)$, i.e., $ \twseed'$ is a mutation of $\twseed$. 
This implies~\eqref{prop:Kseed (i)} and~\eqref{prop:Kseed (iii)} by \eqref{Eq: initial T0} and \eqref{Eq: initial tT0}.

If $s_iw \not> s_i$, then $\ang{i}$ is  a central object in $\Cw$ (up to a multiple of $q$), which says that $\ang{i}$ is frozen.  If $s_iw > s_i$, then $\ang{i}$ is not central, i.e., not frozen. 
 Hence we have~\eqref{prop:Kseed (ii)}	
\end{proof}

\Prop\label{prop:wmut}
Let $i\in I$ and $w\in\weyl$ such that $s_i<s_iw<w$.
Let $\wseed$ be a seed of $\Cw[s_iw]$ \cmu, and set 
$\twseed\seteq\st{\ang{i}}\cup\K_i(\wseed)$, a seed of $\Cw$.
\bnum
\item \label{it: wmut (i)}
Let $Z\in\Cw$ be the mutation of $\ang{i} \in\twseed$. Then we have
\bna
\item $\de_i(Z)=1$,
\item $Y\seteq\tEm_iZ\in\Cw[s_iw]$ is a \monomial[\wseed],
\item $X\seteq\tFs_iY$ is a \monomial[\wseed],
\item there is an exact sequence
$$0\To \K_i(X)\To \ang{i}\conv Z\To \K_i(Y)\To0.$$
\ee
\item \label{it: wmut (ii)}
Conversely, if $X$ and $Y$ are \monomials[\wseed] such that
\bna\setlength{\itemindent}{4ex}
\item $X\simeq\tFs_iY$,
\item $X$ and $Y$ have no common \cfac,
\ee
then $\tF_i^{\de_i(X)}Y\simeq\tEs_i\K_i(X)$ is a mutation of
$\ang{i}
\in\twseed$.
\ee
\enprop
\Proof 
Note that $\twseed$ is a seed of $\Cw$ \cmu by Proposition \ref{prop:Kseed}.

\smallskip
\noindent
\eqref{it: wmut (i)}\ 
Since $Z$ is a mutation of $\ang{i}$, we have $\de_i(Z)=1$
and there exist \monomials[\wseed] $X$ and $Y$ and an exact sequence 
$$0\To \K_i(X)\To \ang{i}\conv Z\To \K_i(Y)\To0.$$
Note that $\eps_i(X)=\eps_i(Y)=0$ since $X,Y\in\Cw[s_iw]$.
Hence we have $Y\simeq\tEm_i Z$ and $X\simeq\tEm_i\tFs_i Z$.
Since $\de_i(Z)>0$, we have
$\tEm_i\tFs_i Z\simeq\tFs_i \tEm_iZ\simeq \tFs_iY$.

\mnoi
\eqref{it: wmut (ii)}\ Set $Z=\tF_i^{\de_i(X)}Y$.
If $\de_i(Y)=0$, then $X\simeq Y \conv \ang{i}$, which contradicts (b). 
It follows that  $\de_i(Y)=\de_i(X)+1$  by \cite[Proposition 2.16 (i)]{KKOP23A}. Thus we have $\de_i(Z)=1$,  
$\tF_i(Z)\simeq\K_i(Y)$ and $\tFs_i(Z)\simeq\K_i(X)$.
Hence we have
an exact sequence $$0\To \K_i(X)\To \ang{i}\conv Z\To \K_i(Y)\To0.$$
Applying Lemma~\ref{lem:MNa} (i) to the setting $M=S$, $N=Y$ and $a=1$ for any $S\in\wseed$, we obtain that $Z$ commutes with $\K_i(S)$. 
 By Lemma~\ref{lem;ncf},  $\K_i(X)$ and $\K_i(Y)$ \ncf. 
Then the assertion follows from Lemma \ref{Lem: criterion for cv} \eqref{Lem: criterion for cv (ii)}.
\QED

\subsection{Seed construction for $\Cw$} \label{subsec:exp}

In this subsection, we present a new construction of monoidal seeds of $\Cw$.

Let $w\in\weyl$ and fix a reduced sequence of $w$
$$
\uw=(i_1,i_2, \ldots,i_r)
$$ 
throughout this subsection.
Recall the notations $k_{\uw}^+$, $k_{\uw}^-$ given in Section \ref{Sec: MC Cw}. 
We set $w_{\ge k} \seteq s_{i_k} s_{i_{k+1}} \cdots s_{i_r}$ for 
$k\in[1,r]$.

Let $\kfc=(c_1,\ldots, c_r)\in \st{\kt,\ft}^r$. 
The sequence $\kfc$ is called a \emph{KF sequence} for $\Cw$.
We define
\begin{align} \label{Eq: TM}
\wseed_{\uw,\kfc} =  (M_k)_{k\in[1,r]},
\end{align} 
where $M_k$ is given by  
\begin{align} \label{Eq: Mk G}
M_k \seteq 
\bc \G_1\cdots \G_{k-1}(\ang{i_k})&\text{if $c_k = \kt$,}\\
\G_1\cdots \G_{k-1}\bl\dM(w_{\ge k}\La_{i_k},\La_{i_k})\br
&\text{if $c_k = \ft$,}
\ec
\end{align}
and
\eqn
\G_t \seteq \bc \K_{i_t}&\text{if $c_t = \kt$,}\\
\F_{i_t}  &\text{if $c_t = \ft$.}\ec
\eneqn

\Rem\label{rem:right}
Let $\kfc'=(c'_s)_{s\in[1,r]}$ be another KF sequence which satisfies
$c'_s=c_s$ for any $s\in[1,r]$ such that $(i_s)^+\le r$.
Then  $\wseed_{\uw,\kfc'}=\wseed_{\uw,\kfc}$.
\enrem

\Lemma \label{Lem: frozen in [1,r]}
The following conditions on $s\in[1,r]$ are equivalent:
\bna
\item $s$ is frozen, 
\item $s=\min\st{k\in\Ft\mid i_k=i_s}$
where $\Ft=\st{k\in[1,r]\mid \sfc_k=\ft}\cup\st{k\in[1,r]\mid k^+=\infty}$.
\item $M_s\simeq\dM(w\La_{i_s},\La_{i_s})$.
\ee
\enlemma

\begin{proof}
By Remark~\ref{rem:right},
we may assume that $\Ft=\st{k\in[1,r]\mid\sfc_k=\ft}$.
Then the result follows from
$\K_i\bl\dM(w'\La_j,\La_j)\br\simeq\dM(s_iw'\La_j,\La_j)$
if $w'<s_iw'$.
\end{proof}

We  set 
\begin{align} \label{Eq: index for Cw}
\sfJ \seteq  [1,r], \qquad \sfJ^\fr \seteq  \{ k \in \sfJ \mid \text{$k$ is frozen} \}, \qquad \sfJ^\ex \seteq  \sfJ \setminus \sfJ^\fr 
\end{align}
where the frozen condition is given in Lemma \ref{Lem: frozen in [1,r]}.
Thanks to Propositions \ref{prop:Fseed} and  \ref{prop:Kseed}, we have the following.
\begin{prop} \label{prop: Twc}
The family $\wseed_{\uw,\kfc} $ with $\sfJ = \sfJ^\ex \sqcup \sfJ^\fr $ is 
a seed of $\Cw$ \cmu.
\end{prop}

\begin{example} \label{Ex: Cw}
Let $R$ be a symmetric quiver Hecke algebra of finite type $A_3$ and let $w := w_0 $ be the longest element in the Weyl group $\weyl$. Let us take
$$
\uw = (i_1, \ldots, i_6) = (1,2,1,3,2,1) \qtq \kfc= (c_1, \ldots, c_6) = (\kt, \ft,   \kt, \ft, \kt, \kt).
$$
Using \eqref{Eq: F_i in crystal}, Definition \ref{Def: Ki and K*i} and \cite[Lemma 9.1.5]{KKKO18} (see also \cite[Lemma 1.7, Proposition 4.1]{KKOP18}), 
the corresponding seed $\wseed_{\uw,\kfc} =  (M_k)_{k\in[1,r]}$ can be computed as 
\bnum
\item $M_1 = \ang{1}$,
\item $M_2 =  \K_1(\dM(w_{\ge 2}\La_{2},\La_{2})) = \K_1(\ang{2,1,3,2}) = \ang{2,1,3,2}$, 
\item $M_3 = \K_1\F_{2} (\ang{1}) = \K_1 (\ang{2,1}) = \ang{2,1}  $, 
\item $M_4 = \K_1\F_{2} \K_1 ( \dM(w_{\ge 4}\La_{3},\La_{3}) ) = \K_1\F_{2} \K_1 ( \ang{3} ) = \K_1(\ang{2,3}) = \ang{1,2,3}$, 

\item $M_5 = \K_1\F_{2} \K_1 \F_3  (\ang{2})  = \K_1\F_{2}  (\ang{1,3,2})   = \ang{2,1,3}$, 
\item $M_6 = \K_1\F_{2} \K_1 \F_3 \K_2  (\ang{1})  = \K_1\F_{2} \K_1  (\ang{3,2,1}) = \K_1  (\ang{3,2,1}) = \ang{3,2,1}$,
\ee
where $ \ang{2,1,3,2} = \ang{2,1} \hconv \ang{3,2}$ and $\ang{2,1,3} = \ang{2,1} \hconv \ang{3}$ (see Section \ref{Sec: QHA}).
Note that $M_1, M_3, M_4, M_6$ are 1-dimensional and $M_2, M_5$ are 2-dimensional. The exchangeable and frozen indices are given as 
$$
\sfJ^\ex = \{ 1,3,5 \} \qtq \sfJ^\fr = \{ 2,4,6 \}.
$$
The seed $\wseed_{\uw,\kfc}$ can be depicted as follows: 
$$
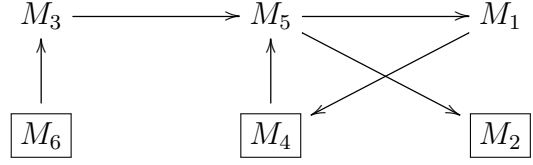
\xymatrix{
M_3  \ar[rr]  && M_5 \ar[rr] \ar[rrd]  && M_1 \ar[lld] \\ 
\text{\fbox{$M_6$}} \ar[u] && \text{\fbox{$M_4$}} \ar[u]  && \text{\fbox{$M_2$}}
}
$$
where the boxed ones are frozen. We remark that this seed $\wseed_{\uw,\kfc}$ does not coincide with  any seed $\wseed_\frakC$ arising from an admissible chain $\frakC $ of $i$-boxes with respect to the reduced expression $\uw$.
\end{example}

\vskip 2em 

\section{Construction of monoidal seeds in  $\Cwv$}  \label{Sec: Swv for Cwv}

In this section, we introduce a new construction of monoidal seeds 
of  $\Cwv$ using $\K_i$ and $\F_i$. 

\subsection{Pair of seeds} \label{Sec: pair}

Let $w,v\in\weyl$ such that $v\le w$. 

 We first prepare the following two lemmas in order to prove
the main theorem of this subsection. 
\begin{lem} \label{lem:Mw Mwv}
Let $\La \in \wlP_+$. If the determinantial module $\dM(w\La, \La)$ is contained in $\Cwv$, then we have $v\La=\La$.
\end{lem}
\begin{proof}
By \cite[Proposition 4.8]{KKOP18} (see also \cite[Proposition 3.52]{Murata24}), we have 
$$
\dM(w\La, v\La) \hconv \dM(v\La, \La) \simeq \dM(w\La, \La).
$$
Since $\dM(w\La, \La)$ is contained in $\Cwv$, we have $\dM(v\La, \La) = \one$, which yields the assertion. 
\end{proof}

\Lemma\label{lem:main}
Let $i\in I$ with $s_iw<w$.
Let $M$ be a simple module in $\Cw$ such that $\de_i(M)=1$ and 
let $N$ be a real simple in $\Cw$ such that $N\simeq\K_i(N)$.
Then there exists a simple subquotient $S$ of $M\conv N$ such that
$\de_i(S)=1$.
\enlemma
\Proof
Assume that $\de_i(S)\not=1$ for any simple subquotient $S$ of $M\conv N$.
Since any simple subquotient $S$ of $M\conv N$ satisfies
$\de_i(S)\le \de_i(M)+\de_i(N)=1$ (see \cite[Proposition 3.2.10]{KKKO18}),
any simple subquotient $S$ of $M\conv N$ commutes with $\ang{i}$. This means that 
$\ang{i}\conv S$ is simple.
As $\K_i(M)\conv N$ is a quotient of $\ang{i}\conv M\conv N$,
 $\ang{i}$ is a \cfac of
any simple subquotient of $\K_i(M)\conv N$.
Hence $[\K_i(M)]\cdot[N]$ is divisible by $[\ang{i}]$ in
$K(\Cw)\vert_{q=1}$, which is the polynomial algebra 
of the variables including $[\ang{i}]$.
Hence, $[\K_i(M)]$ or $[N]$ is divisible by $[\ang{i}]$.
Therefore, \cite[Theorem 4.1]{KK19} implies that $\ang{i}$ is
a \cfac of $\K_i(M)$ or $N$, which is a contradiction by Lemma \ref{Lem: basics for Ki} \eqref{Lem: basics for Ki (ii)}.
\QED

The following lemma, a consequence of \eqref{Eq:cfactorst}, is one of the main ingredients to construct a seed of $\Cwv$.

\Lemma\label{lem:freez}
Let $w,v\in\weyl$ such that $v\le w$.
Let $\wseed=\bl (M_k)_{k\in\sfK},\tB)$ with $\sfK=\sfK^\ex\sqcup\sfK^\fr$ be a completely admissible seed in $R\gmod$.
Let $\sfJ=\sfJ^\ex\sqcup\sfJ^\fr$ satisfies
\bna
\item \label{it: J contain K}
$\sfJ\subset\sfK$ and $\sfJ^\ex\subset\sfK^\ex$,
\item \label{it: Mk in Cwv} $M_k\in \Cwv$ for any $k\in\sfJ$,
\item \label{it: good restriction}
$\tB\vert_{(\sfK\setminus\sfJ)\times\sfJ^\ex}=0$.
\ee
Then $\seed=\bl (M_k)_{k\in\sfJ},\tB\vert_{\sfJ\times\sfJ^\ex}\br$
is a completely admissible seed in
$\Cwv$.
\enlemma
\Proof
It is trivial that $\seed$ is completely admissible in $R\gmod$.
Hence it is enough to show that the conditions
\eqref{it: J contain K}--\eqref{it: good restriction} are satisfied after mutations.

Let $s\in\sfJ^\ex$ and $\wseed'=\bl (M'_k)_{k\in\sfK},\tB')$
be the mutation of $\wseed$ at $s$.
Then it is easy to see that $\tB'\vert_{(\sfK\setminus\sfJ)\times\sfJ^\ex}=0$.
Since $M_s\hconv M'_s$ and $M'_s\hconv M_s$ are $\seed$-monomials
by~\eqref{it: good restriction}, $M_s\conv M'_s$ also belongs to $\Cwv$.
Hence $M'_s$ also belongs to $\Cwv$ by \eqref{Eq:cfactorst}.
Hence $\seed'=\bl (M'_k)_{k\in\sfJ},\tB'\vert_{\sfJ\times\sfJ^\ex}\br$
is a seed in $\Cwv$.

Thus we obtain the desired result.
\QED

Let $i\in I$ and assume that
$$
v<s_iv \qtq s_iw<w.
$$
Let $\wseed=\bl \{ M_k \}_{k\in \Ks},\tB_\circ\br$ be a seed of $\Cw[s_iw]$ \cmu,
 where $\tB_\circ\seteq  (b_{j,k})_{(j,k)\in \Ks\times \Ks^\ex}$ 
is an exchange matrix.
We define 
$$
\seed \seteq \wseed\cap\Cwv[s_iw,v] \qtq \Js \seteq \st{k\in \Ks\mid M_k\in\seed},
$$
and we assume that
$$
\text{ $\seed$ satisfies \eqref{cond:frozen} with
$\Cwv[s_iw,v]$ instead of $\Cwv$.} 
$$
Set 
\begin{equation} \label{Eq: S and J}
\begin{aligned}
\seed_\ex &\seteq  \{ M \in \seed \mid \text{$M$ satisfies \eqref{cond:frozen (a)} in
\eqref{cond:frozen}} \}, \quad  \Js^\ex \seteq  \{ k\in \Js \mid M_k \in \seed_\ex  \}, \\
\seed_\fr &\seteq  \{ M \in \seed \mid \text{$M$ satisfies \eqref{cond:frozen (b)} in
\eqref{cond:frozen}} \},\quad  \Js^\fr \seteq  \{ k\in \Js \mid M_k \in \seed_\fr  \}.
\end{aligned}
\end{equation}
Hence we have
$$\seed_\ex\subset\wseed_\ex\qtq\seed\cap\wseed_\fr\subset\seed_\fr,
$$
where the second inclusion follows from Lemma \ref{lem:Mw Mwv}.

Define 
$$
\tseed \seteq  \{\ang{i}\} \cup \K_i(\seed)  \qtq \twseed \seteq  \{\ang{i}\} \cup\K_i(\wseed),
$$
and  set $\tJs \seteq  \st{\ko} \sqcup \Js$,  $\tKs \seteq \st{\ko} \sqcup \Ks$
and
$$
\tM_\ko \seteq \ang{i}, \qquad   \tM_k \seteq \K_i(M_k) \quad  \text{ for $k\in\Ks$.}
$$
Note that $\tseed\subset\twseed$ and  $\tseed$ satisfies \eqref{cond:frozen} in $\Cwv[w,v]$ by Proposition \ref{prop: FK ex fr}.   
We now define $\tseed_\ex$, $\tseed_\fr$, $\tJs^\ex$ and $\tJs^\fr$ in the same manner as \eqref{Eq: S and J}. Theorem \ref{th:ifro} tells us that 
$$
\tJs^\ex = 
\begin{cases}
\{ \ko \} \cup \Js^\ex  &\text{ if $s_iv  \le s_iw$,}\\
\Js^\ex  &\text{ otherwise,}
\end{cases}
\qtq
\tJs^\fr = \tJs \setminus \tJs^\ex. 
$$
Note that 
$\tseed_\ex \subset \twseed_\ex $ and $  \tseed \cap \twseed_\fr \subset \tseed_\fr.$

Proposition \ref{prop:Kseed} says that there is an exchange matrix 
$\tB \seteq  (\tb_{j,k})_{(j,k)\in\tKs\times\tKs^\ex}$ such that the pair
$(\twseed, \tB)$ is a seed of $\Cw$. Moreover, by Proposition \ref{prop:Kseed} \eqref{prop:Kseed (iii)} together with the seeds 
\eqref{Eq: initial T0} and \eqref{Eq: initial tT0},
we have
\begin{align} \label{Eq: tB B}
\tB\vert_{\Ks\times\Ks_\ex}=\tB_\circ.
\end{align}
We further assume that 
\eq \label{eq:frbarr}
\begin{aligned}
&\text{$\seed$ is a Laurent family in $\Cwv[s_iw,v]$}, \\
& \tB_\circ  \vert_{(\Ks\setminus \Js) \times \Js^\ex} =0.
\end{aligned}
\eneq
Note that \eqref{eq:frbarr} implies the pair $(\seed,\tB\vert_{\Js\times\Js^\ex})$ is a completely admissible seed
in $\Cwv[s_iw,v]$  by Lemma~\ref{lem:freez}.

\Th\label{th:main} \ 
\bnum
\item \label{it: still outside0}
We have 
$$
\tB\vert_{(\tKs\setminus\tJs) \times \tJs^\ex}=0.
$$
\item \label{it: still ca}
The pair $ \bl\tseed, \tB\vert_{\tJs \times \tJs^\ex}\br$ is a completely admissible seed in $\Cwv$.
\ee

\enth
\begin{proof}
Since~\eqref{it: still ca} follows from~\eqref{it: still outside0} by Lemma~\ref{lem:freez}, 
we focus on proving~\eqref{it: still outside0}.
We may assume that 
\begin{align} \label{Eq: svsw}
s_i v \le s_i w
\end{align}
since otherwise $\ko\in\tJs^\fr$ by Theorem~\ref{th:ifro} and 
the assertion follows directly from \eqref{Eq: tB B} and \eqref{eq:frbarr}. 
In this case, we have $ \tJs^\ex = \{ \ko \} \cup \Js^\ex $. 
By \eqref{Eq: tB B} and \eqref{eq:frbarr}, it is enough to show that 
$$
\tB\vert_{ (\tKs\setminus\tJs) \times \{ k_0 \} }=0.
$$
Let us first show that  $\tseed$  admits a mutation at $\ko$.

Set $C\seteq  \conv[{k\in \Js} ]\tM_k$.
Note that $C$ is   an  \monomial[\tseed]. 
By \eqref{Eq: svsw} and Theorem \ref{th:ifro}, there exists a simple module $M\in\Cwv$ such that $\de_i(M)>0$.
Replacing $M$ with $\tF_i^{\de_i(M)-1}M$, we may assume that $\de_i(M)=1$ from the beginning.
Assume that $\de(C,M)>0$. Then Lemma~\ref{lem:main} implies that
there exists a simple subquotient $M'$ of $M\conv C$ such that
$\de_i(M')=1$. 
Hence $M'$ is a simple module in $\Cwv$ such that
$\de_i(M')=1$ and $\de(C,M')<\de(C,M)$ by \cite[Corollary 4.1.2]{KKKO18} (see also \cite[Corollary 3.18]{KKOP21A}).  
Arguing by induction on $\de(C,M)$,
we assume from the outset that
there exists a simple $M\in\Cwv$ such that
$$
\de_i(M)=1 \qtq \text{$M$ commutes with $C$.}
$$
Hence $M$ commutes with $\tM_k$ for any $k\in\Js$.

As $\ang{i}$ and $M$ commute with every $\tM_k$, 
the modules $\ang{i}\hconv M$ and $M\hconv\ang{i}$ commute
with $\tM_k$ for all $k\in\tJs$.
Proposition \ref{prop: laurent} and \eqref{eq:frbarr} say that
$\tseed$ is a Laurent family in $\Cwv$, which implies that  there exist
$\bfa = (a_k)_{k\in \tJs}, \bfb  = (b_k)_{k\in \tJs} \in\Z_{\ge0}^{\oplus \tJs}$ 
such that
$$
M\hconv\ang{i}\simeq \tseed (\bfa) \qtq \ang{i}\hconv M\simeq\tseed(\bfb),
$$
where $\tseed (\bfa)$ is defined in \eqref{Eq:M(a)}.
Since $\ang{i}$ is not a \cfac of $M\hconv\ang{i}$ and  $\ang{i}\hconv M$ by Lemma \ref{Lem: basics for Ki} \eqref{Lem: basics for Ki (ii)},
we have $ a_{\ko}= b_{\ko}=0$.
Define  $\bfc \seteq  (c_k)_{k\in \tJs}\in\Z_{\ge0}^{\oplus \tJs}$ by $ c_k=\min( a_k, b_k)$.
Then $[\ang{i}]\cdot[M]$ is divisible by $[\tseed (\bfc)]$ in
$K(\Cw)\vert_{q=1}$, which means that $[M]$ is divisible by  $[\tseed (\bfc)]$.
Hence \cite[Theorem 4.1]{KK19} implies that $[\tseed (\bfc)]$ is 
a \cfac of $M$, i.e., there exists a simple $X\in\Cw$ such that
$M\simeq X\conv \tseed (\bfc)$.
Hence,  \eqref{Eq:cfactorst} implies that $X\in \Cwv$ and
$$
X\hconv\ang{i}\simeq \tseed (\bfa-\bfc) \qtq \ang{i}\hconv X\simeq \tseed (\bfb-\bfc).
$$
Replacing $M$ with $X$ from the beginning, we may assume that
$$
M\hconv\ang{i}\simeq \tseed (\bfa) \qtq \ang{i}\hconv M\simeq\tseed (\bfb)
$$
with 
$\bfa\in\Z_{  \ge0}^{\oplus \tJs}$ and $\bfb\in\Z_{\ge0}^{\oplus \tJs}$ such that
\begin{align} \label{Eq: bfa bfb}
 a_{\ko}= b_{\ko}=0 \qtq  a_k b_k=0 \qt{for any $k\in\tJs$.}
\end{align}
Moreover $M$ commutes with all $\tM_k$ ($k\in\Js$). Thus we have an exact sequence
\begin{align} \label{Eq: exact mut}
0\To \tseed (\bfa) \To\ang{i}\conv M\To \tseed (\bfb) \To0.
\end{align}

We shall  show that $M$ commutes with $\tM_k$ for $k\in\Ks$.
Applying $\tM_k\conv -$ to the above exact sequence \eqref{Eq: exact mut}, we obtain
$$
0\To \tM_k\conv \tseed (\bfa)  \To \tM_k\conv\ang{i}\conv M\To 
\tM_k\conv \tseed (\bfb)\To0.
$$
Note that $\tM_k\conv \tseed (\bfa)$ and $\tM_k\conv \tseed (\bfb)$ are simple since $\twseed$ is a seed in $\Cw$. 
Hence $\ang{i}\conv \tM_k\conv M\simeq \tM_k\conv\ang{i}\conv M$ has length $2$.
If $\tM_k$ and $M$ do not commute, then
$$[\ang{i}\conv \tM_k\conv M]=[\ang{i}\conv (\tM_k\hconv M)]+
[\ang{i}\conv (M\hconv \tM_k)],$$
and  $\ang{i}\conv (\tM_k\hconv M) $, $\ang{i}\conv (M \hconv  \tM_k) $ are simple.
Since $\ang{i}\conv \tM_k\conv M $ has simple head, we have
$$
\ang{i}\conv (\tM_k\hconv M)\simeq \tM_k\hconv(\ang{i}\hconv M)\simeq
\tM_k\hconv \tseed (\bfb)  \simeq \tM_k \conv \tseed (\bfb).
$$
By \eqref{Eq: bfa bfb}, $\tM_k \conv \tseed (\bfb)$ cannot have $\ang{i}$ 
as a convolution factor, which is a contradiction.
Thus $M$ commutes with $\tM_k$ for any $k\in\Ks$, which means that  $M$ is a mutation of $\twseed$ at $\ko$ by Lemma \ref{Lem: criterion for cv} \eqref{Lem: criterion for cv (ii)}.

It follows from \eqref{Eq: exact mut}
that
$\tb_{k, \ko}=0$ for any $k\in\tKs\setminus\tJs$, which implies that 
$$
\tB\vert_{ (\tKs\setminus\tJs) \times \{ k_0 \} }=0.
$$
This  completes the proof. 
\end{proof}

\subsection{Seed construction for $\Cwv$} \label{Sec:wvseed}
Recall the seed construction for $\Cw$ given in Section \ref{subsec:exp}.
Let $w,v \in \weyl$ with $v\le w$ and take a reduced sequence of $w$
$$
\uw=(i_1,i_2,\ldots,i_r).
$$
For $k\in[1,r+1]$, we set $w_{\ge k} \seteq s_{i_k} s_{i_{k+1}} \cdots s_{i_r}$ and define $v_{\ge k}^{\uw}$ inductively by
\eqn
v_{\ge1}^{\uw}&&=v,\\
v_{\ge k}^{\uw} &&=\bc v_{\ge k-1}^{\uw}&\text{if $s_{i_{k-1}}v_{\ge k-1}^{\uw} > v_{\ge k-1}^{\uw} $,}\\
s_{i_{k-1}}v_{\ge k-1}^{\uw}&\text{if $s_{i_{k-1}}v_{\ge k-1}^{\uw} < v_{\ge k-1}^{\uw}$,}\ec
\hs{3ex}\text{for $k \in [2, r+1]$.}
\eneqn
Note that $w_{\ge r+1}= v_{\ge r+1}^{\uw}= \id$.
We set 
$$w_{\le k}\seteq w(w_{\ge k+1})^{-1}=s_{i_1}\cdots s_{i_k}\qtq v_{\le k}\seteq v(v_{\ge k+1})^{-1}\qt{for $k\in[1,r]$.}
$$

We set
\begin{align} \label{Eq: [1,r]wv}
[1,r]_v^{\uw} \seteq  \st{k\in[1,r]\mid v_{\ge k}^{\uw} = v_{\ge k+1}^{\uw}}.
\end{align}
If no confusion arises, we simply write $v_{\ge k}$ and $[1,r]_v$
instead of $v_{\ge k}^{\uw}$ and $[1,r]_v^{\uw}$, respectively.
We then consider the seed of $\Cw$
\begin{align*}
\wseed \seteq  \wseed_{\uw,\kfc} = (M_k)_{k\in [1,r]}, 
\end{align*}
where $\kfc=(c_1,\ldots, c_r)\in \st{\kt,\ft}^r$ given by 
\begin{align} \label{Eq: ct for T}
 c_t \seteq  
 \bc
\kt &\text{if $t \in [1,r]_v$,}\\
\ft&\text{otherwise}
\ec
\end{align}
(see  \eqref{Eq: TM} for $\wseed_{\uw,\kfc}$). 
This sequence $\kfc$ is called the \emph{KF sequence associated with $\uw$ and $v$.}

Define 
\begin{align} \label{Eq: def of Swv}
\seed_{\uw, v} \seteq  ( M_k)_{k\in [1,r]_v^{\uw}} \subset \wseed.
\end{align}
It follows from \eqref{Eq: Mk G} that
$$
M_k= \G_1\cdots \G_{k-1}(\ang{i_k})\qt{for any $k\in [1,r]_v$,}
$$
where $ \G_t = \K_{i_t}$ if $c_t = \kt$ and $ \G_t = \F_{i_t}$ if $c_t = \ft$.

\Prop
Let $k\in [1,r]$ such that $c_k = \ft$. 
Then we have
$$
M_k=\dM(w\La_{i_k},v_{\le k-1}\La_{i_k}),
$$
where $ v_{\le k-1}\seteq  v (v_{\ge k})^{-1}$.
In particular, $M_k \not \in \Cwv$ and $M_k$  commutes with any simple in
$\Cwv$.
\enprop
\Proof
Note that 
$$
M_k= \G_1\cdots \G_{k-1}( \dM( w_{\ge k} \La_{i_k}, \La_{i_k})).
$$ 
For $1 \le a \le b \le r$, we set $v_{[a,b]} \seteq  v_{\le a-1}^{-1} v_{\le b} = v_{\ge a} v_{\ge b+1}^{-1}$.
We shall use descending induction on $t \in [1,k]$ to show
$$
\G_{ t} \G_{ t+1} \cdots \G_{ k-1 }\dM(w_{\ge k}\La_{i_k}, \La_{i_k})\simeq
\dM(w_{\ge t}\La_{i_k}, v_{[t,k-1]}\La_{i_k}).
$$
It is obvious at $t=k$, and for $t<k$, we have
\eqn
&& \G_{ t}\cdots \G_{ k-1 }\dM(w_{\ge k}\La_{i_k}, \La_{i_k})
\simeq \G_{ t}\dM(w_{\ge t+1}\La_{i_k}, v_{[t+1,k-1]}\La_{i_k})
\simeq\dM(w_{\ge t}\La_{i_k}, v_{[t,k-1]}\La_{i_k}).
\eneqn

The last statement follows from the fact that $v_{\le k-1} \lneq v $ and  
$M_k$ commutes with any simple in $\Cwv[w,v_{\le k-1}]\supset\Cwv$.
\QED

We set  
\begin{align} \label{Eq: index for Cwv}
\sfJ_{\uw, v} \seteq  [1,r]_{v}^{\uw}, \quad 
\sfJ_{\uw, v}^{\fr} \seteq  \{ k \in \sfJ_{\uw, v} 
\mid  s_{i_k} v_{\ge k}^{\uw} \not \le w_{\ge k+1}^{\uw}  \},
\quad  \sfJ_{\uw, v}^{\ex} \seteq  \sfJ_{\uw, v} \setminus \sfJ_{\uw, v}^{\fr}.
\end{align}
\begin{lem} \label{Lem: frozen Jfr}
$\sfJ_{\uw, v}^{\fr} = \{ k\in \sfJ_{\uw, v} \mid \text{$M_k$ satisfies $\eqref{cond:frozen}$ \eqref{cond:frozen (b)}}\}. $
\end{lem}
\begin{proof}
We set  $ \sfJ := \sfJ_{\uw, v}^{\fr}$ and $\sfJ' := \{ k\in \sfJ_{\uw, v} \mid \text{$M_k$ satisfies $\eqref{cond:frozen}$ \eqref{cond:frozen (b)}}\}.$
By Theorem~\ref{th:ifro} and Proposition \ref{prop: FK ex fr}, we have  $ \sfJ \subset \sfJ'$. 

Let $k \in \sfJ'$ and assume that $s_{i_k} v_{\ge k}  \le w_{\ge k+1} $. Then Theorem~\ref{th:ifro} says that there is a simple module $X \in \Cw[w_{\ge k}, v_{\ge k}]$ such that $\de(\ang{i_k}, X) > 0$.
By Lemma \ref{lem:Kmain} \eqref{it:Mainde} and Proposition \ref{prop:stability} \eqref{prop:stability (ii)}, we have 
$$
\de(M_k, \G_1\cdots \G_{k-1}(X)) = \de(\ang{i_k}, X) > 0,
$$
where $ \G_t = \K_{i_t}$ if $c_t = \kt$ and $ \G_t = \F_{i_t}$ if $c_t = \ft$. 
Since $\G_1\cdots \G_{k-1}(X)$ belongs to $\Cwv$, this is a contradiction to  $k \in \sfJ'$.
Hence we have $ \sfJ' \subset \sfJ$, which completes the proof. 
\end{proof}

\begin{remark} \label{Rmk: KL R poly}
We define non-negative integers $d_{w,v}$ ($v \le w$) inductively as follows:
\begin{itemize}
\item $d_{w,w}=0$ for any $w \in \weyl$,
\item for any $v \le w$ and $i\in I$ such that $ s_iw < w$, 
\begin{align*}
d_{w,v} = 
\bc
d_{s_iw,s_iv} & \text{ if $s_iv < v$,} \\
d_{s_iw,v} + 1 & \text{ if $v < s_iv$ and $ s_iv \not \le s_i w$,} \\
d_{s_iw,v} & \text{ if $v < s_iv$ and $ s_iv  \le s_i w$.}
\ec
\end{align*}
\end{itemize}
Then we have $d_{w,v} = |\sfJ_{\uw, v}^{\fr}|$. 
This $d_{w,v}$ coincides with the formula for the coefficient of the second largest
power in \emph{Kazhdan-Lusztig R-polynomials} in \cite[(3.1)]{ELPSW26}. 
We refer the readers to \cite[Section 3.4]{ELPSW26} for details on Kazhdan-Lusztig R-polynomials and the number of frozen variables in open Richardson varieties.  
\end{remark}

\begin{thm} \label{thm: Swv for Cwv}
The set 
$\seed_{\uw, v}$ is a Laurent family in $\Cwv$. Moreover, 
$\seed_{\uw, v}$ with $\sfJ_{\uw, v} = \sfJ_{\uw, v}^\ex \sqcup \sfJ_{\uw, v}^\fr$ is  a completely admissible seed  of  $\Cwv$.
\end{thm}
\begin{proof}
For $k\in [1,r+1]$, we define $\seedk \subset \Cwv[w_{\ge k},v_{\ge k}]$ inductively by
\begin{equation} \label{Eq: S(k)}
\begin{aligned}
\seedk[r+1]& \seteq \emptyset,\\
\seedk& \seteq \bc
\F_{i_k}(\seedk[k+1])&\text{if $ c_k = \ft$,}\\
\st{\ang{i_k}}\sqcup\K_{i_k}(\seedk[k+1])&\text{if $c_k = \kt$,}\ec
\hs{3ex}\text{for $k\in[1,r]$.}
\end{aligned}
\end{equation}
By the construction, we have $\seed_{\uw, v} = \seedk[1]$.
Proposition \ref{prop: laurent} and \ref{prop: FK ex fr} say that 
$\seedk$ is a Laurent family of $\Cwv[w_{\ge k},v_{\ge k}]$ and satisfies \eqref{cond:frozen} for each $k\in [1,r+1]$. 
As $\F_i$ preserves the completely admissibility, the assertion follows by applying Theorem \ref{th:main} to the above construction \eqref{Eq: S(k)}.
\end{proof}

\begin{example} \label{Ex: Cwv}
We continue with Example \ref{Ex: Cw}. Recall the reduced expression $\uw = (1,2,1,3,2,1)$, and take $v = s_2s_3$.
Then $w_{\ge k}$ and $v_{\ge k}$ are computed as 
\begin{table}[H]
\begin{center}
\begin{tabular}{|c||c|c|c|c|c|c|}
\hline
$k$ & $1$ & $2$ & $3$ & $4$ & $5$ & $6$ \\
\hline
\hline
$w_{\ge k}$ & $ s_1s_2s_1s_3s_2s_1 $  & $s_2s_1s_3s_2s_1 $ & $s_1s_3s_2s_1 $ & $s_3s_2s_1 $ & $s_2s_1 $ & $s_1 $  \\
\hline
$v_{\ge k}$ & $s_2s_3$  & $s_2s_3$  & $s_3$ & $s_3$ & $\id$ & $\id$ \\
\hline
\end{tabular}
\end{center}
\end{table}
Hence, we have $[1,6]_v =\st{ 1,3,5,6 }$, the associated KF sequence
$$ 
 \kfc= (c_1, \ldots, c_6) = (\kt, \ft,   \kt, \ft, \kt, \kt),
$$
and the decomposition $\sfJ_{\uw, v} = [1,6]_v = \sfJ_{\uw, v}^{\ex} \sqcup \sfJ_{\uw, v}^{\fr} $ given by 
$$
\sfJ_{\uw, v}^{\ex} = \{ 3 \} \qtq
\sfJ_{\uw, v}^{\fr} = \{ 1,5,6 \}.
$$
Since the KF sequence $\kfc$ coincides with the KF sequence in Example \ref{Ex: Cw}, we have 
$$
\seed_{\uw, v} = \{ M_1, M_3, M_5, M_6  \}, 
$$
where $M_k$ are given in Example \ref{Ex: Cw}. The seed $\seed_{\uw, v}$ can be depicted as follows: 
$$
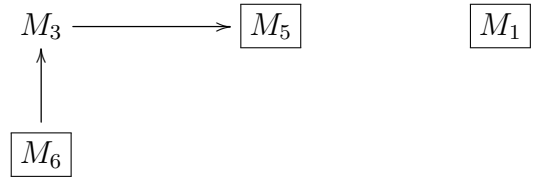
\xymatrix{
M_3  \ar[rr]  && \text{\fbox{$M_5$}}   && \text{\fbox{$M_1$}}  \\ 
\text{\fbox{$M_6$}} \ar[u] &&    &&  
}
$$
where the boxed ones are frozen. Note that the above quiver can be obtained from the quiver in Example \ref{Ex: Cw} by deleting $M_2$, $M_4$ and freezing $M_1$, $M_5$. The mutation of $M_3$ is $\ang{3}$. 
\end{example}

The next proposition gives the relation 
between our seed of $\Cwv$ and the one of Leclerc.
Note that the the isomorphism class of $M_k$ ($k \in [1,r]^{\uw}_v$) 
in the proposition gives a prime element of 
the Grothendieck ring $K(\Cw)\vert_{q=1}$
(\cite{GLS13F}).

\begin{prop} \label{prop:SeedLeclerc}
Let $v \le w \in \weyl$ and let $\uw$ be a reduced expression of $w$.
Let
$\seed_{\uw, v} =  ( M_k)_{k\in [1,r]_v^{\uw}}$
be the seed of $\Cwv$  associated with $\uw$ and $v$.
\bnum
\item
Then there exists a family $\st{n_{j,k}}_{j,k\in  [1,r]_v^{\uw}}$ 
of non-negative integers such that
$$M(w_{\le k } \La_{i_k}, v_{\le k} \La_{i_k})\simeq\conv[{j\in [1,r]_v^{\uw}}]
\hs{1ex}M_j^{\circ n_{j,k}}
\qt{for any $k\in  [1,r]_v^{\uw}$}$$
and $n_{j,k}\in\Z_{\ge0}$ satisfy the unitriangular condition: 
$n_{j,k}=0$ for any $j>k$ and $n_{k,k}=1$.
\item In particular, we have
\begin{align*}
\seed_{\uw, v}& =  \set{ M_k}{k\in [1,r]_v^{\uw}}  \\
&= 
\set{\text{prime convolution factors of $\dM(w_{\le k} \La_{i_k}, v_{\le k} \La_{i_k} )$}}{ k\in [1,r]^{\uw}_v}.    
\end{align*}
Note that  the right hand side is the seed of Leclerc in \cite[Corollary 4.4]{Lec16} when $\g$ is of finite type $ADE$.
\ee
\end{prop}

\begin{proof}
Note that (ii) immediately follows from (i)
and \cite[Theorem 4]{KK19}. We prove (i)  by induction on $\ell(w)$.
Assume that $\uw=(i_1,i_2,\ldots, i_r)$. Set $i=i_1$ and
 $\uw'=(i_1',i_2',\ldots, i'_{r-1})=(i_2,i_3,\ldots, i_r)$.  Then 
$\uw'$ is a reduced expression of $w'\seteq s_iw < w$.
Set $v'=v^{\uw}_{\ge 2}$.
Let $\sfc=(c_k)_{k\in[1,r]}$  be the $KF$-sequence for $\uw$ and $v$, and
let $\sfc'=(c_k')_{k\in[1, r-1]}$  be the $KF$-sequence for $\uw'$ and $v'$.
Then $c_{k+1}=c'_k$ for $k\in[1,r-1]$.
Let $$\seed_{\uw', v'} :=  ( M'_k)_{k\in [1,r-1]_{v'}^{\uw'}}
\qtq 
\seed_{\uw, v} :=  (M_k)_{k\in [1,r]_{v}^{\uw}}.$$

\snoi
{\bf Case 1 : } Assume that $s_iv < v$.
Then $s_iv < s_iw $, $c_1=\ft$, 
$w_{\le k+1}=s_iw'_{\le k}$ and $v_{\le k+1}=s_iv'_{\le k}$. and
$$[1,r]_{v}^{\uw}=\st{k+1\mid k\in [1,r-1]_{v'}^{\uw'}}\qtq
M_{k+1}=\F_i(M'_k).$$

Since
$$\dM(w_{\le k+1}\La_{i_{k+1}},v_{\le k+1}\La_{i_{k+1}})
\simeq \F_i\bl\dM(w'_{\le k}\La_{i'_{k}},v'_{\le k}\La_{i'_k})\br,$$
we have the desired result.

\mnoi
{\bf Case 2 : } Assume that $s_iv > v$.
Then $v'=v$, $\sfc_1=\kt$, and
$$[1,r]^{\uw}_v = \{1\} \sqcup \set{k+1}{k\in  [1,r-1]^{\uw'}_{v}}. $$
Hence we have 
$$M_1=\ang{i}\qtq
M_{k+1}\simeq\K_i(M'_k)\qt{for $k\in[1,r-1]^{\uw'}_v$}$$

Set $$L_k=\dM(w_{\le k } \La_{i_{k}}, v^{\uw}_{\le k} \La_{i_{k}})
\qtq L'_k=\dM(w'_{\le k } \La_{i'_{k}}, v'^{\uw'}_{\le k} \La_{i'_{k}}).$$
Then we have
$L_1=M_1$,   and $L_{k+1} \simeq \ang{i}^{n}\conv \K_i(L'_k)$ for some $n\in \Z_{\ge0}$ by Lemma \ref{lem:iaKi}.  
Hence if $L'_k=\conv[{j\in[1,r-1]^{\uw'}_{v}}]
\hs{.5ex}(M'_j){}^{\circ n_{j,k}}$,
then we have 
\eqn
L_{k+1}\simeq
\ang{i}^{\circ n}\conv\K_i(L'_{ k})
\simeq \ang{i}^{\circ n}\conv\conv[{j\in[1,r-1]^{\uw'}_{v'}}]
\hs{.5ex}  (\K_i M'_{j}){}^{\circ n_{j,k}}
\simeq M_1^{\circ n}\conv[{j\in[1,r-1]^{\uw'}_{v'}}]
\hs{.5ex}(M_{j+1}){}^{\circ n_{j,k}}. \qedhere
\eneqn
\end{proof}

\vskip 2em 

\section{Mutation invariance of the seeds $\seed_{\uw, v}$} \label{Sec: mutation invariance}

In this section, we prove that the mutation class of $\seed_{\uw, v}$ does not depend on the choice of reduced expressions $\uw$ of  $w\in\weyl$. 

Let $w,v\in\weyl$ such that $v\le w$, and choose and fix a reduced sequence
$$
\uw=(i_1,i_2,\ldots,i_r)
$$ 
throughout this section. 
We consider the seed $\wseed_{\uw,\kfc} $ for $\Cw$ and the seed $\seed_{\uw, v}$ for $\Cwv$, where  
$\kfc=(c_1,\ldots, c_r)\in \st{\kt,\ft}^r$
is defined in \eqref{Eq: ct for T}. 
We keep all notations appearing in Section \ref{Sec:wvseed}, and  simply write 
$$
\seed_{\uw} \seteq  \seed_{\uw, v}
$$
in this section. We remark that the mutation class of $\wseed_{\uw,\kfc}$
does not depend on the choice of $ \kfc\in\st{\kt,\ft}^r$ by Proposition \ref{prop: Twc}.
 Recall that for a set $A$, 
the symmetric group $ \sg_{r} = \langle \sgr_k \mid k\in [1,r-1]  \rangle$ 
acts on $A^{r}$ by the place permutation, i.e., 
for  $\bfa = (a_1, a_2, \ldots,a_r) \in  A^r$,
$$
\sgr_k  \bfa \seteq  (a_{ 1}, \ldots,  a_{k-1}, a_{k+1}, a_{k}, a_{k+2}, \ldots,a_{ r}),
$$
where $\sgr_k = (k, k+1)$ is the $k$-th transposition in $\sg_{r}$.

\subsection{Commutation equivalence}
In this subsection, we show that $\seed_{\uw, v}$ and $\seed_{\uw', v}$ 
belong to the same mutation class when $\uw'$ is obtained from $\uw$ by a single commutation move.

\Lemma \label{Lem: comm for K F}
Let $i,j\in I$ such that $(\al_i, \al_{j})=0$. For a simple module $M$ such that $\al_i, \al_j \not\in \gW(M)$, we have 
\bnum
\item \label{it: KiKj} $\K_i \K_j (M) \simeq \K_j \K_i (M)$,
\item \label{it: KiFj} $\K_i \F_j (M) \simeq \F_j \K_i (M)$,
\item  \label{it: FiFj} $\F_i \F_j (M) \simeq \F_j \F_i (M)$.
\ee
\enlemma
\begin{proof}
It is easy to see that
$\K_i(\ang{j}\hconv M)\simeq\ang{j}\hconv\K_i(M)$.

\snoi
\eqref{it: KiKj} Since $\ang{i} \conv \ang{j} \simeq \ang{j} \conv \ang{i} $ is simple and  $\de_i(M) = \de_i ( \ang{j}^k\hconv M)$ for any $k \ge 0$, we have $\K_i \K_j (M) \simeq \K_j \K_i (M)$ by Definition \ref{Def: Ki and K*i}.

\snoi
\eqref{it: KiFj} By Proposition \ref{prop:Refwv}, we have $ \F_j (\ang{i}^k\hconv M) \simeq \ang{i}^k\hconv \F_j(M)$ and 
$$
\de_i(M) = \de(\ang{i}, M) =  \de(\F_j (\ang{i}), \F_j(M)) = \de(\ang{i}, \F_j(M)) =  \de_i(\F_j(M)),
$$  
which implies~\eqref{it: KiFj}.

\snoi
\eqref{it: FiFj} follows from the fact that the functors $\F_i$ satisfy the braid relations (\cite{Kato20, KKOP24D}).
\end{proof}

Let $a\in[1,r-1]$ and assume that 
$$
(\al_{i_a},\al_{i_{a+1}})=0.
$$ 
We denote by $\kfc=(c_1,\ldots, c_r)\in \st{\kt,\ft}^r$ the 
KF sequence associated with $\uw$ and $v$, and by $\kfc'$ the 
KF sequence associated with $ \sgr_a\uw$ and $v$ (see \eqref{Eq: ct for T}).
We set $\Kt \seteq  [1,r]_v^{\uw}$ and $\Kt' \seteq  [1,r]_v^{\sgr_a\uw}$, and write 
\begin{align*}
\seed_{\uw} \seteq  ( M_k)_{k\in \Kt} 
\qtq \seed_{\sgr_a \uw} \seteq ( M_k' ) _{k\in \Kt'}.
\end{align*}

\Prop \label{Prop: comm}
We have 
$$
\kfc' = \sgr_a \kfc \qtq \seed_{\sgr_a \uw} = \seed_{\uw} \quad \text{ as a set.}
$$ 
Precisely, we have  
$ M_k' = M_k $ for $k \in \Kt \setminus \{a, a+1 \} $ and 
\bnum
\item if $ (c_a, c_{a+1}) = (\kt, \kt)$, then $ \Kt' = \Kt$ and $M_{b}' = M_{c} $ where $\{ b,c \} = \{a, a+1 \}$,
\item if $ (c_a, c_{a+1}) = (\ft, \kt)$, then $ \Kt' = (\Kt\setminus\{ a+1\}) \cup \{ a\}$ and $M_{a}' = M_{a+1} $,
\item if $ (c_a, c_{a+1}) = (\ft, \ft)$, then $ \Kt' = \Kt$.
\ee
\enprop
\begin{proof}
It follows by applying Lemma \ref{Lem: comm for K F} to the definition of $\seed_{\uw}$.
\end{proof}

\subsection{Braid equivalence}

In this subsection, we show that $\seed_{\uw, v}$ and $\seed_{\uw', v}$ 
belong to the same mutation class when $\uw'$ is obtained from $\uw$ by a single braid move.
Recall that we assume that the Cartan matrix $\sfC$ is symmetric. 

\Lemma \label{Lem: KiKi}
Let $i,j\in I$ such that $(\al_i,\al_j)=-1$.
Let $M$ be a simple module such that $\al_i,\al_j\not\in\gW(M)$, and 
set 
$$
m \seteq \de_i(M) \qtq n\seteq \de_j(M).
$$
Then we have 
\bnum
\item \label{Lem: KiKi (i)}
$\de_j(\ang{i^c}\hconv M)=c+\de_j(M)$ for any $c\in\Z_{\ge0}$.
In particular, we have $\de_j(\K_iM)=\de_i(M)+\de_j(M)$,
\item \label{Lem: KiKi (ii)}
$\K_j\K_iM\simeq \hd\bl\ang{ji}^{\circ m}\conv\ang{j^n}\conv M\br
\simeq\hd\bl\ang{j^{m+n}}\conv\ang{i^m}\conv M\br
\simeq\ang{ji}^{\circ m}\hconv\K_jM$,
\item \label{Lem: KiKi (iii)}
$\de_i(\K_j\K_iM)=\de_j(M)$,
\item \label{Lem: KiKi (iv)}
$\E_i(\K_jM)\simeq0$ and
$\de_i(\F_j\K_iM)=\de(\ang{ij},\K_iM)=\de_j(M)$,
\item \label{Lem: KiKi (v)}
$\E_i(\F_j\K_iM)\simeq0$ and $\de_j(\F_i\F_jM)=\de_i(M)$.

\ee
\enlemma
\Proof
\eqref{Lem: KiKi (i)}\ Since $(\ang{j}, M)$ and $(\ang{i^c},\ang{j})$ are unmixed pairs for $c \in \Z_{\ge0}$, we have
\eqn
\La(\ang{j},\ang{i^c}\hconv M)&&=\La(\ang{j},\ang{i^c})+
\La(\ang{j}, M)\qtq\\
\La(\ang{i^c}\hconv M,\ang{j})&&=\La(\ang{i^c},\ang{j})+\La(M,\ang{j})
\eneqn
by \cite[Corollary 2.13]{KKOP23A} and \cite[Lemma 2.7 and 2.8]{KK19}.
Hence $\de_j(\ang{i^c}\hconv M)=\de_j(\ang{i^c})+\de_j(M)$.

\snoi
\eqref{Lem: KiKi (ii)}\ Since $(\ang{j}, M)$ is an unmixed pair,
$\ang{j^{m+n}}\conv\ang{i^m}\conv M$ has a simple head.
By (i), we have
\begin{equation} \label{Eq: KjKiM}
\begin{aligned}
\K_j\K_iM &\simeq 
\ang{j^{m+n}}\hconv(\ang{i^m}\hconv M)
\simeq \hd\bl\ang{j^{m+n}}\conv\ang{i^m}\conv M\br \\
&\simeq\hd\bl\ang{ji}^{\circ m}\conv\ang{j^n}\conv M\br \\ 
&\simeq 
\ang{ji}^{\circ m}\hconv\K_jM.
\end{aligned}
\end{equation}

\snoi
\eqref{Lem: KiKi (iii)}\ We have 
\begin{align*}
\La(\ang{i},\K_j\K_i M)
&=\La\bl\ang{i},\ang{ji}^{\circ m}\hconv(\ang{j^n}\hconv M)\br\\
&=\La\bl\ang{i},\ang{ji}^{\circ m})+
\La\bl\ang{i},\ang{j^n}\hconv M\br\\
&= \La\bl\ang{i},\ang{ji}^{\circ m})+
\La\bl\ang{i},\ang{j^n})+\La(\ang{i},M)\\
&=-m+n+\La(\ang{i},M),
\end{align*}
where 
the first equality follows from \eqref{Eq: KjKiM},
the second and third follow from the fact that $(\ang{i}, \ang{j^n}\hconv M )$ and $(\ang{i}, M )$ are unmixed,
and the fourth follows from $\La( \ang{i}, \ang{ji})=-1$ and $\La( \ang{i}, \ang{j})=1$.
Similarly, we have 
\begin{align*}
\La(\K_j\K_i M,\ang{i}) &= \La(\ang{j^{m+n}}\hconv\K_i M,\ang{i})
=\La(\ang{j^{m+n}},\ang{i})+\La(\K_i M,\ang{i})\\
&=m+n-\La(\ang{i},\K_iM)\\
&=m+n-\La(\ang{i},M),
\end{align*}
where 
the first equality follows from \eqref{Eq: KjKiM},
the second follows from the fact that $(\ang{j}, \ang{i})$ is unmixed,
the third follows from $\La(\ang{j}, \ang{i})=1$ and $\de(\ang{i}, \K_i M)=0$,
and the fourth follows from $\La(\ang{i}, \K_i M) = \La(\ang{i}, \ang{i^m} \hconv M) = \La(\ang{i}, M)$.
Hence, we have
$$2\de_i(\K_j\K_iM)=
\bl-m+n+\La(\ang{i},M)\br+\bl m+n-\La(\ang{i},M)\br=2n.$$

\snoi
\eqref{Lem: KiKi (iv)}\ 
Since $ \al_i \not \in \gW (\ang{j}\conv M ) $, we have  $\E_i(\K_jM)\simeq0$.
 
By the same manner as above, we have
\eqn
\La(\ang{ij},\K_iM)&&=
\La\bl\ang{ij},\ang{i}^m\hconv M\br\\
&&=
\La\bl\ang{ij},\ang{i}^m\br+\La\bl\ang{ij},M\br\\
&&=
-m+\La(\ang{i},M)+\La(\ang{j},M)
\eneqn
and
\eqn
\La(\K_iM,\ang{ij})&&=
\La(\K_iM,\ang{i})+\La(\K_iM,\ang{j})\\
&&=-\La(\ang{i}, \K_iM)+\La(\ang{i^m},\ang{j})+\La(M,\ang{j})\\
&&=-\La(\ang{i}, M)+m+\La(M,\ang{j}).
\eneqn
Hence, by Proposition \ref{prop:Refwv}, we have 
\begin{align*}
2\de(\ang{i},\F_j\K_iM) &=2\de\bl\F_j^{-1}\ang{i},\K_iM\br
=2\de(\ang{ij},\K_iM) \\ 
&=\bl-m+\La(\ang{i},M)+\La(\ang{j},M)\br+
\bl-\La(\ang{i}, M)+m+\La(M,\ang{j})\br\\
&=\La(\ang{j},M)+\La(M,\ang{j}) \\
&= 2\de_j(M).
\end{align*}

\snoi
\eqref{Lem: KiKi (v)}\ By \cite[Corollary 3.8]{KKOP18}, we have 
\begin{align*}
\ep_i(\F_j\K_i M) &= \tLa(\ang{i},\F_j\K_i M) =\tLa(\F_j^{-1}\ang{i},\K_iM) =\tLa(\ang{ij}, \ang{i^m}\hconv M)\\
&
=\tLa(\ang{ij}, \ang{i^m})+\tLa(\ang{ij}, M)\\
&=0,
\end{align*}
where the last equality follows from the fact that $(\ang{ij}, \ang{i} )$ and $(\ang{ij}, M)$ are unmixed.
This tells us that $\E_i (\F_j\K_iM) \simeq0$.

Since $\ang{j}=\F_i\F_j\ang{i}$, 
we have
\[ \de\bl\ang{j},\F_i\F_jM\br
=\de\bl \F_i\F_j\ang{i},\F_i\F_jM\br
=\de\bl \ang{i},M\br. \qedhere \]
\QED

\Cor\label{cor:KKK}
Let $i,j\in I$ such that $(\al_i,\al_j)=-1$.
Let $M$ be a simple module such that $\al_i,\al_j\not\in\gW(M)$.
Then, we have
\bnum
\item \label{it:KKK}
$\K_i\K_j\K_iM\simeq\K_j\K_i\K_jM$,
\item \label{it:FKK}
$\F_i\K_j\K_iM\simeq \K_j\F_i\K_jM$,
\item \label{it:FFK}
$\F_i\F_j\K_iM\simeq\K_j\F_i\F_jM$.
\item \label{it:FFF}
$\F_i\F_j\F_iM\simeq\F_j\F_i\F_jM$.
\ee
\encor
\Proof
Set $m\seteq \de_i(M)$ and $n \seteq \de_j(M)$.

\snoi
\eqref{it:KKK}\ Since $\ang{ji}$ commutes with $\ang{i}$ and $\ang{j}$,
 Lemma \ref{Lem: KiKi} \eqref{Lem: KiKi (ii)} says that 
$$\K_j\K_iM\simeq
(\ang{j^{m+n}}\hconv\ang{i^m})\hconv M
\simeq 
(\ang{j^n}\conv\ang{ji}^{\circ m})\hconv M,$$
which implies 
\eqn \K_i\K_j\K_i M&&\simeq
\ang{i^n}\hconv\K_j\K_iM\\
&&\simeq \ang{i^n}\hconv\bl(\ang{j^n}\conv\ang{ji}^{\circ m})\hconv M\br\\
&&\simeq \bl\ang{i^n}\hconv(\ang{j^n}\conv\ang{ji}^{\circ m})\br\hconv M\\
&&\simeq \bl\ang{ij}^{\circ n}\conv\ang{ji}^{\circ m})\br\hconv M,
\eneqn
where the first equality follows from  Lemma \ref{Lem: KiKi} \eqref{Lem: KiKi (iii)}. 
Exchanging the roles of $i$ and $j$ in the above, we also obtain
$$
\K_j\K_i\K_j M\simeq\bl\ang{ji}^{\circ m}\conv\ang{ij}^{\circ n})\br\hconv M.
$$
The assertion follows from the fact that $\ang{ij}$ commutes with $\ang{ji}$.

\snoi
\eqref{it:FKK}\ By Lemma \ref{Lem: KiKi} \eqref{Lem: KiKi (ii)},
we have
\eqn
\F_i\K_j\K_iM&&\simeq\F_i\bl\ang{ji}^{\circ m}\hconv\K_jM\br\\
&&\simeq\F_i(\ang{ji}^{\circ m})\hconv(\F_i\K_jM)\\
&&\simeq\ang{j^{m}}\hconv\F_i\K_jM\\
&&\simeq\K_j\F_i\K_jM,
\eneqn
where the last equality follows by exchanging the roles of $i$ and $j$ in Lemma \ref{Lem: KiKi}  \eqref{Lem: KiKi (iv)}. 

\snoi
\eqref{it:FFK}\ By Lemma \ref{Lem: KiKi} \eqref{Lem: KiKi (v)}, we have
\eqn
\F_i\F_j\K_iM
&&\simeq\F_i\F_j\bl\ang{i^m}\hconv M\br\\
&&\simeq\F_i\F_j\ang{i^m}\hconv \F_i\F_j M\\
&&\simeq\ang{j^m}\hconv \F_i\F_j M\\
&&\simeq\K_j\F_i\F_j M.
\eneqn

\snoi
\eqref{it:FFF} It follows from \cite{Kato20, KKOP24D}.
\QED

Let $a\in[1,r-2]$ and assume that 
\begin{align} \label{Eq: cond for i}
 i_a = i_{a+2} \qtq  (\al_{i_a},\al_{i_{a+1}})=-1.
\end{align}
We set $\uw'$ to be the reduced expression obtained from $\uw$ by applying a single braid move at the $a+1$-position, i.e.,   
$$
\uw' = (i_1, \ldots, i_{a-1}, i_{a+1}, i_{a}, i_{a+1},i_{a+3}, \ldots, i_r  ).
$$
We denote by $\kfc=(c_1,c_2,\ldots, c_r)\in \st{\kt,\ft}^r$ the 
KF sequence associated with $\uw$ and $v$, and by $\kfc' =(c_1',c_2',\ldots, c_r')$ the 
KF sequence associated with $ \uw'$ and $v$.
We set $\Kt \seteq  [1,r]_v^{\uw}$ and $\Kt' \seteq  [1,r]_v^{\uw'}$, and write 
\begin{align*}
\seed_{\uw} \seteq ( M_k )_{k\in \Kt} 
\qtq \seed_{\uw'} \seteq( M_k')_{k\in \Kt'}.
\end{align*}
Note that, by \eqref{Eq: cond for i} and the definition of $\kfc$ (see \eqref{Eq: ct for T}), the triple
$(c_{a}, c_{a+1}, c_{a+2})$ can be neither $(\kt, \kt, \ft)$ nor $(\ft, \kt, \ft)$.

\Prop\label{prop:Inv} 
For $ k \in [1,r] \setminus \{ a, a+1, a+2 \}$ and $t \in \Kt \setminus \{ a,a+1,a+2 \}$, we have 
$$
c_k' = c_k \qtq M_k' = M_k.  
$$
Moreover, we have the following.
\bnum
\item \label{it: ckck' (i)}
If $( c_a, c_{a+1}, c_{a+2}) = (\kt, \kt, \kt)$, then 
$( c_a', c_{a+1}', c_{a+2}') = (\kt, \kt, \kt)$, $ \Kt' = \Kt$, and 
$$
M_a' = \mu_a(M_a), \quad M_{a+1}' = M_{a+2},\quad   M_{a+2}' = M_{a+1},
$$
which means $\sgr_{a+1}(\seed_{\uw'}) = \mu_a(\seed_{\uw})$.

\item \label{it: ckck' (ii)}
If $( c_a, c_{a+1}, c_{a+2}) = (\ft, \kt, \kt)$, then $( c_a', c_{a+1}', c_{a+2}') = (\kt, \ft, \kt)$, $\Kt' = \Kt \cup\{ a \} \setminus\{ a+1\}$, and 
$$
M_a' =  M_{a+2} \qtq  M_{a+2}' = M_{a+1},
$$
which means $ \seed_{\uw'} =  \seed_{\uw}$ as a set.

\item \label{it: ckck' (iii)}
If $( c_a, c_{a+1}, c_{a+2}) = (\ft, \ft, \kt)$, then $( c_a', c_{a+1}', c_{a+2}') = (\kt, \ft, \ft)$, $\Kt' = \Kt \cup\{ a \} \setminus\{ a+2\}$, and 
$$
M_a' =  M_{a+2},
$$
which means $ \seed_{\uw'} =  \seed_{\uw}$ as a set.

\item \label{it: ckck' (iv)}
If $( c_a, c_{a+1}, c_{a+2}) = (\ft, \ft, \ft)$, then $( c_a', c_{a+1}', c_{a+2}') = (\ft, \ft, \ft)$, $\Kt' = \Kt $, and $ \seed_{\uw'} =  \seed_{\uw}$.
\ee
\enprop

\begin{proof}

Set $\wseed \seteq \wseed_{\uw, \kfc}=(M_k)_{k\in[1,r]}$ and $\wseed' \seteq \wseed_{\uw',\kfc'}=(M'_k)_{k\in[1,r]}$, and define 
$$
(N_k)_{k\in[1,r]} \seteq 
\bc \mu_a(\wseed)&\text{in case (i) and (ii),}\\
\wseed&\text{in case (iii).}
\ec
$$
As in Section \ref{subsec:exp}, 
we write $S_k \seteq  \ang{i_k}$ if $c_k=\kt$ and
$S_k \seteq \dM(w_{\ge k}\La_{i_k},\La_{i_k})$ if $c_k=\ft$ so that 
$$
M_k= \G_1\G_2\cdots \G_{k-1} (S_k),
$$
where $\G_k=\K_{i_k}$ if $c_k=\kt$ and
$\G_k=\F_{i_k }$ if $c_k=\ft$.

We may assume that $a=1$. Write 
$$
w'\seteq s_{i_3}\cdots s_{i_r}, \quad i \seteq i_1 \qtq j \seteq i_2,
$$
so that $w=s_is_js_iw'$. It is obvious that $c_k' = c_k$ for $ k >3$. Since $M = \G_4\cdots \G_{k-1}(S_k)$ satisfies $\E_{i}M=\E_{j}M=0$,
Corollary~\ref{cor:KKK} implies $N_k=M_k = M'_k$ for $k>3$.

\mnoi
\eqref{it: ckck' (i)}\ Assume $ (c_1,c_2,c_3) = (\kt,\kt,\kt)$. In this case, it is clear that $ (c_1',c_2', c_3' ) =  (\kt,\kt,\kt)$ and $\Kt = \Kt'$.
By a direct computation, we have
\begin{align*}
M_1=\ang{i}, \qquad  M_2=\K_i\ang{j}\simeq\ang{ij}, \qquad M_3=\K_i\K_j\ang{i}\simeq\ang{ji}, \\
M'_1=\ang{j}, \qquad M'_2=\K_j\ang{i}\simeq\ang{ji}, \qquad M'_3=\K_j\K_i\ang{j}\simeq\ang{ij}.
\end{align*}
From the exact sequence
$0\to\ang{ji}\to\ang{i}\conv\ang{j}\to\ang{ij}\to0$,
we have
$N_1\simeq\ang{j}$, $N_2\simeq\ang{ij}$ and $N_3\simeq\ang{ji}$ 
by Lemma \ref{Lem: criterion for cv} \eqref{Lem: criterion for cv (ii)}.
 
Thus we have 
$$
\mu_1(\wseed)= \sgr_{2} (\wseed'),
$$
which implies the assertion.

\snoi
\eqref{it: ckck' (ii)}\ Assume $ (c_1,c_2,c_3) = (\ft,\kt,\kt)$. In this case, 
we have $s_i v < v$ and $s_j s_i v > s_iv$, which says $ (c_1',c_2', c_3' ) =  (\kt,\ft,\kt)$ and $\Kt' = \Kt \cup\{ a \} \setminus\{ a+1\}$.
By a direct computation, we have
\begin{align*}
M_1&=\dM(w\La_i,\La_i),   &&M_2=\F_i\ang{j}\simeq\ang{ij}, \quad  M_3=\F_i\K_j\ang{i}
\simeq\ang{j}, \text{ and } \\ 
M'_1&=\ang{j},  &&M'_2=\K_j\dM(s_is_jw'\La_i,\La_i)\simeq\dM(s_js_is_jw'\La_i,\La_i)\simeq \dM(w\La_i,\La_i), \\ 
M'_3 &=\K_j\F_i\ang{j}\simeq\ang{ij}.
\end{align*}
Hence we have 
$
\wseed= \wseed'
$
as a set, which implies the assertion.

\snoi
\eqref{it: ckck' (iii)}\ 
Assume $ (c_1,c_2,c_3) = (\ft,\ft,\kt)$. In this case, 
we have $ s_js_iv < s_i v < v$ and $ s_j s_i v < s_i s_j s_i v $, which says $ (c_1',c_2', c_3' ) =  (\kt,\ft,\ft)$ and $\Kt' = \Kt \cup\{ a \} \setminus\{ a+2\}$.
By a direct computation, we have
\begin{align*}
M_1 & =\dM(w\La_i,\La_i), \qquad  M_2=\F_i\dM(s_js_iw'\La_j,\La_j)
\simeq \dM(w\La_j,\La_j), \\ 
M_3 &=\F_i\F_j\ang{i}
\simeq\ang{j}, \\
M'_1 & =\ang{j}, \qquad M'_2=\K_j\dM(s_is_jw'\La_i,\La_i)\simeq\dM(s_js_is_jw'\La_i,\La_i)\simeq \dM(w\La_i,\La_i), \\ 
M'_3 & =\K_j\F_i\dM(s_jw'\La_j,\La_j)
\simeq\dM(s_js_is_jw'\La_j,\La_j)\simeq\dM(w\La_j,\La_j)
\end{align*}
Hence we have 
$\wseed= \wseed'$ as a set,  which implies the assertion.

\snoi
\eqref{it: ckck' (iv)} Assume $ (c_1,c_2,c_3) = (\ft,\ft,\ft)$. In this case, it is obvious by the definition of $\seed_{\uw}$.
\end{proof}

We now have the main theorem of this section from Proposition \ref{Prop: comm} and \ref{prop:Inv} . 

\begin{thm} \label{thm: one mutation class for Cwv}
Let $w, v \in \weyl$ with $v \le w$. Then the mutation class of the seed $\seed_{\uw, v} $ of $ \Cwv$ does not depend on the choice of  reduced expressions $\uw$ of $w$. 
\end{thm}

\vskip 2em 

\section{Cluster algebras structure for $\Cwv$} \label{Sec: cluster alg str of Cwv}

In this section, we discuss the cluster algebra structure of $\Cwv$.

Let $w , v \in \weyl$ with $v \le w$. 
Let $\uw = (i_1,i_2, \ldots, i_r)$ be a reduced expression of $w$, and set $\sfJ \seteq  [1,r]$.
We consider the completely admissible seeds
$$
\text{$\wseed_{\uw,\kfc} = \{ M_k \}_{k\in \sfJ}$ of $\Cw$ and
$\seed_{\uw, v} = \{ M_k \}_{k\in \sfJ_{\uw, v}}$ of $ \Cwv$}
$$ 
given by \eqref{Eq: TM}, \eqref{Eq: ct for T} and \eqref{Eq: def of Swv},
 respectively (see Propositions \ref{prop: Twc} and Theorem~\ref{thm: Swv for Cwv}).
Thanks to Proposition \ref{prop: Twc} and Theorem \ref{thm: one mutation class for Cwv}, the mutation classes of $\wseed_{\uw,\kfc}$ and  $\seed_{\uw, v}$ do not depend  on  the choice of $\uw$.

Let $\sfJ^\ex$ and $\sfJ^\fr$ be the index sets for $\wseed_{\uw,\kfc}$ given in \eqref{Eq: index for Cw}, and let $\sfJ_{\uw, v}^{\ex} $ and $\sfJ_{\uw, v}^{\fr} $ be the index sets for $\seed_{\uw, v}$ given in \eqref{Eq: index for Cwv}.
Let $ \tB_{\uw, \kfc} = (b_{i,j})_{i\in \sfJ, j \in \sfJ^\ex}$ be the exchange matrix determined by $\wseed_{\uw,\kfc}$ with $ \sfJ = \sfJ^\ex \sqcup \sfJ^\fr$,and set 
$$
\tB \seteq  \tB_{\uw, \kfc} |_{ \sfJ_{\uw, v} \times \sfJ_{\uw, v}^\ex}.  
$$
Note that $\tB$ is the exchange matrix for $\seed_{\uw, v}$.
We set 
$\La \seteq  \bl\La(M_i, M_j )\br)_{i,j\in \sfJ_{\uw, v}}$ and  let
$$
\Sigma_{\uw, v} \seteq  ( \{  q^{- (\wt(M_k), \wt(M_k) )/4} [M_k] \}_{k\in \sfJ_{\uw, v}}, -\La, \tB) 
$$ 
be a quantum seed  in  $K(\Cwv)$. We  denote by 
$$
\ca_{\uw, v} \seteq  \ca(\Sigma_{\uw, v}) \qtq  \uca_{\uw, v} \seteq  \uca(\Sigma_{\uw, v})
$$
the  quantum  cluster algebra 
and the quantum upper cluster algebra 
with the initial seed $\Sigma_{\uw, v}$, respectively. 

The following lemma follows from Theorem~\ref{thm: Swv for Cwv}
and Proposition~\ref{prop:Laurentm}.
\begin{lem} \label{lem: quasi-Laurent}
Every seed which is mutation equivalent to $\seed_{\uw, v}$ is a Laurent family in $\Cwv$.
\end{lem}

We now obtain the main theorem.
\begin{thm} \label{Thm: A K U}
Let  $w,v \in \weyl$ with $v \le w$ and let $\uw$ be a reduced expression of $w$. 
 Then we have 
$$
\ca_{\uw, v} \subset K(\Cwv) \subset \uca_{\uw, v}.
$$

\end{thm}
\begin{proof}
Since $\{ M_k \}_{k\in \sfJ_{\uw, v}} $ is contained in $   \Cwv $, 
we have $\ca_{\uw, v} \subset K(\Cwv)$.

For any seed $\seed$ mutation-equivalent to $\seed_{\uw, v}$, 
we have an injective homomorphism  
$$
\vphi_\seed\cl K(\Cwv)\rightarrowtail \T_{\seed}
$$
by Lemma \ref{lem: quasi-Laurent} and \eqref{Eq: qlp}. This implies $ K(\Cwv) \subset \uca_{\uw, v}$. 
\end{proof}

The following corollary was proved in \cite[Theorem 5.16]{Bi26}
by a different method. 
\begin{coro} \label{Cor: MC for tRwv}
Let $\g$ be of finite $ADE$ type. For $v\le w$,  the category 
$\tCwv$ is a monoidal categorification of the coordinate ring $\C[\Rwv]$ of open Richardson variety associated with $w$ and $v$.
\end{coro}
\begin{proof}
By taking localizations by frozen variables, we have
$$
\widetilde \ca_{\uw, v} \subset K(\tCwv) \subset \widetilde  \uca_{\uw, v},
$$
where $\widetilde\ca_{\uw, v}$ and $\widetilde\uca_{\uw, v}$ are cluster algebra and upper cluster algebra with invertible frozen variables, respectively.

By Proposition \ref{prop:SeedLeclerc}, $\seed_{\uw,v}$ is the seed of Leclerc in \cite[Corollary 4.4]{Lec16}. 
By \cite[Theorem 5.14]{Bi26}, it is the same with the seed of M\'enard in \cite{Menard22}. 
By \cite[Section 10]{CGGLSS25}, M\'enard's seed is a seed (associated with a right inductive weave) of a braid variety.
Since the cluster algebras of braid varieties are equal to their upper cluster algebras, we obtain  
$$\widetilde \ca_{\uw, v} = K(\tCwv)= \widetilde  \uca_{\uw, v}.$$
Now the assertion follows from Theorem \ref{thm: Swv for Cwv}.
\end{proof}

\Conj \label{Conj: MC}
For any $w,v \in \weyl$ with $v \le w$, we have $\ca_{\uw, v} = \uca_{\uw, v}$. This means that 
 $\Cwv$ provides a monoidal categorification of $K(\Cwv)$.
\enconj

\end{document}